\documentclass[11pt,a4paper]{article}

\usepackage[margin=1.3in]{geometry}
\usepackage[utf8]{inputenc}
\usepackage[T1]{fontenc}

\usepackage{enumerate}

\usepackage{booktabs}

\usepackage{tablefootnote}

\usepackage{float}

\usepackage[colorlinks, pagebackref=true]{hyperref}
\hypersetup{
  linkcolor=[rgb]{0.3,0.3,0.6},
  citecolor=[rgb]{0.2, 0.6, 0.2},
  urlcolor=[rgb]{0.6, 0.2, 0.2}
}

\usepackage{mathtools, amsthm, amsfonts, amssymb}
\usepackage{tikz}
\usetikzlibrary{decorations.markings}
\usetikzlibrary {arrows.meta}

\usepackage{authblk}

\usepackage{listings}
\usepackage{accents} %
\usepackage{microtype}
\usepackage{xcolor}

\usepackage{thmtools, thm-restate} 
\declaretheorem[name=Theorem, parent=section]{theorem}
\declaretheorem[name=Corollary, sibling=theorem]{corollary}

\declaretheorem[name=Lemma, sibling=theorem]{lemma}

\theoremstyle{definition}
\declaretheorem[name=Definition, sibling=theorem]{definition}
\declaretheorem[name=Remark, sibling=theorem]{remark}
\declaretheorem[name=Problem, sibling=theorem]{problem}

\declaretheorem[name=Example, sibling=theorem]{example}
\theoremstyle{remark}

\newcommand{\RR}{\mathbb{R}}
\newcommand{\QQ}{\mathbb{Q}}
\newcommand{\NN}{\mathbb{N}}
\newcommand{\ZZ}{\mathbb{Z}}

\newcommand{\X}{\mathcal{X}}
\newcommand{\asympleq}{\lesssim}

\let\eps\varepsilon

\newcommand{\closed}{\mathrm{c}}
\newcommand{\open}{\mathrm{o}}

\DeclareMathOperator{\subrank}{Q}
\DeclareMathOperator{\rank}{R}
\DeclareMathAccent{\wtilde}{\mathord}{largesymbols}{"65}
\DeclareMathOperator{\asympsubrank}{\underaccent{\wtilde}{Q}}
\DeclareMathOperator{\asymprank}{\underaccent{\wtilde}{R}}

\newcommand{\sr}{\mathcal{R}}

\newcommand{\G}{\mathcal{G}}

\title{The asymptotic spectrum distance, graph limits,\\ and the Shannon capacity}

\author{David de Boer}
\author{Pjotr Buys}
\author{Jeroen Zuiddam}

\affil{University of Amsterdam}

\date{}

\begin{document}
\maketitle
\small
\centerline{\textbf{Abstract}}
\vspace{0.4em}
Determining the Shannon capacity of graphs (Shannon 1956) is a long-standing open problem in information theory, graph theory and combinatorial optimization. Over decades, a wide range of upper and lower bound methods (of combinatorial, algebraic, analytic, and computational nature) have been developed to analyze this problem (e.g., Lovász 1979, Alon~1998, Bohman--Holzman 2003, Bukh--Cox 2019, Polak--Schrijver~2019, Guruswami--Riazanov 2021). However, despite tremendous effort, even small instances of the problem have remained open (e.g., the Shannon capacity of any odd cycle of length at least seven), and central computational and structural questions unanswered.

In recent years, a new dual characterization of the Shannon capacity of graphs, asymptotic spectrum duality, has unified and extended known upper bound methods and structural theorems. This duality originated from the work of Strassen in algebraic complexity theory (Strassen, J.~Reine Angew.\ Math.\ 1987, 1988, 1991) and has recently seen much activity (Christandl--Vrana--Zuiddam, J.~Amer.\ Math.\ Soc.~2023, Wigderson--Zuiddam, Bull.~Amer.\ Math.\ Soc.~2026). %

In this paper, building on asymptotic spectrum duality, we develop a new theory of graph distance, that we call \emph{asymptotic spectrum distance}, and corresponding limits (reminiscent of, but different from, the celebrated theory of cut-norm, graphons and flag algebras). We propose a graph limit approach to the Shannon capacity problem: to determine the Shannon capacity of a graph, construct a sequence of easier to analyse graphs converging to it.

(1) We give a very general construction of non-trivial converging sequences of graphs (in the asymptotic spectrum distance). As a consequence, we prove new continuity properties of the Shannon capacity, the Lovász theta function and other graph parameters, answering questions of Schrijver and Polak.

(2) We construct Cauchy sequences of finite graphs that do not converge to any finite graph, but do converge to an infinite graph (in the asymptotic spectrum distance). We establish strong connections between convergence questions of finite graphs and the asymptotic properties of Borsuk-like infinite graphs on the circle, and use ideas from the theory of dynamical systems to study these infinite graphs.

(3) We observe that all best-known lower bound constructions for Shannon capacity of small odd cycles can be obtained from a “finite” version of the graph limit approach, in which we approximate the target graph by another graph with a large independent set that is an “orbit” under a natural group action. We develop computational and theoretical aspects of this approach and use these to obtain a new Shannon capacity lower bound for the fifteen-cycle.

The graph limit point-of-view brings many of the constructions
that have appeared over time in a unified picture, makes new connections to ideas in topology and dynamical systems, and offers new paths forward. The theory of asymptotic spectrum distance applies not only to Shannon capacity of graphs, but indeed we will develop it for a general class of mathematical objects and their asymptotic properties.

\newpage\normalsize
\tableofcontents

\newpage

\section{Introduction}
\label{sec:intro}

Posed by Shannon in 1956 \cite{MR0089131}, determining the amount of information that can be transmitted perfectly over a noisy communication channel, the Shannon zero-error capacity, is a longstanding and central open problem in information theory, graph theory and combinatorial optimization \cite{MR1658803, MR1919567, MR1956924}. 
Mathematically, this problem asks for determining the rate of growth of the independence number of powers of graphs. 

Over decades, a wide range of upper and lower bound methods (of combinatorial, algebraic, analytic, and computational nature) have been developed to analyze this problem, by Shannon \cite{MR0089131}, Lovász \cite{Lovasz79}, Haemers \cite{haemers1979some, MR0642046}, Alon \cite{alon1998shannon}, Alon--Lubetzky \cite{MR2234473}, Polak--Schrijver~\cite{MR3906144, MR3959682}, Google DeepMind \cite{romera-paredes_mathematical_2024} and many others \cite{MR0207590, MR0337668, MR0321810, MR1210400, MR1967195, MR3016977, blasiakThesis, Mathew2017NewLB, MR3876424, Bukh18, MR4079641, fritz2021unified, LinearShannonCapacity, MR4525548, zhu2024improved}.
However, despite tremendous effort, even small instances of the problem have remained open (in particular, determining the Shannon capacity of any odd cycle of length at least seven), and central computational and structural questions unanswered. %

In recent years, a new dual characterization of the Shannon capacity of graphs, called asymptotic spectrum duality \cite{MR4039606}, has unified and extended known upper bound methods and structural theorems \cite{MR4357434,MR4245284,fritz2021unified,MR4525548}. This duality originates from Strassen's work in algebraic complexity theory on fast matrix multiplication algorithms and tensors \cite{strassen1987relative, strassen1988asymptotic, strassen1991degeneration, MR4495838} (and goes back to the real representation theorems of Stone and Kadison--Dubois \cite{MR0707730}) and applies much more generally to asymptotic problems in various fields \cite{MR4053385, DBLP:conf/focs/RobereZ21, MR4609385}; see the survey of Wigderson--Zuiddam~\cite{wigderson2022asymptotic}.  %

In this paper, building on asymptotic spectrum duality, 
we develop a new theory of graph distance and limits---reminiscent of, but different from, the celebrated theory of cut-norm, graphons and flag algebras in the study of homomorphism densities~\cite{MR2274085, MR2371204, MR2455626, MR3012035}. The crucial property of this \emph{asymptotic spectrum distance} is that converging graphs have converging Shannon capacity,
\[
G_i \to H \,\,\Rightarrow\,\, \Theta(G_i) \to \Theta(H),
\]
thus suggesting a graph limit approach to the Shannon capacity problem: to determine (or, say, lower bound) the Shannon capacity of a ``hard'' graph, construct a sequence of easier to analyse (e.g., highly structured) graphs converging to it.

Central results of this paper are (1) how to construct converging sequences of graphs and (2) where to look for graphs that are ``easier to analyse''.
Indeed, a special case of our results is that for any odd cycle graph we can construct many non-trivial converging sequences of very structured graphs converging to it (as illustrated in \autoref{fig:ex-c7}). Using a ``finite'' version of this idea we will obtain a new lower bound on the Shannon capacity of the fifteen-cycle.

\newcommand{\w}{3.2cm}
\newcommand{\s}{1.4cm}
\begin{figure}[H]
\begin{minipage}{\w}
\centering
\begin{tikzpicture}
\def \n {11}
\def \m {2}
  \foreach \i in {0,1,...,\n}
    \coordinate (v\i) at ({360/\n * (\i - 1)}:\s);

  \foreach \j in {1,...,\m} {
      \foreach \i [evaluate=\i as \next using {int(mod(\i+\j,\n))}] in {0,1,...,\n} {
        \draw (v\i) -- (v\next);
    ` }
  }

  \foreach \i in {0,1,...,\n}
    \filldraw[black] (v\i) circle (2pt);
\end{tikzpicture}
\end{minipage}
\begin{minipage}{\w}
\centering
\begin{tikzpicture}
\def \n {18}
\def \m {4}
  \foreach \i in {0,1,...,\n}
    \coordinate (v\i) at ({360/\n * (\i - 1)}:\s);

  \foreach \j in {1,...,\m} {
      \foreach \i [evaluate=\i as \next using {int(mod(\i+\j,\n))}] in {0,1,...,\n} {
        \draw (v\i) -- (v\next);
    ` }
  }

  \foreach \i in {0,1,...,\n}
    \filldraw[black] (v\i) circle (2pt);
\end{tikzpicture}
\end{minipage}
\begin{minipage}{\w}
\centering
\begin{tikzpicture}
\def \n {25}
\def \m {6}
  \foreach \i in {0,1,...,\n}
    \coordinate (v\i) at ({360/\n * (\i - 1)}:\s);

  \foreach \j in {1,...,\m} {
      \foreach \i [evaluate=\i as \next using {int(mod(\i+\j,\n))}] in {0,1,...,\n} {
        \draw (v\i) -- (v\next);
    ` }
  }

  \foreach \i in {0,1,...,\n}
    \filldraw[black] (v\i) circle (2pt);
\end{tikzpicture}
\end{minipage}
$\cdots \to$
\begin{minipage}{\w}
\centering
\begin{tikzpicture}
\def \n {7}
\def \m {1}
  \foreach \i in {0,1,...,\n}
    \coordinate (v\i) at ({360/\n * (\i - 1)}:\s);

  \foreach \j in {1,...,\m} {
      \foreach \i [evaluate=\i as \next using {int(mod(\i+\j,\n))}] in {0,1,...,\n} {
        \draw (v\i) -- (v\next);
    ` }
  }

  \foreach \i in {0,1,...,\n}
    \filldraw[black] (v\i) circle (2pt);
\end{tikzpicture}
\end{minipage}
\caption{Beginning of non-trivial sequence converging to the seven-cycle (\autoref{th:intro:rational-right-cont}).}
\label{fig:ex-c7}
\end{figure}
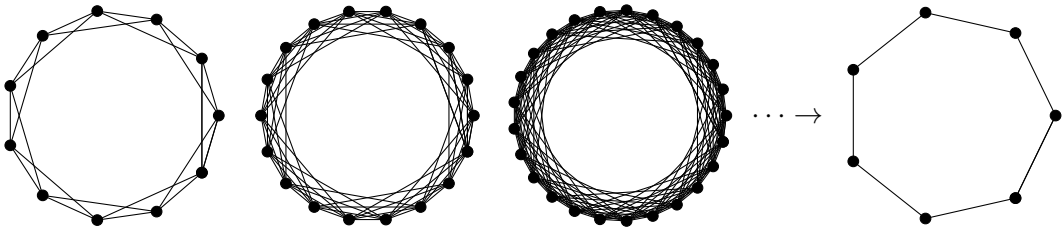

As we will see, the graph limit point-of-view opens the door to analyzing the Shannon capacity problem from many new angles, including: via infinite graphs, topological and analytical methods, and the theory of dynamical systems. It brings many of the best-known Shannon capacity constructions in a unified picture, and offers new paths forward. 

We briefly summarize here our main results and we will expand on these in the rest of the introduction:
\begin{itemize}
    \item \textbf{Convergence and continuity.} We develop asymptotic spectrum distance, including methods for bounding it. These lead to a very general construction of non-trivial converging sequences of graphs (in a family of circulant graphs that we call fraction graphs, and are also called circular graphs). As a consequence, we prove new continuity properties of the Shannon capacity, the Lovász theta function and other graph parameters. (\autoref{subsec:intro:conv})
    
    \item \textbf{Infinite graphs as limit points.} We construct Cauchy sequences of finite graphs that do not converge to any finite graph, but do converge to an infinite graph. More generally we establish strong connections between convergence questions of (finite) graphs and the asymptotic properties of (Borsuk-like) infinite graphs on the circle, and use ideas from the theory of dynamical systems to study these infinite graphs. (\autoref{subsec:intro:infinite})

    \item \textbf{Graph limit approach and orbit constructions.} We observe that all best-known lower bound constructions for Shannon capacity of small odd cycles can be obtained from a ``finite'' version of the graph limit approach, in which we approximate the target graph by another graph with a large independent set that is an ``orbit'' under a natural group action.\footnote{This framework is reminiscent of fast matrix multiplication algorithm constructions based on Coppersmith--Winograd-type~\cite{coppersmith1987matrix} and laser method type \cite{strassen1987relative} reductions.} We develop computational and theoretical aspects of this approach and use it to obtain a new Shannon capacity lower bound for the fifteen-cycle. (\autoref{subsec:intro:orbits}) %

\end{itemize}

The asymptotic spectrum distance theory applies not only to graphs and the Shannon capacity problem, but indeed is much more general and applies to other kinds of mathematical objects and their asymptotic (under large powers) properties. We will thus develop parts of the theory in the generality of preordered semirings. Remarkably, comparing asymptotic spectrum distance for the semiring of graphs and the semiring of tensors (which has a very long history and rich theory \cite{strassen1987relative, strassen1988asymptotic, strassen1991degeneration,wigderson2022asymptotic, MR4495838}), we see a striking difference: while for graphs we have non-trivial converging sequences and non-converging Cauchy sequences, it follows from recent works \cite{bdr, DBLP:conf/innovations/BrietCLSZ24} that tensors show much more discrete behaviour, and in particular that for tensors over finite fields no such sequences exist. (\autoref{subsec:intro:gen}) %

We have formalized the results of this paper (including all prerequisites) in Lean~4, see the repository:
\url{https://github.com/jzuiddam/asymptotic-spectrum-distance}.

In the rest of the introduction we discuss basic notions and main results in more detail.

\subsection{Shannon capacity and asymptotic spectrum distance}
We discuss some basic notions (a full discussion is in \autoref{sec:dist}) and then elaborate on our results.
The Shannon capacity of a graph $G$ is defined %
as the rate of growth of the independence number under taking large powers of the graph (under the ``strong product''), that is, $
\Theta(G) = \lim_{n\to \infty} \alpha(G^{\boxtimes n})^{1/n}$ \cite{MR0089131}.
By Fekete's lemma this equals the supremum $\sup_n \alpha(G^{\boxtimes n})^{1/n}$, so that we may think of Shannon capacity as a maximisation problem. %

Asymptotic spectrum duality \cite{MR4039606} characterizes the Shannon capacity as the minimization 
$\Theta(G) = \min_{F \in \X} F(G),$
where $\X$, called the asymptotic spectrum of graphs, is a natural family of real-valued functions on graphs (namely, multiplicative, additive, cohomomorphism-monotone, normalized functions, which we discuss in detail later, \autoref{subsec:asymp-spec-dual}).
The asymptotic spectrum $\X$ contains the Lovász theta function, fractional Haemers bound, fractional clique covering number, and other well-known upper bounds on Shannon capacity. While we do not have an explicit description of all elements of~$\X$, several structural properties are known \cite{MR4357434}.

We define the asymptotic spectrum distance by
$d(G, H) = \sup_{F\in \X} |F(G) - F(H)|$ for any two graphs~$G,H$.\footnote{Strictly speaking, this is not a distance on graphs, but it is a distance if we identify graphs under ``asymptotic equivalence'', \autoref{sec:dist}.} 
From asymptotic spectrum duality it can be seen that converging graphs have converging Shannon capacity (\autoref{lem:Theta-conv}).
As we will see (\autoref{lem:alt-charac}), asymptotic spectrum distance  can be phrased in an operational form using asymptotic spectrum duality. Namely, $d(G,H)$ is small if and only if a ``polynomial blow up'' of $G$ can be mapped (cohomomorphically) to a slightly larger polynomial blow up of $H$, and similarly for $G$ and $H$ reversed.%
\footnote{Write $G \leq H$ if there is a cohomomorphism $G \to H$ (a map $V(G)\to V(H)$ that maps non-edges to non-edges). Then for any $a,b \in \NN$ we have $d(G,H) \leq a/b$ if and only if $(E_b\boxtimes G)^{\boxtimes n} \leq ((E_b\boxtimes H) \sqcup E_a)^{\boxtimes(n + o(n))}$ and $(E_b \boxtimes H)^{\boxtimes n} \leq ((E_b\boxtimes G) \sqcup E_a)^{\boxtimes(n + o(n))}$ (\autoref{lem:alt-charac}), where $E_n \coloneqq \overline{K}_n$ denotes the graph with $n$ vertices and no edges.}

\subsection{Converging sequences}\label{subsec:intro:conv}

We construct (non-trivial) converging sequences of graphs in the asymptotic spectrum distance. For this we use a natural family of vertex-transitive graphs which we call fraction graphs.\footnote{In the literature these have received much attention (under many names), by Vince \cite{MR0968751}, Schrijver and Polak (as circular graphs)~\cite{SvenThesis, MR3906144}, as cycle-powers~\cite{MR3016977}, and as (the complement of) rational complete graphs~\cite{MR2089014} and circular complete graphs~\cite{MR2249284}.} For any $p,q\in \NN$ with $p/q\geq2$ we let $E_{p/q}$ be the graph with vertex set $V = \{0,\ldots,p-1\}$ and with distinct $u,v \in V$ forming an edge if $u-v$ or $v-u$ is strictly less than $q$ modulo $p$.\footnote{Despite the use of the letter $p$ we do not require $p$ to be prime.} For example, $E_{8/3}$ is the following graph:
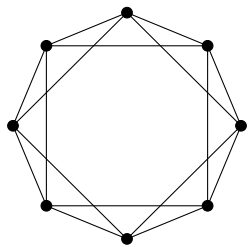
\[
\begin{tikzpicture}
  \foreach \i in {0,1,...,7}
    \coordinate (v\i) at ({360/8 * (\i - 1)}:1.5cm);

  \foreach \i [evaluate=\i as \next using {int(mod(\i+1,8))}] in {0,1,...,7}
    \draw (v\i) -- (v\next);
  \foreach \i [evaluate=\i as \next using {int(mod(\i+2,8))}] in {0,1,...,7}
    \draw (v\i) -- (v\next);

  \foreach \i in {0,1,...,7}
    \filldraw[black] (v\i) circle (2pt);
\end{tikzpicture}
\]
The graphs $E_{p/q}$, which we call \emph{fraction graphs}, have many nice properties. We write $G \leq H$ if there is a cohomomorphism $G \to H$ (a map $V(G)\to V(H)$ that maps non-edges to non-edges). Then: $p/q \leq s/t$ (in $\QQ$) if and only if $E_{p/q} \leq E_{s/t}$ in the cohomomorphism preorder. That is, fraction graphs are ordered as the rational numbers, and in particular inherit their denseness. (We note that $E_{p/q}$ and $E_{np/nq}$, for $n \in \NN$, are not isomorphic graphs, but they are equivalent under cohomomorphism.) 
The fraction graphs contain all cycle graphs as $C_n = E_{n/2}$, and the edgeless graphs as $E_m = E_{m/1} = \overline{K}_m$.

We develop techniques to bound the asymptotic spectrum distance of vertex-transitive graphs and their induced subgraphs (\autoref{subsec:dist-vt}), and a vertex removal strategy for fraction graphs (\autoref{subsec:vertex-removal-strategy}). Using these we obtain a general construction of converging sequences:

\begin{theorem}[\autoref{th:rational-right-cont}]\label{th:intro:rational-right-cont}
    For any $a/b \geq 2$, if $p_n/q_n$ converges to $a/b$ from above, then $E_{p_n/q_n}$ converges to $E_{a/b}$.\footnote{We note that, in \autoref{th:intro:rational-right-cont}, if the rational numbers $p_n/q_n$ are distinct, then (using the fractional clique covering number) the graphs $E_{p_n/q_n}$ can be seen to be pairwise non-equivalent and even asymptotically non-equivalent, as we will explain later, so the sequence is indeed non-trivial.}
\end{theorem}

\autoref{th:intro:rational-right-cont} solves a problem and extends the work of Schrijver and Polak \cite[Chapter~9]{SvenThesis}. 
They constructed (by means of an independent set construction) for every integer $m \in \NN$ a sequence of fraction graphs~$E_{p_n/q_n}$ that converges to $E_{m} = E_{m/1}$ from below, which implies that for any sequence $p_n/q_n$ converging to $m$ from below we have that $E_{p_n/q_n}$ converges to $E_m$.

We leave it as an open problem whether $p_n/q_n \to a/b$ from below also implies $E_{p_n/q_n} \to E_{a/b}$. In the next section we discuss how this problem is closely related to properties of infinite graphs.

\autoref{th:intro:rational-right-cont} can alternatively be phrased in terms of continuity of the graph parameters in the asymptotic spectrum, $F \in \X$. Namely, it says that $p/q \mapsto F(E_{p/q})$ is right-continuous (and the same for $F = \Theta$).

With similar methods we get: %

\begin{theorem}[\autoref{thm: continuous at irrationals}]\label{th:intro:cont-irr}
    For any irrational $r \in \RR_{\geq2}$, if $p_n/q_n$ converges to $r$, then $E_{p_n/q_n}$ is Cauchy.
\end{theorem}

It follows from \autoref{th:intro:cont-irr} with a simple argument (\autoref{th:not-complete2}) that there are Cauchy sequences that do not converge to any finite graph. Indeed, the fractional clique cover number of any sequence in \autoref{th:intro:cont-irr} converges to the irrational number~$r$, which cannot be the fractional clique cover number of a finite graph. This answers a question raised in \cite[Chapter~9]{SvenThesis}.

There is a natural (albeit abstract) way to complete the space of finite graphs with the asymptotic spectrum distance, by identifying every graph $G$ with its evaluation function $\widehat{G} : \X \to \RR : F \mapsto F(G)$. The coarsest topology on $\X$ that makes the $\widehat{G}$ continuous is compact, and thus we can think of graphs as living in the space $C(\X)$ of continuous functions on $\X$ endowed with the sup-norm, which is complete. 

The above (keeping in mind the analogous theory of cut-norm and graphons) raises the question whether there are concrete, graph-like models for limit points; indeed we prove that infinite graphs provide such models. %

\subsection{Infinite graphs as limit points}\label{subsec:intro:infinite}

We prove that certain infinite graphs (on the circle) are limit points for non-converging Cauchy sequences of finite graphs, and investigate the properties of these graphs, and strong connections to the Shannon capacity of finite graphs.

We define two kinds of infinite graphs on the circle (that may be thought of as infinite versions of the fraction graphs). Let $C \subseteq \RR^2$ denote a circle with unit circumference. For $r \in \RR_{\geq2}$ we define $E_r^\open$ to be the infinite graph with vertex set $C$ for which two vertices are adjacent if and only if they  have distance strictly less than $1/r$ on $C$. We define $E_r^\closed$ to be the infinite graph with vertex set $C$ for which two vertices are adjacent if and only if they have distance at most $1/r$ on $C$. The subtle difference between the ``open'' and ``closed'' version of these circle graphs will play an important role.

We naturally extend asymptotic spectrum duality and distance to infinite graphs (with finite clique covering number), and are then able to establish the circle graphs as limit points:

\begin{theorem}[\autoref{th:irr-closed-open-equiv}]\label{th:intro:irr-closed-open-equiv}
    For any irrational $r \in \RR_{\geq2}$, if $p_n/q_n$ converges to $r$, then $E_{p_n/q_n}$ converges to $E_r^\open$.
\end{theorem}

For any two graphs $G,H$ we write $G \asympleq H$ if $G^{\boxtimes n} \leq H^{\boxtimes (n + o(n))}$. We say $G$ and~$H$ are asymptotically equivalent, if $G \asympleq H$ and $H \asympleq G$. It follows from our proof of \autoref{th:intro:irr-closed-open-equiv} that $E_r^\closed$ is asymptotically equivalent to $E_r^\open$ for irrational $r$. For rational $r$ we find:

\begin{theorem}[\autoref{th:op-cl-left}] Let $r = p/q\in \QQ_{\geq2}$. The following are equivalent:
    \begin{enumerate}[\upshape(i)]
        \item $E_r^\closed$ and $E_r^\open$ are asymptotically equivalent.
        \item If $a_n/b_n \to p/q$ from below, then $E_{a_n/b_n} \to E_{p/q}$.
    \end{enumerate}
\end{theorem}

Graphs $G$ and $H$ are called equivalent if $G \leq H$ and $H \leq G$. Equivalent graphs are asymptotically equivalent (but asymptotically equivalent graphs may not be equivalent).
As a step towards understanding $E_r^\closed$ and $E_r^\open$ better, we prove that they are not equivalent: %

\begin{theorem}[\autoref{sec:nonequiv}]
    Let $r \in \mathbb{R}_{>2}$. Then $E_{r}^\closed$ and $E_{r}^\open$ are not equivalent.
\end{theorem}

The proof of this result is based on a complete characterization of the self-co\-ho\-mo\-morphisms of the circle graphs $E_r^\closed$ and $E_r^\open$ (\autoref{lem:rotation-reflection}), for which we use ideas from the theory of dynamical systems.

\subsection{Independent sets from orbit constructions and reductions}\label{subsec:intro:orbits}

Having delved into asymptotic spectrum distance and how to construct converging sequences, we focus in this part on the construction of explicit independent sets, their structure and ways of transforming them between graphs (reductions).

Over time, several structured and concise constructions of independent sets in powers of graphs (say, odd cycles) have been found and investigated, starting with the famous independent set $\{t \cdot (1,2) : t \in \ZZ_5\}$ in~$C_5^{\boxtimes 2}$ \cite{MR0089131}, and later \cite{MR0337668, LinearShannonCapacity, MR3906144}, among others. In the recent remarkable work by Google DeepMind \cite{romera-paredes_mathematical_2024}, a large language model (LLM) not only recovered the best-known lower bounds for the Shannon capacity of small odd cycles, but moreover did so in a concise manner (in the sense of Kolmogorov-complexity).\footnote{More precisely, the authors use an LLM to produce short computer programs which produce independent sets of size matching the best-known Shannon capacity lower bounds for small odd cycles.} %

We propose a simple, general framework in which essentially all best-known lower bounds on Shannon capacity of odd cycles fit---offering an explanation for the aforementioned structure and conciseness. This framework, which can be thought of as a finite version of the graph limit approach, consists of (1) relating the target graph $G = C_n = E_{n/2}$ to another fraction graph $H = E_{p/q}$ (and ``intermediate'' or ``auxiliary'' graph), and then (2) constructing a large independent set in a power of $H$ using an ``orbit'' construction. The aforementioned constructions of Shannon \cite{MR0089131} and Polak--Schrijver~\cite{MR3906144} are indeed of this form; other bounds (e.g., by Baumert, McEliece, Rodemich, Rumsey, Stanley and Taylor~\cite{MR0337668}) can be recovered in this way. We discuss this in \autoref{subsec:odd-cycl-overview} (and in particular \autoref{tab:overview}). 

The framework described above is remarkably similar to the state-of-the-art methods used for constructing fast matrix multiplication algorithms. Indeed, the (asymptotically) fastest matrix multiplication algorithms are obtained by a reduction of the ``matrix multiplication tensor'' to an ``intermediate tensor'' for which the relevant properties are easier to analyse \cite{coppersmith1987matrix,blaser2013fast,DBLP:conf/stoc/AmbainisFG15,v017a002,alman2018limits,williams2023new, wigderson2022asymptotic}.

In \autoref{subsec:fifteen} we develop methods for reducing a target graph $G$ to an auxiliary graph and in particular introduce a new ``nondeterministic rounding'' technique. 
Using this technique applied to an orbit construction in a power of a fraction graph close to the fifteen-cycle, we find a new Shannon capacity lower bound:

\begin{theorem}[\autoref{th:C15-4}]
    $\Theta(C_{15}) \geq \alpha(C_{15}^{\boxtimes 4})^{1/4} \geq 2842^{1/4} \approx 7.30139$.
\end{theorem}

Our bound improves the previous bound $\Theta(C_{15}) \geq 7.25584$ which was obtained by Codenotti, Gerace and Resta \cite{MR1961489} and independently by Polak and Schrijver (personal communication).

In \autoref{subsec:disc} we investigate independent sets in products of fraction graphs further. We focus on describing, for fixed $k$, the function $p/q \mapsto \alpha(E_{p/q}^{\boxtimes k})$, and more generally the multivariate version of this function, $(p_1/q_1, \ldots, p_k/q_k) \mapsto \alpha(E_{p_1/q_1} \boxtimes \cdots \boxtimes E_{p_k/q_k})$. 
For~$k=1$, a simple argument gives $\alpha(E_{p/q}) = \lfloor p/q\rfloor$. For $k=2$, Hales \cite{MR0321810} and Badalyan and Markosyan \cite{MR3016977} proved
\[
\alpha(E_{p_1/q_1} \boxtimes E_{p_2/q_2}) = \min \{ \lfloor\lfloor p_1/q_1 \rfloor p_2/q_2\rfloor, \lfloor\lfloor p_2/q_2 \rfloor p_1/q_1\rfloor\}.
\]
We give a simple proof of this result using orbits (\autoref{th:two-factors}). 
As our main result, we consider $k=3$ and determine $\alpha_3(p_1/q_1,p_2/q_2,p_3/q_3) \coloneqq \alpha(E_{p_1/q_1} \boxtimes E_{p_2/q_2} \boxtimes E_{p_3/q_3})$ for all $p_i/q_i \in [2,3]$. Since $\alpha_3$ is monotone in every coordinate, it suffices to determine the discontinuities to describe the function:

\begin{theorem}[\autoref{th:discont}, \autoref{fig:hasse}]\label{th:intro:discont}
The discontinuities of $\alpha_3$ restricted to the set $(\QQ\cap[2,3])^3$ are, up to permutation,~at
    \begin{align*}
        \alpha_3(2,2,2) &= 8            & \alpha_3(5/2, 5/2, 8/3) &= 11\\
        \alpha_3(2,2,3) &= 12           & \alpha_3(8/3, 8/3, 8/3) &= 12\\
        \alpha_3(2,3,3) &= 18           & \alpha_3(11/5, 11/4, 11/4) &= 11\\
        \alpha_3(2, 5/2, 5/2) &= 10     & \alpha_3(11/4, 11/4, 11/4) &= 13\\
        \alpha_3(5/2, 5/2, 3) &= 15     & \alpha_3(14/5, 14/5, 14/5) &= 14\\        
        \alpha_3(9/4, 7/3, 5/2) &= 9    & \alpha_3(3,3,3) &= 27.
    \end{align*}
\end{theorem}

\subsection{Asymptotic spectrum distance of tensors, matrix multiplication}\label{subsec:intro:gen}

Let us take a step back and note again that the theory of asymptotic spectrum distance applies much more broadly than just to graphs. Namely, it applies to any semiring with a preorder for which the asymptotic spectrum duality holds; the exact conditions for this we discuss in \autoref{sec:gen-asymp-spec}. 

An interesting and much studied setting where asymptotic spectrum duality exists (and thus asymptotic spectrum distance) is the semiring of tensors (under direct sum and tensor product) with the ``restriction preorder'' (\autoref{sec:gen-asymp-spec}). This was in fact the original application of asymptotic spectrum duality in the seminal work \cite{strassen1987relative, strassen1988asymptotic, strassen1991degeneration}, motivated by the problem of constructing fast matrix multiplication algorithms. %

Remarkably, recent works of Draisma, Blatter and Rupniewski~\cite{bdr} and Briët, Christandl, Leigh, Shpilka and Zuiddam \cite{DBLP:conf/innovations/BrietCLSZ24} show that (contrary to the graph setting) many asymptotic tensor parameters are discrete. In particular, it can be shown that over any finite field, with the asymptotic spectrum distance, any Cauchy sequence of tensors must be eventually constant, and thus %
the space of tensors with asymptotic spectrum distance is complete. 

The discreteness in the tensor setting is in stark contrast to what we see in the graph setting of this paper:
as we have discussed (\autoref{subsec:intro:conv}), in the space of graphs with the asymptotic spectrum distance we can construct non-trivial converging sequences of graphs, and also non-converging Cauchy sequences. It is this surprisingly  ``continuous'' behaviour of graphs that we aim to capitalize on in this work.

\subsection{Paper overview}

In \autoref{sec:dist} we review the basic notation and background around Shannon capacity, including asymptotic spectrum duality, and we prove the basic properties and characterizations of asymptotic spectrum distance. (This section deals only with finite graphs for simplicity. The straightforward extension to infinite graph we discuss in \autoref{sec:asymp-spec-infinite}.)
In \autoref{sec:frac} we use fraction graphs to construct non-trivial converging sequences.
In \autoref{sec:completion} we work with infinite graphs and in particular with graphs on the circle in the real plane. We prove that they are limit points for certain sequences of fraction graphs. We prove a strong relation between convergence questions for fraction graphs and asymptotic equivalence of ``open'' and ``closed'' circle graphs.
In \autoref{sec: self-cohomomorphisms} we study self-cohomomorphisms of fraction graphs and circle graphs and use these to prove non-equivalence of the ``open'' and ``closed'' version of circle graphs.
In \autoref{sec:computation} we discuss the structure and computation of independent sets in products of fraction graphs, and prove a new lower bound on the Shannon capacity of the fifteen-cycle.
\autoref{sec:discussion} is a discussion of future directions and problems. 

\section{Shannon capacity and asymptotic spectrum distance}
\label{sec:dist}

The goal of this section is to define a distance function on graphs that is geared towards attacking the Shannon capacity problem---and indeed arises naturally from it---and to discuss its basic properties.

We begin with a brief discussion of Shannon capacity and related basic notions. We then discuss asymptotic spectrum duality, which provides a dual description for the Shannon capacity. 
Asymptotic spectrum duality naturally leads to a distance on graphs, the asymptotic spectrum distance. Finally, we discuss a method for constructing pairs of graphs that have small distance in the asymptotic spectrum distance. 

In this section we only work with finite graphs, for simplicity and since this suffices for \autoref{sec:frac}. Most ideas, in particular asymptotic spectrum duality and distance, extend directly to infinite graphs. Later, in \autoref{sec:completion} and \autoref{sec: self-cohomomorphisms}, we will indeed use infinite graphs, and we will discuss the needed material for infinite graphs there (\autoref{sec:asymp-spec-infinite}). In fact, most material in this section holds much more generally, which we discuss in~\autoref{sec:gen-asymp-spec} (and see the general survey \cite{wigderson2022asymptotic}).

\subsection{Shannon capacity and basic notions}

Let $G$ and $H$ be (finite, simple) graphs. We denote by $G \boxtimes H$ the strong product of~$G$ and $H$.\footnote{The vertex set $V(G\boxtimes H)$ is the cartesian product $V(G) \times V(H)$ and the edge set is given by $E(G\boxtimes H) = \{\{(a,x),(b,y)\} : (a=b \vee \{a,b\} \in E(G)) \wedge (x=y \vee \{x,y\} \in E(H)) \}$. This product is also called the AND product. It can alternatively be described in terms of adjacency matrices: Let $A(G)$ be the adjacency matrix of $G$ in which we put ones on the main diagonal. Then $A(G \boxtimes H) = A(G) \otimes A(H)$, where $\otimes$ is the Kronecker product of matrices.} 
We denote by $\alpha(G)$ the independence number of $G$ (which is defined as the size of the largest independent set in $G$) and by $\Theta(G) = \lim_{n\to\infty} \alpha(G^{\boxtimes n})^{1/n}$ the Shannon capacity of $G$ \cite{MR0089131}. It is known by Fekete's lemma that the limit exists and can be replaced by a supremum. Determining the Shannon capacity of a graph is a hard problem, even for small graphs. Various techniques are known to upper bound the Shannon capacity, including the (fractional) clique covering number \cite{MR0089131}, (fractional) orthogonal rank (projective rank) \cite{MR1417351,MR3649190}, Lovász theta function \cite{Lovasz79}, and (fractional) Haemers bound \cite{haemers1979some, blasiakThesis, Bukh18}. %

We use the following basic terminology throughout the paper.
We denote by $G \sqcup H$ the disjoint union of $G$ and $H$. 
We denote by $\overline{G}$ the complement graph of $G$. %
We denote the complete graph on $n$ vertices by~$K_n$, and its complement by $E_n = \overline{K}_n$, so that $E_n$ is the graph with $n$ vertices and no edges. 
With the addition $\sqcup$ and multiplication $\boxtimes$ the set of (isomorphism classes of) graphs becomes a commutative semiring. In this semiring, the graphs $E_0$ and $E_1$ are the additive and multiplicative unit, respectively.

We say that a map $f : V(G) \to V(H)$ is a cohomomorphism $G \to H$ if for every distinct $u,v \in V(G)$ if $\{u,v\} \not\in E(G)$, then $f(u) \neq f(v)$ and  $\{f(u),f(v)\} \not\in E(H)$. We write~$G \leq H$ if there is a cohomomorphism $G \to H$. We call $\leq$ the cohomomorphism preorder. Cohomomorphisms are precisely the maps that preserve independent sets, and in particular $G \leq H$ implies $\alpha(G) \leq \alpha(H)$.
We say that~$G$ and~$H$ are equivalent if $G \leq H$ and $H \leq G$. 

The independence number can be written in terms of cohomomorphisms and the graphs $E_n$ by $\alpha(G) = \max \{n \in \NN : E_n \leq G\}$. The ``opposite'' graph parameter $\overline{\chi}(G) = \min \{n \in \NN : G \leq E_n\}$ is known as the clique covering number, and equals the smallest number of cliques in the graph $G$ that are needed to cover all its vertices.\footnote{The (standard) notation $\overline{\chi}$ for the clique covering number originates from the relation of this parameter to the chromatic number, which is denoted by $\chi$. Namely, $\overline{\chi}(G) = \chi(\overline{G})$.}

\subsection{Asymptotic spectrum duality}\label{subsec:asymp-spec-dual}

Understanding the Shannon capacity %
naturally asks for understanding how graphs behave when taking large powers (under the strong product~$\boxtimes$). For this we will use ideas developed in the theory of asymptotic spectrum duality \cite{strassen1988asymptotic,Zui18_thesis,MR4039606,MR4609385,wigderson2022asymptotic}, which (for convenience of the reader) we will explain here in the setting of graphs, following \cite{MR4039606}.

For any graphs $G,H$, we write $G \asympleq H$ if there is a function $f: \NN \to\NN$ with $f(n) = o(n)$ such that for all $n \in \NN$ we have $G^{\boxtimes n} \leq H^{\boxtimes (n+f(n))}$. We call ${\asympleq}$ the asymptotic preorder. We say that~$G$ and $H$ are asymptotically equivalent if $G \asympleq H$ and $H \asympleq G$. Note that in particular, Shannon capacity is monotone under $\asympleq$.

The asymptotic spectrum of graphs $\mathcal{X}$ is defined as the set of all maps $F$ from graphs to non-negative reals, such that for any graphs $G,H$ we have: $F(G \sqcup H) = F(G) + F(H)$ (additivity), $F(G \boxtimes H) = F(G) F(H)$ (multiplicativity), $F(E_n) = n$ for every $n \in \NN$ (normalisation), and if $G \leq H$, then $F(G) \leq F(H)$ (monotonicity). 

It is an open problem to fully determine $\X$.
Known elements in $\mathcal{X}$ are the Lovász theta function $\vartheta$ \cite{Lovasz79}, the fractional Haemers bound (over different fields)~\cite{Bukh18}, the fractional orthogonal rank (projective rank), and the fractional clique covering number~$\overline{\chi}_f$ (see \cite{MR4039606}). Continuously many elements in $\X$ can be constructed using the interpolation method of Vrana \cite{MR4357434}. All of these are upper bounds on Shannon capacity, and it is not hard to prove that indeed every element of $\X$ is an upper bound on Shannon capacity. 

The asymptotic spectrum duality theorem says, in particular, that the upper bounds in $\X$ are sufficiently strong to determine the Shannon capacity. In fact, it says that the functions in~$\X$ are strong enough to completely characterize the asymptotic preorder~$\asympleq$.

\begin{theorem}[Asymptotic spectrum duality]\label{th:duality-Theta}\label{th:chi-bar-max} Let $G,H$ be graphs.
\begin{enumerate}[\upshape(a)]
    \item %
    $\Theta(G) = \min_{F \in \X} F(G)$.
    \item %
    $\overline{\chi}_f(G) = \max_{F \in \X} F(G)$.
    \item %
    $G \asympleq H$ if and only if for every $F \in \X$, $F(G) \leq F(H)$.
\end{enumerate}
\end{theorem}

\begin{remark}
    The Shannon capacity $\Theta$ itself is not in $\X$, which follows from results of Haemers~\cite{haemers1979some} and Alon~\cite{alon1998shannon}. Thus, in \autoref{th:duality-Theta}, the element $F$ minimizing $\min_{F \in \X} F(G)$ depends on $G$. On the other hand, the fractional clique covering number~$\overline{\chi}_f$ is in $\X$, and it is the pointwise largest function in there. This is closely related to the nontrivial fact that $\overline{\chi}_f(G) = \lim_{n\to\infty} \overline{\chi}(G^{\boxtimes n})^{1/n}$ \cite[Theorem 67.17]{MR1956924}.
\end{remark}

\begin{remark}
    \autoref{th:duality-Theta} characterizes $\Theta$ as a minimization, rather than an infimum. This is a consequence of compactness of $\X$, which intuitively follows from the fact that the defining conditions of $\X$ are closed conditions and the fact that for every graph $G$, the values that any $F \in X$ can take on $G$ are bounded (by $\alpha(G)$ and $\overline{\chi}(G)$, say). We discuss compactness  further in \autoref{subsec-graph-cont-func}.
\end{remark}

\subsection{Distance and convergence}\label{sec:metric}

The functions in the asymptotic spectrum $\X$ provide, in the precise sense of \autoref{th:duality-Theta}, all asymptotic information of a graph. They also naturally allow us to measure distance between graphs in a way that is appropriate in the context of Shannon capacity. We discuss this distance and basic properties.

\begin{definition}\label{def:dist}
For any graphs $G,H$, we define the \emph{asymptotic spectrum distance} by
\[
d(G,H) = \sup_{F \in \X} |F(G) - F(H)|.
\]
\end{definition}

\begin{remark}
    The asymptotic spectrum distance is not quite a distance on graphs, because for distinct graphs $G,H$ we may have $d(G,H)=0$. However, if we let $\mathcal{G}$ denote the set of all graphs in which we identify graphs that are asymptotically equivalent (recall that $G,H$ are asymptotically equivalent if $G \asympleq H$ and $H \asympleq G$, \autoref{subsec:asymp-spec-dual}), then the asymptotic spectrum distance is indeed a distance on $\mathcal{G}$ (this follows from \autoref{th:duality-Theta}) and the set $\mathcal{G}$ is a metric space.
\end{remark}

\begin{remark}\label{rem:compact}
The supremum in \autoref{def:dist} is in fact attained. This follows from compactness of $\X$, as we will discuss in \autoref{subsec-graph-cont-func}.
\end{remark}

\begin{lemma}\label{lem:Theta-conv}
    If a sequence of graphs $(G_i)_{i \in \NN}$ converges to a graph $H$ in the asymptotic spectrum distance, then $(\Theta(G_i))_{i\in \NN}$ converges to $\Theta(H)$.
\end{lemma}
\begin{proof}
    Let $\eps > 0$. There is an $N$ such that for every $i > N$ and every $F \in \X$ we have $|F(G_i) - F(H)| < \eps$.
    Let $F_{G_i}, F_H \in \X$ such that $F_{G_i}(G_i) = \Theta(G_i)$ and $F_H(H) = \Theta(H)$ (\autoref{th:duality-Theta}).
    Then we have
    $\Theta(H) = F_H(H) > F_H(G_i) - \eps \geq \Theta(G_i) - \eps$
    and
    $\Theta(G_i) = F_{G_i}(G_i) > F_{G_i}(H) - \eps \geq \Theta(H) - \eps$.
    Thus $|\Theta(G_i) - \Theta(H)| < \eps$.
\end{proof}

\begin{lemma}\label{lem:alt-charac} Let $G,H$ be graphs and $a,b \in \NN$, $b \geq1$. The following are equivalent.
\begin{enumerate}[\upshape(i)]
    \item $d(G,H) \leq a/b$
    \item $E_b \boxtimes G \asympleq (E_b \boxtimes H) \sqcup E_a$ and $E_b \boxtimes H \asympleq (E_b \boxtimes G) \sqcup E_a$
    \item $(E_b\boxtimes G)^{\boxtimes n} \leq ((E_b\boxtimes H) \sqcup E_a)^{\boxtimes(n + o(n))}$ and $(E_b \boxtimes H)^{\boxtimes n} \leq ((E_b\boxtimes G) \sqcup E_a)^{\boxtimes(n + o(n))}$
\end{enumerate}
\end{lemma}
\begin{proof}
    We give the proof (which is a simple application of \autoref{th:duality-Theta}) in a more general setting in \autoref{lem:abstract-dist-char}.
\end{proof}

\subsection{Graphs as continuous functions on a compact space}\label{subsec-graph-cont-func}

So far we have treated the asymptotic spectrum $\X$ simply as a set of functions on the semiring of graphs, and we made the semiring of (asymptotic equivalence classes~of) graphs into a metric space using the distance induced by these functions. We will now take a different perspective that will allow us to apply ideas from analysis and that will in particular provide a natural (albeit abstract) embedding into a complete space.

In this different perspective we will turn $\X$ into a topological space and we will turn graphs into continuous functions on this space. This works as follows. 
We associate to every graph $G$ its evaluation function
\[
\widehat{G} : \X \to \RR : F \mapsto F(G).
\]
We may think of $\widehat{G}$ as the (infinitely long) vector of evaluations of the elements in the asymptotic spectrum on $G$.\footnote{The object $\widehat{G}$ acts as an ``asymptotic profile'' of $G$ in a way that is reminiscent of (but different from) the homomorphism profile of Lovász \cite[Section 5.4.1]{MR3012035}. The latter counts the number of homomorphisms to (or from) a given graph, and uniquely determines graphs up to isomorphism.}
The basic properties of these objects are: 
\begin{lemma}\label{lem:G-hat-prop}
Let $G$ and $H$ be graphs.
\begin{enumerate}[\upshape(i)]
    \item $\widehat{G\boxtimes H} = \widehat{G} \widehat{H}$
    \item $\widehat{G \sqcup H} = \widehat{G} + \widehat{H}$
    \item $G \asympleq H$ if and only if $\widehat{G} \leq 
    \widehat{H}$ pointwise
    \item $\min \widehat{G} = \Theta(G)$
    \item $\max \widehat{G} = \overline{\chi}_f(G)$
\end{enumerate}    
\end{lemma}
\begin{proof}
    These follow directly from \autoref{th:duality-Theta} and the definition of $\X$.
\end{proof}

We claim that $\X$ can be given a topology that is compact, Hausdorff and such that every function $\widehat{G}$ is continuous on $\X$.
Indeed, we get this by giving $\X$ the topology generated by the functions $\widehat{G}$, that is, the coarsest (i.e.,~weakest) topology for which every $\widehat{G}$ is continuous.\footnote{This is also called the weak or initial topology on $\X$ induced by the functions $\widehat{G}$. This is the topology on $\X$ that is generated (under unions and finite intersections) by the inverse images $\widehat{G}^{-1}(U)$ where $G$ goes over all graphs and $U$ over all open subsets of the nonnegative reals (with the Euclidean topology).} Then~$\X$ is easily seen to be Hausdorff (but we will not use this). Importantly, it is compact:

\begin{lemma}\label{lem:compact}
    The set $\X$ is compact in the coarsest topology such that for every graph~$G$ the function $\widehat{G}$ is continuous. %
\end{lemma}

\begin{proof}
    This is a ``standard'' fact \cite{MR0707730,strassen1988asymptotic,wigderson2022asymptotic}, but we give here a proof for the convenience of the reader.
    We consider $\X$ as a subset of $\prod_{G \in \mathcal{G}} \RR$ with inclusion map $i : F \mapsto (F(G))_{G \in \mathcal{G}}$.
    Endowing $\prod_{G \in \mathcal{G}} \RR$ with the product topology, the map $i$ is continuous if and only if all evaluation functions~$\widehat{G}$ are continuous  \cite[Theorem~19.6]{MR3728284}. The coarsest topology on~$\X$ that makes $i$ continuous is by definition the subspace topology.
    For every graph $G$, the set of values $\{F(G) : F \in \X\}$ is a subset of $[0, \overline{\chi}(G)]$, where~$\overline{\chi}$ denotes the clique covering number. Thus $\X$ is a subset of $\prod_{G \in \mathcal{G}} [0, \overline{\chi}(G)]$.
    Since $\prod_{G \in \mathcal{G}} [0, \overline{\chi}(G)]$ is compact by Tychonoff's theorem \cite[Theorem~37.3]{MR3728284}, it is now sufficient to show that~$\X$ is closed.

    For every graph $H$, denote by $\pi_H : \prod_{G \in \mathcal{G}} \RR \to \RR$ the projection on the coordinate indexed by~$H$. These are continuous by definition of the product topology and so is any polynomial combination of them. We can write $\X= Z_{\text{add}} \cap Z_{\text{mul}} \cap Z_{\text{norm}} \cap Z_{\text{mon}}$, where
    \begin{alignat*}{3}
        Z_{\text{add}} &= \bigcap_{G,H \in \mathcal{G}} (\pi_{G \sqcup H} - (\pi_G + \pi_H))^{-1}(\{0\}), \quad
        &&Z_{\text{mul}} &&= \bigcap_{G,H \in \mathcal{G}} (\pi_{G \boxtimes H} - \pi_G \pi_H)^{-1}(\{0\}), \\
        Z_{\text{norm}} &= \bigcap_{n \in \mathbb{N}} \pi_{E_n}^{-1} (\{n\}),
        &&Z_{\text{mon}} &&= \bigcap_{\substack{G,H \in \mathcal{G} \\ G \leq H}} (\pi_H - \pi_G)^{-1}([0, \infty)). 
    \end{alignat*}
    All sets are written as the intersections of pre-images of closed sets under continuous maps and are thus closed. We conclude that $\X$ is closed.
\end{proof}

\begin{remark}
    Since the image of a compact space under a continuous map is compact, we see that $\{F(G) : F \in \X\} = \widehat{G}(\X)$ is a compact subset of the reals, that is, closed and bounded. In particular, infimum and supremum over this set are attained (cf.~\autoref{th:duality-Theta} (ii) and (iii), and \autoref{lem:G-hat-prop} (iv) and (v)).
\end{remark}

Denote by $C(\X)$ the set of continuous functions $\X \to \RR$. Equip $C(\X)$ with the supremum-norm, which is given by $||f||_\infty = \sup_{F \in \X} |f(F)|$ for $f \in C(\X)$. The distance of any $f,g \in C(\X)$ is given by $||f-g||_\infty$. Then $C(\X)$ is a complete metric space~\cite[Theorem~7.15]{MR0385023}. 

The set of continuous functions $C(\X)$ is a semiring (and in fact ring) under point\-wise addition and pointwise multiplication.
The map $G \mapsto \widehat{G}$ is a semiring homomorphism from the semiring of graphs to $C(\X)$. %

\begin{lemma} For any graphs $G,H$, we have 
    $d(G,H) = ||\widehat{G} - \widehat{H}||_\infty$.
\end{lemma}
\begin{proof} %
    $||\widehat{G} - \widehat{H}||_\infty = \sup_{F \in \X} |\widehat{G}(F) - \widehat{H}(F)| = \sup_{F \in \X} |F(G) - F(H)| = d(G,H)$. %
\end{proof}

Through the above construction, graphs correspond to elements of~$C(\X)$. Since~$C(\X)$ is complete, any Cauchy sequence of graphs will indeed have a limit point in~$C(\X)$, but as we will see (\autoref{th:not-complete2}), this limit point may not correspond to a graph! (This answers a question from \cite[Section 9.3]{SvenThesis}.)

To finish this part, we mention that the continuous function point of view naturally lets us use Dini's theorem from analysis to prove the following sufficient condition for convergence of graphs. %

\begin{lemma}\label{th:dini}
    Let $(G_i)_{i\in \NN}$ be a sequence of graphs and let $H$ be a graph. Suppose that $F(G_1), F(G_2), F(G_3), \ldots$ converges to $F(H)$ for every $F \in \X$. If %
    $F(G_1), F(G_2)$, $F(G_3), \ldots$ is nondecreasing for every $F \in \X$, or nonincreasing for every $F \in \X$, %
    then~$G_i$ converges to $H$.
\end{lemma}
\begin{proof}
    The sequence of continuous functions $\widehat{G}_i \in C(\X)$ converges pointwise to the continuous function $\widehat{H} \in C(\X)$. Moreover, $\widehat{G}_i(F)$ is monotone in $i$ for every $F \in \X$ (either all nondecreasing or all nonincreasing). Thus $\widehat{G}_i$ converges uniformly to $\widehat{H}$ by Dini's theorem~\cite[Theorem~7.13]{MR0385023}.
\end{proof}

\subsection{Distance of induced subgraphs of vertex-transitive graphs}\label{subsec:dist-vt}
We give a general method for constructing pairs of graphs that are close in the asymptotic spectrum metric (which we defined in \autoref{sec:metric}). This method makes use of induced subgraphs of vertex-transitive graphs.

Let $G$ be a graph, let $S \subseteq V(G)$, and let $G[S]$ denote the induced subgraph of $G$ to~$S$. Then we have the inequality $G[S] \leq G$ in the cohomomorphism preorder.
In this section we will quantify how close~$G[S]$ is to $G$. 

As a starting point, we state a basic relation between the independence numbers $\alpha(G[S])$ and $\alpha(G)$, which follows from the well-known No-Homomorphism Lemma of Albertson and Collins \cite[Theorem 2]{MR0791653} (see also \cite[Corollary 1.23]{MR2089014}).

\begin{lemma}\label{lem:vertex-transitive-alpha}
Let $G$ be a vertex-transitive graph and $S \subseteq V(G)$. Then
\[
\alpha(G[S]) \leq \alpha(G) \leq \frac{|V(G)|}{|S|} \alpha(G[S]).
\]
\end{lemma}

We will now extend 
\autoref{lem:vertex-transitive-alpha} to elements in the asymptotic spectrum $\X$, that is, we will replace the independence number $\alpha$ by any function $F \in \X$.\footnote{We note that the independence number $\alpha$ is not in $\X$.} The proof is based on an inequality for vertex-transitive graphs $G$ of \cite{MR4357434}, which says $G \leq E_N \boxtimes G[S]$ for $N = \lceil|V(G)| |S|^{-1} \ln |V(G)|\rceil$. (Recall that $E_N = \overline{K}_N$ denotes the graph with $N$ vertices and no edges.) %

\begin{lemma}\label{lem:vertex-transitive-spectrum}
Let $G$ be a vertex-transitive graph and $S \subseteq V(G)$. Let $F \in \X$. Then
\[
F(G[S]) \leq F(G) \leq \frac{|V(G)|}{|S|} F(G[S]).
\]
\end{lemma}
\begin{proof}
The first inequality follows directly from $F$ being monotone under cohomomorphism. 
For any vertex-transitive graph $G$ and $S\subseteq V(G)$ we have $G \leq E_N\boxtimes G[S]$
for $N = \lceil|V(G)| |S|^{-1} \ln |V(G)|\rceil$. %
Let $m \in \NN$. Then~$G[S]^{\boxtimes m}$ is an induced subgraph of~$G^{\boxtimes m}$. Note that $G^{\boxtimes m}$ is vertex-transitive. Then %
$G^{\boxtimes m} \leq E_{N_m} \boxtimes G[S]^{\boxtimes m}$ with $
N_m = \lceil |V(G)|^m |S|^{-m} m \ln |V(G)| \rceil$.
Apply $F$ to both sides of the inequality, take the $m$-th root on both sides, and let $m$ go to infinity to get the claim.
\end{proof}

Finally, we mention that, 
from either of \autoref{lem:vertex-transitive-alpha} or \autoref{lem:vertex-transitive-spectrum} follows the analogous inequality for the Shannon capacity:\footnote{The Shannon capacity $\Theta$ is not in $\X$ \cite{haemers1979some,alon1998shannon}, so this requires a short proof.}  %
Let $G$ be a vertex-transitive graph and $S \subseteq V(G)$. Then
$\Theta(G[S]) \leq \Theta(G) \leq |V(G)| \cdot |S|^{-1} \cdot %
\Theta(G[S]).$
Indeed, this follows from \autoref{lem:vertex-transitive-alpha} applied to powers of $G$. Alternatively, there is an $F \in \X$ such that $\Theta(G[S]) = F(G[S])$ by \autoref{th:duality-Theta}, and then applying  \autoref{lem:vertex-transitive-spectrum} gives the result.
For ease of reference, we combine the results of this section %
into the following lemma.

\begin{lemma}\label{lem:vertex-transitive-gen}
Let $G$ be a vertex-transitive graph and $S \subseteq V(G)$. Let $F \in \X \cup \{\alpha, \Theta\}$. Then
\[
F(G[S]) \leq F(G) \leq \frac{|V(G)|}{|S|} F(G[S]).
\]
\end{lemma}

\autoref{lem:vertex-transitive-spectrum}
lets us construct graphs $G$ and $G[S]$ that are close to each other in the %
asymptotic spectrum distance, as long as we have a handle on the ratio $|V(G)|/|S|$. The next goal will be to obtain Cauchy sequences using this method. For this we need to do more, as we will discuss in the next sections.

\section{Converging sequences of finite graphs}
\label{sec:frac}

In \autoref{sec:dist} we have defined the asymptotic spectrum distance on graphs, and discussed its basic properties, including characterizations of convergence and the motivating connection to the Shannon capacity. A central question is how to construct (non-trivial) converging sequences, or even just Cauchy sequences.

In this section we will give several kinds of  constructions of non-trivial converging sequences of graphs. For this we will use a natural family of vertex-transitive graphs which we call fraction graphs (and which are also called circular graphs).

First we will discuss the definition and basic properties of fraction graphs. An important component is a vertex removal strategy for constructing fraction graphs that are close in the asymptotic spectrum distance. 

Next we use the vertex removal strategy to construct non-trivial converging sequences of graphs. In particular, for any fraction graph we construct a non-trivial sequence of fraction graphs that converges to it. 
As a consequence of these constructions, we obtain a right-continuity property of all elements in the asymptotic spectrum (e.g.,~the Lovász theta function) on fraction graphs. 

Finally, we construct Cauchy sequences of fraction graphs that have no limit in the set of finite graphs. Thus the set of finite graphs is not complete in the asymptotic spectrum distance. In \autoref{sec:completion} we will study infinite graphs on the circle as the missing limit points.

\subsection{Fraction graph definition and basic properties}\label{sec:vertex-removal}\label{subsec:vertex-removal-strategy}

Fraction graphs have been studied since the work of Vince \cite{MR0968751}, who used them to study variations on the chromatic number. These graphs have been studied under many names, including circular graphs~\cite{SvenThesis, MR3906144}, cycle-powers~\cite{MR3016977}, and (the complement of) rational complete graphs~\cite{MR2089014} and circular complete graphs~\cite{MR2249284}.

\begin{definition}
For any $p,q \in \NN$ with $p/q \geq 2$ we let $E_{p/q}$ be the graph with vertex set $V = \{0,\dots,p-1\}$ and with distinct $u,v \in V$ forming an edge if $u-v \pmod p < q$ or $v-u \pmod p < q$, that is, if their distance modulo $p$ is strictly less than $q$. The graphs $E_{p/q}$ we call \emph{fraction graphs}.
It will sometimes be convenient to allow also $p,q \in \NN$ with $p/q < 2$, in which case we will use the (natural) convention that $E_{p/q}$ is the complete graph on $p$ vertices.
Also, it will sometimes be convenient to identify the vertex set $V = \{0,\dots,p-1\}$ with $\ZZ_p = \ZZ/p\ZZ$, the integers modulo $p$.\footnote{The graph $E_{p/q}$ is the Cayley graph of the group $\ZZ_p$ with generator set $\{\pm 1, \dots, \pm (q-1)\}$. We recall that we do not require $p$ to be prime.}
\end{definition}
For example, $E_{n/1} = E_n$ and $E_{n/2} = C_n$, where $E_n = \overline{K}_n$ denotes the ``empty graph'' on $n$ vertices and $C_n$ denotes the cycle-graph on~$n$ vertices. For illustration, we draw the fraction graphs $E_{3/1}$, $E_{5/2}$ and~$E_{8/3}$ in \autoref{fig:frac}. %
\begin{figure}[H]
\begin{minipage}{4.5cm}
\centering
\begin{tikzpicture}
  \foreach \i in {0,1,2}
    \coordinate (v\i) at ({360/3 * (\i - 1)}:1.0cm);
  
  \foreach \i in {0,1,2}
    \filldraw[black] (v\i) circle (2pt);
\end{tikzpicture}
\end{minipage}
\begin{minipage}{4.5cm}
\centering
\begin{tikzpicture}
  \foreach \i in {1,2,...,5}
    \coordinate (v\i) at ({360/5 * (\i - 1)}:1.5cm);

  \foreach \i [evaluate=\i as \next using {int(mod(\i,5)+1)}] in {1,2,...,5}
    \draw (v\i) -- (v\next);

  \foreach \i in {1,2,...,5}
    \filldraw[black] (v\i) circle (2pt);
\end{tikzpicture}
\end{minipage}
\begin{minipage}{4.5cm}
\centering
\begin{tikzpicture}
  \foreach \i in {0,1,...,7}
    \coordinate (v\i) at ({360/8 * (\i - 1)}:1.5cm);

  \foreach \i [evaluate=\i as \next using {int(mod(\i+1,8))}] in {0,1,...,7}
    \draw (v\i) -- (v\next);
  \foreach \i [evaluate=\i as \next using {int(mod(\i+2,8))}] in {0,1,...,7}
    \draw (v\i) -- (v\next);

  \foreach \i in {0,1,...,7}
    \filldraw[black] (v\i) circle (2pt);
\end{tikzpicture}
\end{minipage}
\caption{From left to right, these are the fraction graphs $E_{3/1}$, $E_{5/2}$ and~$E_{8/3}$.}
\label{fig:frac}
\end{figure}
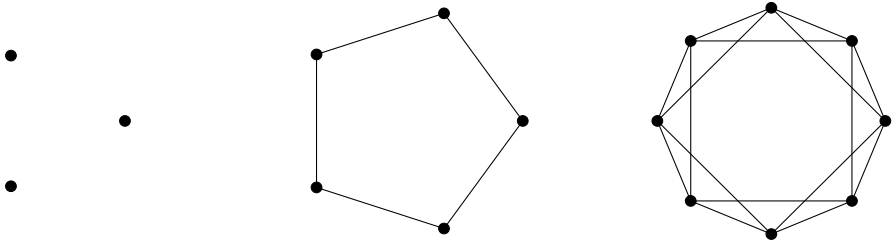

A fundamental property of the fraction graphs is that they are ordered under cohomomorphism as the rational numbers. %

\begin{lemma}[{\cite[Theorem 6.3]{MR2089014}}]\label{th:ordering}
For $p,q,p',q' \in  \NN$ with $p/q,p'/q'\geq2$, we have $p/q \leq p'/q'$ if and only if $E_{p/q} \leq E_{p'/q'}$.
\end{lemma}

In fact, a cohomomorphism $E_{p/q} \to E_{p'/q'}$ is simply $x \mapsto \lfloor x p'/p\rfloor$ (assuming that $p/q\leq p'/q'$).

In particular, it follows from \autoref{th:ordering} that if $p/q = p'/q'$, then $E_{p/q}$ and $E_{p'/q'}$ are equivalent under cohomomorphism, that is, $E_{p/q} \geq E_{p'/q'}$ and $E_{p/q} \leq E_{p'/q'}$. Also, we see that for any two inequivalent fraction graphs, there exists a fraction graph that is strictly in between them in the cohomomorphism order, making the fraction graphs a natural family for constructing interesting converging sequences of graphs. 

The ``label'' of a fraction graph can be computed using the fractional clique covering number $\overline{\chi}_f$.

\begin{lemma}[{\cite[Cor.~6.20,~6.24]{MR2089014}}]\label{th:fraction-graph-chi-bar}
For $p,q\in \NN$, $p/q \geq2$, we have $\overline{\chi}_f(E_{p/q}) = p/q$.
\end{lemma}

Since $\overline{\chi}_f$ is an element of the asymptotic spectrum $\X$, it follows from \autoref{th:fraction-graph-chi-bar} that in fact $E_{p/q} \leq E_{p'/q'}$ if and only if $E_{p/q} \asympleq E_{p'/q'}$, that is, cohomomorphism and asymptotic cohomomorphism coincide on fraction graphs.

\begin{remark}
We warn the reader that, while the fraction graphs are ordered (under cohomomorphism) as the rational numbers (\autoref{th:ordering}), in many other ways they do not behave as the rational numbers, in particular when it comes to how they behave under the strong product or the disjoint union (the arithmetic operations of our semiring). Indeed, it is not true in general that $E_{p/q} \boxtimes E_{s/t}$ is equivalent to $E_{(ps)/(qt)}$. For example, suppose $E_{5/2}^{\boxtimes 2}$ is equivalent to $E_{p/q}$, then $p/q = \overline{\chi}_f(E_{p/q}) = \overline{\chi}_f(E_{5/2})^2 = 25/4 \geq 6$, so $E_{5/2}^{\boxtimes 2} \geq E_6$. However, this contradicts the well-known fact that $\alpha(E_{5/2}^{\boxtimes 2}) = 5 < 6$, which follows, for instance, from $\vartheta(E_{5/2}^{\boxtimes 2}) = 5$ \cite{Lovasz79}, where $\vartheta$ is the Lovász theta function. %
\end{remark}

An important property of the fraction graphs (which we will use many times in our constructions) is that removing any vertex results in a graph that is again equivalent to a fraction graph, and moreover, there is a simple formula to determine which fraction graph.

\begin{lemma}[{\cite[Lemma~6.6]{MR2089014}}]\label{th:remove-point}
Let $p,q \in \NN$ be relatively prime and $p/q \geq2$. Let $p'$ and $q'$ be the unique integers such that $0 < p' < p$, $0 < q' < q$ and $pq' - qp' = 1$.
Then the graph obtained by removing any vertex from $E_{p/q}$ is equivalent to $E_{p'/q'}$.
\end{lemma}

For example, \autoref{th:remove-point} says that the induced subgraph of $E_{8/3}$ obtained by removing any vertex is equivalent to $E_{5/2}$ (which we recall is $C_5$), since $8\cdot 2 - 3 \cdot 5 = 1$.

We note that \autoref{th:remove-point} generalizes considerably: every induced subgraph of a fraction graph is equivalent to a fraction graph \cite[Theorem 3.1]{ZhuCirculantPerfect}, \cite[Corollary~3.1]{MR1905134}, \cite[Page 518]{MR2249284}.\footnote{There is a $\mathcal{O}(|V|^2)$ time algorithm for determining which fraction graph this is \cite[Theorem~3.1]{MR1905134}.} A more general version of this will play a role later in the paper (\autoref{lem: finite_induce}).

We now use \autoref{lem:vertex-transitive-gen} and \autoref{th:remove-point} to get our vertex removal strategy. %

\begin{theorem}\label{th:vrm}
    Let $F \in \X\cup\{\alpha, \Theta\}$. Let $0< p' < p$ and $0 < q' < q$ be integers such that $p/q, p'/q' \geq 2$ and $pq' - qp' = 1$. Then 
    \[
    F(E_{p'/q'}) \leq F(E_{p/q}) \leq \frac{p}{p-1} F(E_{p'/q'}).
    \]
\end{theorem}
\begin{proof}
    Let $G = E_{p/q}$. 
    Let $S$ be $V(E_{p/q})$ minus any vertex, so that $|V(G)| = p$ and $|S| = p-1$. 
    The numbers $p,q,p',q'$ are chosen such that $G[S]$ is equivalent to $E_{p'/q'}$ by \autoref{th:remove-point}. By \autoref{lem:vertex-transitive-gen} we have 
    $
    F(G[S]) \leq F(G) \leq |V(G)| \cdot {|S|}^{-1} \cdot F(G[S]),
    $
    which gives the claim.
\end{proof}
\autoref{th:vrm} lets us construct pairs of graphs with an upper bound on their asymptotic spectrum distance. For example, using $8\cdot 2 - 3 \cdot 5 = 1$ as before, we can prove that $d(E_{8/3}, E_{5/2}) \leq \tfrac5{14}$. Indeed, for any $F \in \X$ we have that 
\[
0 \leq F(E_{8/3}) - F(E_{5/2}) \leq \tfrac87 F(E_{5/2}) - F(E_{5/2}) = \tfrac17 F(E_{5/2})
\]
by \autoref{th:vrm}, and $F(E_{5/2}) \leq \overline{\chi}_f(E_{5/2}) = 5/2$ by \autoref{th:chi-bar-max} and \autoref{th:fraction-graph-chi-bar}. Therefore we find that $d(E_{8/3}, E_{5/2}) = \sup_{F \in \X} |F(E_{8/3}) - F(E_{5/2})| \leq \tfrac17 \cdot \tfrac 52 = \tfrac5{14}$.

Our next goal is to not just construct pairs of graphs with an upper bound on their distance, but to construct an infinite sequence of graphs that converges, which we will do in the next section.

\subsection{Convergence to any fraction graph from above}\label{sec:right-continuity}

In this section we will use 
the vertex removal strategy of \autoref{subsec:vertex-removal-strategy} to construct non-trivial converging sequences of fraction graphs.
More precisely, for any fraction graph~$E_{a/b}$ we will construct a non-trivial sequence of fraction graphs $E_{p_n/q_n}$ converging to~$E_{a/b}$. The convergence will be ``from above'' in the sense that the numbers $p_n/q_n$ will satisfy~$a/b \leq p_n/q_n$ and will be decreasing in $n$, and thus $E_{p_n/q_n}$ approaches $E_{a/b}$ from above in the cohomomorphism order. It will naturally follow that for any sequence~$p_n/q_n$ converging to $a/b$ from above, we have that $E_{p_n/q_n}$ converges to $E_{a/b}$. 

In terms of the asymptotic spectrum $\X$, our result says that for any $F \in \X$ (and in fact for any $F \in \X\cup \{\alpha, \Theta\}$), the function $\QQ_{\geq2} \to \RR : p/q \mapsto F(E_{p/q})$ is right-continuous. So for instance, the Lovász theta function is right-continuous on fraction graphs.

Our result solves a problem of Schrijver and Polak \cite[Chapter~9]{SvenThesis}. 
They constructed for every integer $m \in \NN$ a sequence of fraction graphs~$E_{p_n/q_n}$ that converges to $E_{m}$ from below, which implies that for any sequence $p_n/q_n$ converging to $m$ from below we have that $E_{p_n/q_n}$ converges to $E_m$. In terms of the asymptotic spectrum $\X$, their result says that for any $F \in \X$ the function $p/q \mapsto F(E_{p/q})$ is left-continuous at every integer.

\begin{theorem}\label{th:rational-right-cont}\hfil
\begin{enumerate}[\upshape(a)]
    \item\label{item:rational-right-cont:a} Let $F \in \X \cup \{\alpha,\Theta\}$.
The function $\mathbb{Q}_{\geq2} \to \mathbb{R}: p/q \mapsto F(E_{p/q})$ is right-continuous. 
    \item 
    The right-continuity in \upshape\eqref{item:rational-right-cont:a} is uniform in $\X$: for each $p'/q' \in \QQ_{\geq 2}$ and each $\eps > 0$ there is a $\delta > 0$ such that
    \[
        F(E_{p'/q'}) \le F(E_{p/q}) < F(E_{p'/q'}) + \eps
    \]
    for all $F \in \X$ and all $p/q \in \QQ_{\geq 2}$ with $p'/q' \le p/q < p'/q' + \delta$.
    \item If $p_n/q_n$ converges to $a/b$ from above, then $E_{p_n/q_n}$ converges to $E_{a/b}$.
\end{enumerate}
\end{theorem}

\begin{proof}
    (a) Choose a fixed coprime pair $p',q'$, we will show that $F$ is right-continuous at $p'/q'$.
    The function $p/q \to F(E_{p/q})$ is monotone and thus it is enough to find an explicit sequence $(p_n/q_n)_{n \geq 1}$ approaching $p'/q'$ from above for which $F(E_{p_n/q_n})$ converges to $F(E_{p'/q'})$.
    
    Because $p'$ and $q'$ are coprime, there exist positive integers $a, b$ such that 
    $aq' - bp' = 1$. Define the sequences $p_n = a + p'\cdot n$ and $q_n = b + q' \cdot n$. For every $n \geq 1$ we find that $0 < p' < p_n$, $0 < q' < q_n$ and $p_n q' - q_n p' = 1$, therefore we can apply 
    \autoref{lem:vertex-transitive-gen}
    to obtain
    \[
    F(E_{p'/q'}) \leq F(E_{p_n/q_n}) \leq \frac{p_n}{p_n-1} F(E_{p'/q'}).
    \]
    Since $p_n/(p_n-1)$ converges to $1$ when $n$ goes to infinity, we find that $F(E_{p_n/q_n})$ converges to $F(E_{p'/q'})$. 

    (b) The uniform right-continuity follows from applying Dini's theorem, \autoref{th:dini}. We give a direct proof for convenience.  %
    Let $\eps > 0$. Choose an $n$ such that $\tfrac{p'}{p_n - 1} < \eps$. Now let $\delta = p_n / q_n - p'/q'$. 
    If $0 \leq p/q - p'/q' < \delta$, then $p'/q' \leq p/q \leq p_n/q_n$, and thus $F(E_{p/q}) - F(E_{p'/q'}) \leq \tfrac{1}{p_n-1} F(E_{p'/q'}) \leq \tfrac{1}{p_n-1} \overline{\chi}(E_{p'/q'}) = \tfrac{p'}{p_n - 1} < \eps$.

    (c) This follows directly from (b).
\end{proof}

\begin{remark}
It follows from \autoref{th:rational-right-cont} that for every graph $G$, we have that 
the function $\mathbb{Q}_{\geq2} \to \mathbb{R}: p/q \mapsto F(E_{p/q} \boxtimes G)$ is right-continuous. For every choice of $G$ this right-continuity is uniform over all $F \in \X$.    
\end{remark}

\begin{remark}
\autoref{th:rational-right-cont} in particular implies that $p/q \mapsto F(E_{p/q})$ is right-continuous at integral $p/q$, and accordingly that if $p_n/q_n$ converges from above to an integer $m$, then $E_{p_n/q_n}$ converges to $E_m$. This special case has a much simpler proof. Indeed, %
if $p_n/q_n$ converges to $m$ from above, then $\max_{F \in \X} F(E_{p_n/q_n}) \leq \overline{\chi}_f(E_{p_n/q_n}) = p_n/q_n \to m$. On the other hand, for every $F \in \X$ we have $F(E_m) = m$. Thus $d(E_{p_n/q_n}, E_m) \to 0$.
\end{remark}

\begin{remark}
 In the context of \autoref{th:rational-right-cont}, we note that if $p_n/q_n$ is a strictly decreasing sequence converging to~$a/b$, then $E_{p_n/q_n}$ will also be strictly decreasing in the cohomomorphism order, and in fact in the asymptotic cohomomorphism order, since the asymptotic clique covering number is strictly decreasing. However, we do not know whether every $F \in \X$ (or the Shannon capacity $\Theta$) is strictly decreasing on the~$(E_{p_n/q_n})_{n \in \NN}$. By a result of Bohman and Holzman \cite{MR1967195} (recently improved by Zhu \cite{zhu2024improved}), we do know that if $a/b=2$, then there is a sequence $p_n/q_n$ converging to $a/b$ from above for which~$\Theta(E_{p_n/q_n})$ is strictly decreasing. %
\end{remark}

\begin{remark}
\autoref{th:rational-right-cont} leaves open whether $\QQ_{\geq2} \to \RR : p/q \mapsto F(E_{p/q})$ is left-continuous at the non-integer points of $\QQ_{\geq2}$ (left-continuity at integer points being proven in \cite[Chapter~9]{SvenThesis}, as mentioned earlier). In other words, if $p_n/q_n$ converges to a non-integer $a/b$ from below, then we do not know whether $E_{p_n/q_n}$ converges to~$E_{a/b}$ (and we even do not know convergence for concrete elements in the asymptotic spectrum like the Lovász theta function). This question turns out to be equivalent to a question about infinite graphs on the circle, which we will discuss in \autoref{sec:completion}.   
\end{remark}

\subsection{Intermezzo: Continued fraction expansion of irrational numbers}\label{subsec:cont-frac}

To prepare for our  next construction of converging sequences of fraction graphs in \autoref{sec:non-complete}, and for the convenience of the reader, we recall here basic facts about the continued fraction expansion of irrational numbers. We will use these facts again later in \autoref{sec: self-cohomomorphisms}.

Given an irrational number $r$ define the sequence $(r_n)_{n \geq 0}$ recursively by $r_0 = r$ and $r_{n+1} = (r_{n}-\lfloor r_{n}\rfloor)^{-1}$. The sequence $(a_n)_{n \geq 0}$ defined by $a_n = \lfloor r_n\rfloor$ describes the coefficients of the continued fraction expansion of $r$. The sequence of \emph{convergents} $(p_n/q_n)_{n \geq 0}$ is defined by $(p_0,q_0) = (a_0,1)$, $(p_1,q_1) = (a_0a_1 +1,a_1)$, and 
\[
    (p_n,q_n) = (a_n p_{n-1} + p_{n-2}, a_nq_{n-1} + q_{n-2})
\]
for $n \geq 2$. We call $p_n/q_n$ the \emph{$n$th-order convergent} of $r$. We recall the following facts concerning convergents.

\begin{lemma}  
\label{lem: CFA}
Let $r \in \mathbb{R}$ be irrational and let $(p_n/q_n)_{n \geq 0}$ be its sequence of convergents. %
\begin{enumerate}[\upshape(a)]
\item 
\label{it: CFA1}
We have 
\[
\lim_{n \to \infty} \frac{p_n}{q_n} = r.
\]
    \item
    \label{it: CFA2}
    For all $n \geq 1$,
    \[
        q_n p_{n-1} - p_n q_{n-1} = (-1)^n.
    \]
    \item 
    \label{it: CFA3}
    The even-order convergents form an increasing sequence, while the odd-order convergents form a decreasing sequence, that is,
    \[
        \frac{p_0}{q_0} < \frac{p_2}{q_2} < \frac{p_4}{q_4} < \cdots < r < \cdots < \frac{p_5}{q_5} < \frac{p_3}{q_3} < \frac{p_1}{q_1}.
    \]
    \item 
    \label{it: CFA4}
    The convergents are best possible approximations to $r$ in the following sense: for every $n\geq 0$ and integers $a,b$ with $|a| \leq |p_n|$,
    \[
        \left|r - \frac{p_n}{q_n}\right| \leq \left|r - \frac{a}{b}\right|.
    \]
\end{enumerate}
\end{lemma}
\begin{proof}
Proofs for these statements can be found in any elementary text on continued fractions. For example, these four items are respectively proved in Section~5, Theorem~2, Theorem~4, and Theorem~{17} of \cite{ContFrac}.
\end{proof}

\begin{corollary}
\label{cor: frac_approx}
Let $r \in \mathbb{R}_{> 1}$ be irrational and let $(p_n/q_n)_{n \geq 0}$ be its sequence of convergents. Then for all $n\geq 0$ we have
\(
    \lceil p_{2n}/r\rceil = q_{2n}
    \text{ and }
    \lfloor p_{2n+1}/r\rfloor = q_{2n+1}.
\)
\end{corollary}
\begin{proof}
    Because $r > 1$ we have that $p_m$ and $q_m$ are strictly positive for all $m$.
    Because $p_{2n}/q_{2n}$ is the best approximation to $r$ among fractions with equal or smaller numerator we see that $q_{2n}$ is the unique positive integer $q$ for which $(q-1) \cdot p_{2n} < r < q \cdot p_{2n}$. This equation is satisfied for $q = \lceil p_{2n}/r\rceil$ and thus $\lceil p_{2n}/r\rceil=q_{2n}$. The proof of the other equality goes analogously. 
\end{proof}

\subsection{Cauchy sequences with no finite limit graph}\label{sec:non-complete}

In \autoref{sec:right-continuity} we used the vertex-removal strategy to prove that if $p_n/q_n$ converges to~$a/b$ from above, then $E_{p_n/q_n}$ converges to $E_{a/b}$.
In this section we give another construction of interesting converging sequences, or rather of Cauchy sequences.

We will prove that if $p_n/q_n$ converges to any irrational number $r \in \RR_{\geq2}$ (without requiring this convergence to be ``from above'' or ``from below''), then $E_{p_n/q_n}$ is a Cauchy sequence.
For this result we will make use of similar ideas as in \autoref{sec:right-continuity} and of the continued fraction expansion of $r$ (\autoref{subsec:cont-frac}).
It will follow from a simple argument (rationality of the fractional clique covering number of finite graphs) that this sequence cannot have a finite graph as a limit point. Thus we find that the set of finite graphs is not complete.
In \autoref{sec:completion} we will discuss how the limit points of these Cauchy sequences of fraction graphs can be described by infinite graphs on the circle.

In terms of the asymptotic spectrum $\X$, our result says that for any element $F \in \X$ (and in fact any $F \in \X\cup\{\alpha, \Theta\}$), the map $p/q \mapsto F(E_{p/q})$ has an extension to $\RR_{\geq2}$ that is continuous at every irrational number. %

\begin{theorem}
\label{thm: continuous at irrationals}
Let $r \in \RR_{\geq2}$ be irrational. 
\begin{enumerate}[\upshape(a)]
    \item\label{item:irr:a} For every $F \in \X \cup \{\alpha,\Theta\}$,
    \begin{equation}
        \label{eq: desired equality}
        \sup_{a/b < r} F(E_{a/b}) = \inf_{a/b > r} F(E_{a/b}).
    \end{equation}
    In particular, $\QQ_{\geq2} \to \RR : p/q \mapsto F(E_{p/q})$ has an extension $\tilde{F}$ to $\RR_{\geq2}$ that is 
    continuous at every irrational number, namely by letting $\tilde{F}(r)= \sup_{a/b < r} F(E_{a/b})$.
    \item \label{item:irr:b} The continuity in {\upshape\eqref{item:irr:a}} is uniform over all $F \in \X$ in the following sense. Given any $\eps > 0$, there is a $\delta > 0$ such that if $|r - p/q| < \delta$, then $|\sup_{a/b<r}F(E_{a/b}) - F(E_{p/q})| < \eps$ for every $F\in \X$.
    \item \label{item:irr:c} If $a_n/b_n$ converges to $r$, then $E_{a_n/b_n}$ is Cauchy.
\end{enumerate}
\end{theorem}
\begin{proof}
    \eqref{item:irr:a} Let $(p_n/q_n)_{n \geq 1}$ and $(p_n'/q_n')_{n \geq 1}$ be the odd-order and even-order convergents of $r$ respectively. It follows from \autoref{lem: CFA} that (i) $p_n/q_n$ and $p_n'/q_n'$ converge to $r$ from above and below, respectively, (ii) for all $n$ we have $p_n q_n' - p_n' q_n  = 1$, and (iii) $p_n \to \infty$ as $n \to \infty$. By \autoref{lem:vertex-transitive-gen} and the fact that $a/b \mapsto F(E_{a/b})$ is nondecreasing we have 
    \[
        F(E_{p_n'/q_n'}) \leq \sup_{a/b < r} F(E_{a/b}) \leq \inf_{a/b > r} F(E_{a/b}) \leq F(E_{p_n/q_n}) \leq \frac{p_n}{p_n - 1}F(E_{p_n'/q_n'}).
    \]
    Thus both the left-hand side and the right-hand side of \eqref{eq: desired equality} are contained in an interval of length $\frac{1}{p_n-1} F(E_{p_n'/q_n'})$ which is upper bounded by $\frac{1}{p_n - 1} \overline{\chi}_f(E_{p_n'/q_n'}) < \frac{r}{p_n-1}$. This bound does not depend on~$F$ and goes to $0$ as $n$ goes to infinity. This proves 
    \autoref{eq: desired equality} and thus $\tilde{F}$ as defined in \eqref{item:irr:a} is continuous at $r$.
    
    \eqref{item:irr:b} To prove uniform continuity, let $\eps >0$. Choose an $n$ such that $\frac{r}{p_n-1} < \eps$. Now take any $\delta > 0$ such that
    \[
    \delta < \min\{p_n/q_n-r,r-p_n'/q_n'\}.
    \]
    Then if $|r-x| < \delta$ we have $x \in [p_n'/q_n',p_n/q_n]$ and thus $|\tilde{F}(r) - \tilde{F}(x)| \leq \frac{r}{p_n-1} < \eps$.

    \eqref{item:irr:c} Let $a_n/b_n$ be a sequence of rational numbers converging to~$r$. Then we have 
    \begin{align*}
        d(E_{a_n/b_n},E_{a_m/b_m}) &= \sup_{F \in \X} |F(E_{a_n/b_n}) - F(E_{a_m/b_m})|\\ 
        &\leq \sup_{F \in \X} |\tilde{F}(a_n/b_n) - \tilde{F}(r)| + \sup_{F \in \X} |\tilde{F}(a_m/b_m) - \tilde{F}(r)|,
    \end{align*}
    where $\tilde{F}$ is defined as above. Since both summands converge to $0$ we see that $E_{a_n/b_n}$ is indeed Cauchy.
\end{proof}

\begin{remark}
    Unlike for \autoref{th:rational-right-cont} (b), 
    the uniform continuity of \autoref{thm: continuous at irrationals} (b) does not directly follow using Dini's theorem (\autoref{th:dini}). However, if we first extend the semiring of graphs to \emph{infinite} graphs (as we will do in the next section), then there is such a natural proof.
\end{remark}

As a consequence of \autoref{thm: continuous at irrationals} we find that the set of finite graphs with the asymptotic spectrum distance is not complete.

\begin{corollary}\label{th:not-complete2}
There is a Cauchy sequence of fraction graphs $E_{a_n/b_n}$ that does not converge to any finite graph.
\end{corollary}
\begin{proof}
    Let $r> 2$ be an irrational number and let $a_n/b_n$ be a sequence of rational numbers converging to~$r$. 
    Then $E_{a_n/b_n}$ is a Cauchy sequence (\autoref{thm: continuous at irrationals} (c)). 
    Note that the fractional clique covering number of $E_{a_n/b_n}$ is given by $\overline{\chi}_f(E_{a_n/b_n}) = a_n/b_n$. Thus any limit graph $G$ of the sequence $E_{a_n/b_n}$ should satisfy $\overline{\chi}_f(G) = r$. This contradicts the fact that $\overline{\chi}_f(G)$ is rational for any finite graph $G$. 
\end{proof}

\section{Infinite graphs as limit points}\label{sec:completion}

In \autoref{sec:frac} we constructed Cauchy sequences of finite graphs in two ways: non-trivial converging sequences of fraction graphs to any fraction graph, and Cauchy sequences of fraction graphs with no finite graph as limit point. What are these limit points? In an abstract sense we know that we can describe them as elements of the space $C(\X)$ of continuous functions on the asymptotic spectrum $\X$, since this space is complete (\autoref{sec:metric}). In this section we will see how infinite graphs give a concrete model for these limit points.

The protagonists of this section are a family of infinite graphs on the circle, which we may think of as ``infinite versions'' of the fraction graphs of \autoref{sec:frac}. These infinite graphs are labeled by real numbers~$r \in \RR_{\geq2}$ instead of rational numbers~$p/q \in \QQ_{\geq2}$.
We present two versions of these graphs: an ``open'' version and a ``closed'' version, depending on whether a strict or non-strict inequality is used in their definition (which we will state in a moment). Understanding this seemingly minor difference will turn out to play an important role in the theory (in this section, and in \autoref{sec: self-cohomomorphisms}). %

We will first define the two versions of circle graphs and prove their basic properties. %
Next, we prove that for any irrational number $r \in \RR_{\geq2}$, if $p_n/q_n$ converges to~$r$, then the Cauchy sequence $E_{p_n/q_n}$ converges to a circle graph. Finally, we prove that the problem of convergence of fraction graphs ``from below'', or left-continuity of the asymptotic spectrum, is equivalent to the asymptotic equivalence of open and closed circle graphs.

In \autoref{sec: self-cohomomorphisms} we will study the circle graphs further to prove that the open and closed version are not equivalent under cohomomorphism, characterizing their self-cohomomorphisms in the process.

\subsection{Circle graph definition and basic properties}

Circle graphs are a natural family of infinite graphs. They have been studied in the context of circular colorings by Zhu \cite{MR1189048} (implicitly) and Bauslaugh--Zhu \cite{MR1642525}. See also, e.g., \cite[Section~17.2.3]{MR2361455}. They are related to Borsuk graphs \cite{MR0227050,MR0514625}. The definition and notation we use is as follows.

\begin{definition}
    Let $C \subseteq \RR^2$ denote a circle with unit circumference.
    Let $r \in \RR_{\geq2}$. Let~$E_r^{\open}$ be the infinite graph with vertex set $C$ for which two vertices are adjacent if and only if they have distance strictly less than $1/r$ on~$C$.
    Let $E_r^{\closed}$ be the infinite graph with vertex set $C$ for which two vertices are adjacent if and only if they have distance at most~$1/r$ on~$C$. 
\end{definition}

For illustration, we draw in \autoref{fig:circ} the neighbourhood of a vertex in the circle graphs $E^\open_3$, $E^\closed_3$ and $E^\closed_{5/2}$. The neighbourhood in $E^\open_3$ is open, while the neighbourhood in $E^\closed_3$ and $E^\closed_{5/2}$ is closed.
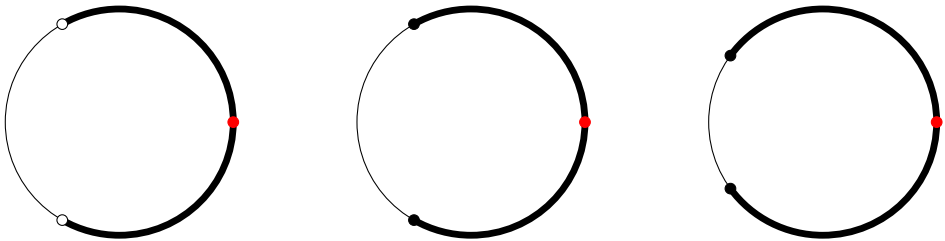
\begin{figure}[H]
\begin{minipage}{4.5cm}
\centering
\begin{tikzpicture}
  \def\r{1.5cm}
  \draw (0,0) circle (\r);

  \draw[line width=0.9mm] (0:\r) arc (0:120:\r);
  \draw[line width=0.9mm] (0:\r) arc (0:-120:\r);

  \filldraw[red] (0:\r) circle (2pt);
  \draw[fill=white] (120:\r) circle (2pt);
  \draw[fill=white] (240:\r) circle (2pt);

\end{tikzpicture}
\end{minipage}
\begin{minipage}{4.5cm}
\centering
\begin{tikzpicture}
  \def\r{1.5cm}
  \draw (0,0) circle (\r);

  \draw[line width=0.9mm] (0:\r) arc (0:120:\r);
  \draw[line width=0.9mm] (0:\r) arc (0:-120:\r);

  \filldraw[red] (0:\r) circle (2pt);
  \filldraw (120:\r) circle (2pt);
  \filldraw (240:\r) circle (2pt);
\end{tikzpicture}
\end{minipage}
\begin{minipage}{4.5cm}
\centering
\begin{tikzpicture}
  \def\r{1.5cm}
  \draw (0,0) circle (\r);

  \draw[line width=0.9mm] (0:\r) arc (0:144:\r);
  \draw[line width=0.9mm] (0:\r) arc (0:-144:\r);

  \filldraw[red] (0:\r) circle (2pt);
  \filldraw (144:\r) circle (2pt);
  \filldraw (-144:\r) circle (2pt);
\end{tikzpicture}
\end{minipage}
\caption{The neighbourhood of a vertex (in red) in the circle graphs $E^\open_3$, $E^\closed_3$ and $E^\closed_{5/2}$ (from left to right).  We take the circles to have unit circumference.}
\label{fig:circ}
\end{figure}

It is not hard to see that, just like the fraction graphs, the graphs $E_r^\open$ are co\-homo\-mor\-phism-monotone with respect to $r$, and the same for the $E_r^\closed$. Moreover, for every fixed~$r$ the closed version $E_r^\closed$ is always at most the open version $E_r^\open$. Conversely, if~$r < s$, then the open version $E^\open_r$ is at most the closed version $E^\closed_s$.
We summarize these properties in the following lemma.

\begin{lemma}\label{lem:mon-inf}\label{lem:cl-lt-op}\label{lem:op-lt-cl}
    If $r,s \in \RR_{\geq2}$ with $r<s$ we have $E_r^\closed \leq E_r^{\open} \leq E_s^\closed \leq E_s^\open$.
\end{lemma}
\begin{proof}
    Each graph in this sequence of inequalities is a spanning subgraph of the previous graph and thus the identity map $f: C \to C$ is a cohomomorphism.
\end{proof}

The next basic property, which we will use several times, is about finite induced subgraphs of the closed circle graphs. It is a priori clear that for any subset $S \subseteq V(E_r^\closed)$ we have $E_r^\closed[S] \leq E_r^\closed$. For finite $S$, the following lemma bounds $E_r^\closed[S]$ further away from~$E_r^\closed$ by means of an open circle graph that is strictly below~$E_r^\closed$.

\begin{lemma}\label{lem:fin-cl-frac}
    Let $r \in \RR_{\geq2}$ and let $S \subseteq V(E_r^{\closed})$ be finite. Then there is an~$a/b < r$ such that $E_r^{\closed}[S] \leq E^\open_{a/b}$.
\end{lemma}
\begin{proof}
    There is an $s \in \RR_{\geq2}$ such that the smallest distance over all distinct and non-adjacent $\alpha,\beta \in S$ equals $1/s$. Then $E^\closed_r[S] \leq E^\open_s$. Since $s < r$, any $a/b$ such that $s < a/b < r$ satisfies the claim by \autoref{lem:mon-inf}.
\end{proof}

We call two (possibly infinite) graphs $G$ and $H$ 
\emph{equivalent} if $G \leq H$ and $H \leq G$.

\begin{lemma}\label{lem:open-circle-frac-equiv}
    Let $p/q \in \QQ_{\geq2}$. Then $E^\open_{p/q}$ and $E_{p/q}$ are equivalent.
\end{lemma}
\begin{proof}
    To prove $E_{p/q} \leq E_{p/q}^\open$, consider $C$ as $[0,1]$ where $0$ and $1$ are identified, and let $f : [p] \to C$ be defined by $v \mapsto v/p$. If $v,w \in [p]$ are not adjacent, then they have distance at least $q$ modulo $p$. Then $v/p$ and $w/p$ have distance at least $q/p$ on~$C$. Thus~$f$ is a cohomomorphism $E_{p/q} \to E_{p/q}^\open$.

    To prove $E_{p/q}^\open \leq E_{p/q}$, let $f : C \to [p]$ be defined by $\alpha \mapsto \lfloor p \alpha \rfloor$. One verifies that if~$\alpha$ and $\beta$ have distance at least $q/p$ on $C$, then $\lfloor p\alpha \rfloor$ and $\lfloor p \beta \rfloor$ have distance at least $q$ modulo $p$. Thus $f$ is a cohomomorphism $E^\open_{p/q} \to E_{p/q}$.
\end{proof}

We will see a generalization of \autoref{lem:open-circle-frac-equiv} in \autoref{lem: induce_equidistant}.

\subsection{Asymptotic spectrum distance of infinite graphs}\label{sec:asymp-spec-infinite}

In \autoref{sec:dist} we discussed asymptotic spectrum duality and distance for finite graphs. The theory extends naturally to infinite graphs, as was first noted by Fritz \cite{fritz2021unified}, and which we will need in this section. We discuss the extension to infinite graphs briefly here, and refer to the appendix for a more detailed discussion of the general theory of which finite and infinite graphs are a special case (\autoref{sec:gen-asymp-spec}).

Let $\mathcal{G}_\infty$ be the set of all (possibly) infinite graphs with finite clique covering number.\footnote{To avoid any issues with $\mathcal{G}_\infty$ not being a set, we will make the technical assumption that all our graphs have bounded cardinality. For concreteness, we may take this to be the cardinality of $\RR$, since this suffices for our examples.} %
Strong product $\boxtimes$, Shannon capacity $\Theta$, disjoint union $\sqcup$, cohomomorphism~$\rightarrow$, cohomomorphism preorder~$\leq$, asymptotic preorder~$\asympleq$, and asymptotic spectrum are defined for infinite graphs in the same way as for finite graphs. We will denote the asymptotic spectrum of infinite graphs by $\X_\infty$ and keep using $\X$ for the asymptotic spectrum of finite graphs. (This will allow us to compare the two, in a moment.) In~$\mathcal{G}_\infty$ we identify any two graphs that are asymptotically equivalent.

The asymptotic spectrum duality for infinite graphs is as follows:

\begin{theorem}[Asymptotic spectrum duality]
Let $G,H \in \G_\infty$.
\begin{enumerate}[\upshape(a)]
    \item $G \asympleq H$ if and only if for every $F \in \X_\infty$, $F(G) \leq F(H)$.
    \item $\Theta(G) = \min_{F \in \X_\infty} F(G)$.
    \item $\X_\infty$ is compact in the topology induced by the evaluation maps $\widehat{G} : F \mapsto F(G)$.
\end{enumerate}
\end{theorem}

\begin{remark}
    Note that, unlike for finite graphs (\autoref{th:duality-Theta}), for infinite graphs we do not know whether $\max_{F \in \X_\infty} F(G)$ is an element of $\X_\infty$. One can prove, with general arguments (see \autoref{th:abstract-duality}), that this maximization equals the asymptotic clique covering number $\lim_{n\to\infty} \overline{\chi}(G^{\boxtimes n})^{1/n}$.
\end{remark}

The other theory we discussed for finite graphs remains true for infinite graphs as well.  We define asymptotic spectrum distance on $\mathcal{G}_\infty$ by $d(G, H) = \sup_{F \in \X_\infty}|F(G) - F(H)|$, and the supremum is attained. We may equivalently think of elements of $\mathcal{G}_\infty$ as elements~$\widehat{G}$ of $C(\X_\infty)$, real-valued continuous functions on $\X_\infty$, equipped with the sup-norm. Shannon capacity is continuous. And finally, Dini's theorem implies that if $(G_n)_{n\in \NN}$ is a sequence in $\mathcal{G}_\infty$ and $H \in \mathcal{G}_\infty$, and if for every $F \in \X_\infty$ we have that $(F(G_n))_{n \in \NN}$ is monotone in $i \in \NN$ and converges to $F(H)$, then $G_n$ converges to $H$. (See \autoref{sec:gen-asymp-spec}.)

It is natural to ask how $\X$ and $\X_\infty$ are related. Indeed, there is a strong connection; to prove it we use the following basic lemma.

\begin{lemma}\label{lem:order-embedding}
    For any two finite graphs $G,H$, we have $G \leq H$ in the cohomomorphism preorder on finite graphs if and only if $G \leq H$ in the cohomomorphism preorder on infinite graphs.
\end{lemma}
\begin{proof}
 This follows directly from the definition of a cohomomorphism $G \to H$ (of graphs or infinite graphs) as a map from $V(G)$ to $V(H)$ that maps any pair of distinct non-adjacent vertices of $G$ to distinct non-adjacent vertices of $H$.
\end{proof}

Thus, \autoref{lem:order-embedding} says that the natural embedding from finite graphs to infinite graphs, $i : \mathcal{G} \to \mathcal{G}_\infty$ is, first of all, well-defined (under cohomomorphism equivalence), and second, is an order-embedding: for every $G,H \in \mathcal{G}$ we have $G \leq H$ if and only if $i(G) \leq i(H)$.
In other words, infinite graphs do not introduce any new inequalities between finite graphs in the cohomomorphism preorder.
This simple observation implies (as a special case of a general result, \autoref{th:abstract-surj}) the following relation between the asymptotic spectra $\X$ and $\X_\infty$. We consider the restriction map $i^* : \X_\infty \to \X : F \mapsto F \circ i$, which simply restricts the domain of any element $F \in \mathcal{X}_\infty$ to finite graphs.

\begin{theorem}\label{th:restr}
    The restriction map $i^* : \X_\infty \to \X$ is surjective.
\end{theorem}

\begin{remark}
    In other words, \autoref{th:restr} says that any element $F \in \X$ can be extended to an element in $\X_\infty$. So, for instance, the Lovász theta function (on finite graphs) can be extended to infinite graphs (perhaps in multiple ways) in such a way that it remains $\boxtimes$-multiplicative, $\sqcup$-additive, $\leq$-monotone and $E_n$-normalised. 
\end{remark}

\begin{remark}
    As a consequence of \autoref{th:restr},  $\sup_{F \in \X_\infty}|F(G) - F(H)|$ and $\sup_{F \in \X}|F(G) - F(H)|$ coincide on any finite graphs $G,H$.
\end{remark}

\autoref{th:restr} will play an important role in several arguments that involve both infinite and finite graphs.

\subsection{Circle graphs as limit points of fraction graphs} %

In \autoref{sec:non-complete} we proved that if $p_n/q_n$ converges to an irrational $r \in \RR_{\geq2}$, then $E_{p_n/q_n}$ is Cauchy. 
We noted that no finite graph can be the limit point of these Cauchy sequences (\autoref{th:not-complete2}). We will prove here that the limit points are circle graphs.

\begin{theorem}\label{th:irr-closed-open-equiv}
Let $r \in \RR_{\geq2}$ be irrational.
\begin{enumerate}[\upshape(a)]
    \item For every $F \in \X_\infty$,
    \[
        F(E_r^\closed) = F(E_r^\open) = \sup_{a/b < r} F(E_{a/b}) = \inf_{r < c/d} F(E_{c/d}).
    \]
    In particular, $E_r^\closed$ and $E_r^\open$ are asymptotically equivalent.
    \item If $p_n/q_n$ converges to $r$, then $E_{p_n/q_n}$ converges to $E^\open_r$.
    \end{enumerate}
\end{theorem}
\begin{proof}
For any $a/b, c/d \in \QQ_{\geq2}$ with $a/b < r < c/d$ we have
\begin{equation}\label{cl-op-sandw}
E_{a/b} \leq E^\closed_r \leq E^\open_r \leq E_{c/d}.
\end{equation}
From \autoref{thm: continuous at irrationals} we know that for every $F \in \X$,
$\sup_{a/b < r} F(E_{a/b}) = \inf_{r < c/d} F(E_{c/d})$.
Thus applying any $F \in \X_\infty$ to \autoref{cl-op-sandw}, and taking the supremum over $a/b<r$ and the infimum over $r < c/d$, gives $F(E^\closed_r) = F(E^\open_r)$. Thus $E^\closed_r$ and $E^\open_r$ are asymptotically equivalent.
\end{proof}

\subsection{Characterizing convergence to fraction graphs from below}

In \autoref{sec:right-continuity} we proved that if $p_n/q_n$ converges to $a/b$ from above, then $E_{p_n/q_n}$ converges to $E_{a/b}$. We left open whether the same is true if $p_n/q_n$ converges to $a/b$ from below. Here we will prove that this question is equivalent to the question whether the ``closed'' and ``open'' version of the circle graphs are asymptotically equivalent.

\begin{theorem}\label{th:op-cl-left}
    Let $r = p/q \in \QQ_{>2}$ be rational. The following are equivalent:
    \begin{enumerate}[\upshape(i)]
        \item $E_r^\closed$ and $E_r^\open$ are asymptotically equivalent 
        \item There is a sequence $m_n = n + o(n)$ such that for every $n \in \NN$ there is a rational number $a_n/b_n < p/q$ such that $E_{p/q}^{\boxtimes n} \leq E_{a_n/b_n}^{\boxtimes m_n}$
        \item For every $\eps > 0$, there is a $\delta > 0$, such that for any $a/b\in \QQ_{\geq2}$, if $0 < p/q - a/b < \delta$, then for every $F \in \X$, $F(E_{p/q}) - F(E_{a/b}) < \eps$
        \item If $a_n/b_n \to p/q$ from below, then $E_{a_n/b_n} \to E_{p/q}$ %
        \item For every $F \in \X$, $\sup_{a/b < p/q} F(E_{a/b}) = F(E_{p/q})$
    \end{enumerate}
\end{theorem}

\begin{proof}%
    (i) $\Rightarrow$ (ii). Suppose $E_r^\closed$ and $E_r^\open$ are asymptotically equivalent. Then there is a sequence $m_n = n + o(n)$ such that $(E_r^\open)^{\boxtimes n} \leq (E_r^\closed)^{\boxtimes m_n}$. Then we have
    \[
    E_{p/q}^{\boxtimes n} \leq (E_r^\open)^{\boxtimes n} \leq (E_r^\closed)^{\boxtimes m_n}.
    \]
    Let $f$ be any cohomomorphism from $E_{p/q}^{\boxtimes n}$ to $(E_r^\closed)^{\boxtimes m_n}$. Then $f(E_{p/q}^{\boxtimes n})$ is a finite subgraph of $(E_r^\closed)^{\boxtimes m_n}$, so by \autoref{lem:fin-cl-frac} (applied to every factor) %
    there are some $a_n/b_n < r$ such that $E_{p/q}^{\boxtimes n} \leq E_{a_n/b_n}^{\boxtimes m_n}$. 
    
    (ii) $\Rightarrow$ (iii).    
    Write $m_n$ as $m_n = n(1+f(n))$ with $f(n) \in o(1)$.
    Applying any $F \in \X$ and taking $n$th roots we find that
    \[
    F(E_{a_n/b_n}) \leq F(E_{p/q}) \leq F(E_{a_n/b_n})^{1 + f(n)} 
    \]
    Letting $L_n = F(E_{a_n/b_n})$ and $R_n = F(E_{a_n/b_n})^{1 + f(n)} $ we have 
    \begin{align*}
    |L_n - R_n| &\leq F(E_{a_n/b_n})^{1 + f(n)} - F(E_{a_n/b_n})\\ 
    &= F(E_{a_n/b_n}) \bigl( F(E_{a_n/b_n})^{f(n)} - 1\bigr)\\
    &\leq (a_n/b_n)((a_n/b_n)^{f(n)} - 1)\\
    &\leq r(r^{f(n)} - 1).
    \end{align*}
    Let $\eps > 0$. Let $n$ be large enough such that $r(r^{f(n)} - 1) < \eps$. Then if $a_n/b_n < a/b < r$, then for every $F \in \X$, $F(E_{p/q}) - F(E_{a/b}) < \eps$.

    (iii) $\Rightarrow$ (iv). These two statements are equivalent by definition.
    
    (iv) $\Rightarrow$ (v). This is an immediate consequence.

    (v) $\Rightarrow$ (i).
    Suppose that $\sup_{a/b < r} F(E_{a/b}) = F(E_{p/q})$. For any $a/b<r$ we have by \autoref{lem:op-lt-cl} that 
    \[
    E^\open_{a/b} \leq E^\closed_r \leq E^\open_r.
    \]
    Applying any $F \in \X_\infty$ and taking the supremum over $a/b < r$, we find 
    \[
    F(E^\closed_r) = F(E^\open_r).
    \]
    This implies by the duality theorem that $E^{\closed}_r$ and $E^\open_r$ are asymptotically equivalent.
\end{proof}

\begin{remark}
    In particular, it follows from \autoref{th:op-cl-left} and the left-continuity of every $F \in \X$ on fraction graphs $E_{p/q}$ for integral $p/q \geq 3$ \cite{SvenThesis}, that for every integral $r \in \mathbb{Z}_{\geq3}$ the graphs $E^\closed_r$ and $E^\open_r$ are asymptotically equivalent.
\end{remark}

\begin{remark}
    Note that \autoref{th:op-cl-left} in particular says that convergence of $F(E_{a/b})$ to $F(E_{p/q})$ for all $F \in \X$ when $a/b \to p/q$ from below, implies uniform convergence over all $F \in \X$. This should not come as a surprise as this is implied by Dini's theorem (\autoref{th:dini}).
\end{remark}

It is natural to conjecture that the set of ``open'' circle graphs $\{E^\open_r: r\in \RR_{\geq2}\}$ is the closure (and thus the completion) of the set of fraction graphs $\{E_{p/q} : p/q \in \QQ_{\geq2}\}$. We cannot quite prove this, but we do know the following partial results from the previous sections.
\begin{itemize}
\item \autoref{th:rational-right-cont} implies  that if $r_i\in \RR_{\geq2}$ converges to $p/q$ from above, then~$E^\open_{r_i}$ converges to $E_{p/q}$. 
\item \autoref{th:irr-closed-open-equiv} implies that if $r_i\in \RR_{\geq2}$ converges to an irrational number~$r \in \RR_{\geq2}$, then $E^\open_{r_i}$ converges to $E^\open_r$ (which is asymptotically equivalent to $E^\closed_r$). %
\item \autoref{th:op-cl-left} implies that, assuming $E^\open_{p/q}$ and $E^\closed_{p/q}$ are asymptotically equivalent, if $r_i \in \RR_{\geq2}$ converges to $p/q$ from below, then $E^\open_{r_i}$ converges to $E_{p/q}$.
\end{itemize}
Combining these statements, we obtain the following implication.

\begin{theorem}
Suppose that for every $r \in \QQ_{>2}$ the graphs $E^\closed_r$ and $E^\open_r$ are asymptotically equivalent.     
Then $\{E_r^\open : r\in \RR_{\geq2}\}$ is the closure of $\{E_{p/q} : p/q \in \QQ_{\geq2}\}$. %
\end{theorem}

\begin{proof}
Let $E^\open_{r_i}$ be a converging sequence. Then in particular, letting $F \in \X_\infty$ be any extension of the fractional clique covering number, we have that $F(E^\open_{r_i}) = r_i$ converges. Let $r \in \RR_{\geq2}$ denote the limit. We claim that $E^\open_{r_i}$ converges to $E^\open_r$.

Suppose $r$ is irrational. Then the claim follows from \autoref{th:irr-closed-open-equiv}.
Suppose $r$ is rational. 
We may assume that the sequence $(r_i)_i$ is monotone.
Suppose our sequence is converging from above. Then the claim follows from \autoref{th:rational-right-cont}.
Suppose our sequence is converging from below. Then the claim follows from \autoref{th:op-cl-left}.
\end{proof}

\section{Non-equivalence of open and closed circle graphs}
\label{sec: self-cohomomorphisms}
In \autoref{sec:completion} we proved that for irrational $r$ the circle graphs $E_r^\open$ and $E_{r}^\closed$ are asymptotically equivalent (\autoref{th:irr-closed-open-equiv}). We also proved that for rational $r$ the circle graphs $E_r^\open$ and~$E_r^\closed$ are asymptotically equivalent if and only if the fraction graphs are ``left-continuous'' at $r$ %
(\autoref{th:op-cl-left}). Whether $E_r^\open$ and~$E_r^\closed$ are indeed asymptotically equivalent for rational~$r$ remains an open problem.

The above raises the natural question whether the open and closed circle graphs are equivalent. Equivalence, in particular, would imply asymptotic equivalence.
We will prove that the circle graphs $E_r^\open$ and $E_{r}^\closed$ are not equivalent for any $r$ (rational or irrational). Thus, any resolution of the asymptotic equivalence problem of rational open and closed circle graphs must be asymptotic in nature (or at least make use of powers of the graphs). %

Our proof of non-equivalence relies on characterizing all cohomomorphisms $E_{r}^\closed\to E_{r}^\closed$ and $E_r^\open\to E_r^\open$. We prove that these self-cohomomorphisms are precisely rotations composed with reflections. Before we develop these ideas for the circle graphs, we first prove the analogous statement for the (finite) fraction graphs $E_{p/q}$, which we will then use as a proof ingredient to deal with the circle graphs.

\subsection{Self-cohomomorphisms of fraction graphs}
We call any cohomomorphism $G \to G$ a \emph{self-cohomomorphism}.
We will prove that for the fraction graphs, these are precisely rotations composed with reflections. We begin by showing that the self-cohomomorphisms of fraction graphs, for reduced fractions, are isomorphisms. (In the language of ``cores'' \cite[Section 1.6]{MR2089014}, this says that the complement of a fraction graph, for reduced fractions, is a core.)
\begin{lemma}[{\cite[Lemma 2.5]{MR1905134}}]\label{rem: core}
   Let $p,q\in \NN, p/q\geq2$ such that $\gcd(p,q)=1$. Then any cohomomorphism from $E_{p/q}$ to itself is an isomorphism.
\end{lemma}
\begin{proof}
    We give a brief proof for the convenience of the reader.
    This is obvious when $p/q = 2$. Suppose $p/q>2$. Then there are (unique) integers $1\leq p'\leq p,1\leq q'\leq q$ with $p'/q' \geq 2$ satisfying $pq'-qp'=1$. A cohomomorphism from $E_{p/q}$ to a proper induced subgraph would induce a cohomomorphism from $E_{p/q}$ to $E_{p'/q'}$ by \autoref{th:remove-point} which is precluded by comparing $\overline{\chi}_f$ (\autoref{th:ordering}, \autoref{th:fraction-graph-chi-bar}). Alternatively, the result can be proven using the No-Homomorphism Lemma; see \cite[Lemma 2.5]{MR1905134}.
\end{proof}

Having \autoref{rem: core}, it remains to show that these isomorphisms are rotations composed with reflections.

\begin{lemma}
    \label{lem: self_co_imp_auto_finite}
    Let $p/q \in \QQ_{\geq2}$, $q \geq 2$ such that $\gcd(p,q)=1$. Let $f: \mathbb{Z}_{p} \to \mathbb{Z}_{p}$ be a cohomomorphism from $E_{p/q}$ to itself. Then $f$ is an isomorphism of the form $f(x) = a - x$ or $f(x) = a + x$ for some $a \in \mathbb{Z}_p$.
\end{lemma}
\begin{proof}
    By \autoref{rem: core}, $f$ is an isomorphism.
    The maximum cliques of $E_{p/q}$ are of the form $C_x = \{x,x+1,\dots, x+(q-1)\}$ for $x \in \mathbb{Z}_p$ and thus there is a bijection $g: \mathbb{Z}_p \to \mathbb{Z}_p$ satisfying $f(C_x) = C_{g(x)}$ for all~$x$. The map $g$ must be an automorphism of the auxiliary graph $H$ on vertex set $\mathbb{Z}_{p}$ connecting $x$ and $y$ if and only if the symmetric difference of~$C_x$ and $C_y$ has size two. Observe that the fact that $q \geq 2$ implies that the graph $H$ is isomorphic to a $p$-cycle. It follows that $g(x+1)-g(x) \in \{\pm 1\}$ for all $x$. Moreover, because $g$ is a bijection, $g(x+1)-g(x)$ must actually constantly be either $1$ or $-1$. In the first case $g(x) = g(0) + x$ for all $x$ and in the latter case $g(x) = g(0)-x$ for all $x$. These cases correspond to $f(x) = g(0) + x$ and $f(x) = g(0) + q - 1 -x$ respectively.
\end{proof}

\subsection{Self-cohomomorphisms of circle graphs}

In this section we characterise the self-cohomomorphisms of irrational circle graphs. As an ingredient we use the following basic property about finite induced subgraphs of circle graphs. %

\begin{lemma}
\label{lem: finite_induce}
Let $r\in \mathbb{R}_{\geq 2}$. Every finite induced subgraph of either $E_r^\closed$ or $E_r^\open$ is equivalent to a fraction graph.
\end{lemma}

\begin{proof}
    In \cite[Theorem 3.1]{MR1905134} it is shown that convex-round graphs are homomorphically equivalent to the complement of a fraction graph. A convex-round graph on~$n$ vertices is a graph with vertex set $\mathbb{Z}_n$, where the neighborhood of any vertex is an interval, i.e. of the form $\{u, u+1, \dots, u_{d_v-1}\}$. Every finite induced subgraph of either~$E_r^\closed$ or~$E_r^\open$ is the complement of a convex-round graph and thus cohomomorphically equivalent to a fraction graph.
\end{proof}

When the finite set of vertices that we induce to is \emph{equidistant} on the circle $C$ (the vertex set of the circle graphs), we can say more. Namely, we can say precisely which fraction graph we are equivalent to, and the equivalence becomes an isomorphism. Moreover, we describe how the natural ``rounding map'' on a circle graph to finitely many equidistant points gives a cohomomorphism to a fraction graph.

\begin{lemma}
    \label{lem: induce_equidistant}
    Let $r\in \mathbb{R}_{\geq 2}$, $N \in \mathbb{Z}_{\geq 2}$, and let $S$ be a set of $N$ equidistant points on $C$. 
    \begin{enumerate}[\upshape(a)]
        \item\label{item: induce_equidistant} 
        The induced subgraphs $E_r^{\open}[S]$ and $E_r^{\closed}[S]$ are isomorphic to $E_{N/\lceil N/r\rceil}$ and  $E_{N/(\lfloor N/r\rfloor+1)}$, respectively.
        \item\label{item:rounding-clockwise} 
        Identify $S$ with $\mathbb{Z}_N$ (in an order-preserving way) and define the map $f: C \to \mathbb{Z}_N$ by rounding any $x\in C$ to the clockwise next element of $S$. Then $f$ is a cohomomorphism from both $E_r^{\open}$ and $E_r^\closed$ to $E_{N/\lfloor N/r\rfloor}$.
    \end{enumerate}
\end{lemma}
\begin{proof}
    \eqref{item: induce_equidistant} Both $E_r^{\open}[S]$ and $E_r^{\closed}[S]$ are vertex-transitive. Let $q_{\open}$ be the unique positive integer for which $(q_{\open}-1)/N < 1/r \leq q_{\open}/N$, that is, $q_{\open}  = \lceil N/r\rceil$ and let $q_{\closed}$ be the unique positive integer for which $(q_{\closed}-1)/N \leq 1/r < q_{\closed}/N$, that is, $q_{\closed} = \lfloor N/r\rfloor+1$.  We will prove that $E_r^{\open}[S]$ is isomorphic to $E_{N/q_{\open}}$. The fact that $E^{\closed}[S]$ is isomorphic to $E_{N/q_\closed}$ can be proved completely analogously.

    Any vertex in $E_r^{\open}[S]$ is adjacent to its $q_{\open}-1$ nearest clockwise $C$-neighbors and $q_{\open}-1$ nearest anti-clockwise $C$-neighbors. If this describes all vertices in $S$ then $E_r^{\open}[S]$ is a complete graph on $N$ vertices where $N \leq 2(q_\open-1) - 1 = 2q_\open - 1$. In this case $N/q_{\open} < 2$ and thus $E_r^{\open}[S]$ is indeed isomorphic to $E_{N/q_{\open}}$ because by definition $E_{p/q}$ denotes a complete graph on $p$ vertices whenever $p/q< 2$. Otherwise we observe that in both the clockwise and anti-clockwise direction any vertex in $E_r^\open$ is adjacent to its $q_\open-1$ nearest $C$-neighbors but not to its $q_\open$ nearest $C$-neighbor, that is, $E_r^{\open}[S]$ is isomorphic to $E_{N/q_\open}$.

    \eqref{item:rounding-clockwise} Take distinct $x,y \in C$ and suppose that $f(x)$ is adjacent or equal to $f(y)$. We may assume that $f(y) = f(x) + m$ for some $0 \leq m < \lfloor N/r\rfloor$. Let $d_x$ and~$d_y$ denote the distance from $x$ and $y$ to the nearest clockwise element of $S$ respectively. If $f(x) = f(y)$ we assume that $d_x \geq d_y$. The clockwise distance from $x$ to $y$ is 
    \[
        m/N + d_x - d_y < (m+1)/N \leq \lfloor N/r\rfloor/N \leq 1/r.
    \]
    Therefore, $x$ and $y$ are adjacent in both $E_r^\open$ and $E_r^\closed$ and thus $f$ is indeed a cohomomorphism.
\end{proof}

For context, recall that for $p/q \in \QQ_{\geq2}$ the open circle graph $E^\open_{p/q}$ and the fraction graph $E_{p/q}$ are equivalent (\autoref{lem:open-circle-frac-equiv}). We note that this is a special case of the above \autoref{lem: induce_equidistant}. Indeed, let $S$ be a set of $p$ equidistant points on $C$. By \autoref{lem: induce_equidistant} \eqref{item: induce_equidistant} the induced subgraph $E^\open_{p/q}[S]$ is isomorphic to $E_{p/q}$ and thus $E_{p/q} \leq E^\open_{p/q}$. Moreover, rounding to $S$, as described in \autoref{lem: induce_equidistant} \eqref{item:rounding-clockwise}, is a cohomomorphism from $E^\open_{p/q}$ to $E_{p/q}$.

Now we will work towards the characterization of self-cohomomorphisms of irrational circle graphs. We first show that these are continuous.

\begin{lemma}\label{lem:cohom-cont}
    Let $r \in \mathbb{R}_{\geq 2}$ be irrational and let $f$ be a cohomomorphism from either $E_{r}^{\open}$ or $E_{r}^{\closed}$ to itself. Then $f$ is continuous. %
\end{lemma}

\begin{proof}
    Let $(p_n/q_n)_{n \geq 0}$ be the sequence of convergents of the continued fraction representation of $r$. It follows from \autoref{cor: frac_approx} and the fact that $r$ is irrational, that for every $n \geq 0$ we have
    \[
        \lceil p_{2n}/r\rceil = \lfloor p_{2n}/r\rfloor+1 = q_{2n}
        \quad\text{ and }\quad
        \lfloor p_{2n+1}/r\rfloor = q_{2n+1}.
    \]

    We will show that $f$ is continuous at an arbitrary $x\in C$. Let $n \in \mathbb{Z}_{\geq 1}$ and let~$S$ be the set of $p_{2n}$ equidistant points on $C$ containing $x$. Label the elements of $S$ as $(x, s_1, \dots, s_{p_{2n}-1})$ ordered (say clockwise) around $C$. By \autoref{lem: induce_equidistant} $S$ induces $E_{p_{2n}/q_{2n}}$ in both $E_{r}^\open$ and $E_{r}^\closed$. In fact, because $r$ is irrational, there is some $\delta_n$ such that if $|x-y| < \delta_n$ then the set $S_y := \{y, s_1, \dots, s_{p_{2n}-1}\}$ still induces $E_{p_{2n}/q_{2n}}$ (this also implies that the sequence $(y, s_1, \dots, s_{p_{2n}-1})$ is still ordered around $C$). 

    It follows from \autoref{lem: finite_induce} that $f(S_y)$ induces a graph equivalent to $E_{a/b}$ with $a/b \leq r$, where we may assume that $a/b$ is reduced. Because of this assumption $\overline{E_{a/b}}$ is a core (see \autoref{rem: core}) which implies $a \leq |f(S_y)| \leq p_{2n}$. Because $f(S_y)$ is the image of a subgraph isomorphic to $E_{p_{2n}/q_{2n}}$ and $f$ is a cohomomorphism, it follows that $p_{2n}/q_{2n}\leq a/b$. It thus follows from \autoref{lem: CFA} \eqref{it: CFA4} that $p_{2n}/q_{2n} = a/b$. Therefore, by \autoref{rem: core}, $f(S_y)$ induces a graph isomorphic to $E_{p_{2n}/q_{2n}}$. \autoref{lem: self_co_imp_auto_finite} also states that any automorphism of $E_{p_{2n}/q_{2n}}$ must be a rotation possibly composed with a reflection and thus the points $(f(y),f(s_1),\dots, f(s_{p_{2n}-1}))$ are still cyclically ordered (possibly anti-clockwise) around $C$. %
    
    We claim that the distance between any subsequent (on $C$) points in $f(S_y)$ is at most $1/{p_{2n-1}}$. Indeed, suppose that there is a pair of $C$-adjacent points whose distance is strictly more. Then there is a set $T$ of $p_{2n-1}$ equidistant points such that there are two points in $T$ for which the closed arc between the two points does not contain an element of $f(S_y)$. By rounding to $T$ as described in  \autoref{lem: induce_equidistant} \eqref{item:rounding-clockwise} we obtain a non-surjective cohomomorphism from $E_{p_{2n}/q_{2n}}$ to $E_{p_{2n-1}/q_{2n-1}}$. 
    However, by \autoref{th:remove-point} and \autoref{lem: CFA} \eqref{it: CFA2}, we find that $E_{p_{2n-1}/q_{2n-1}}$ with a vertex removed is equivalent to $E_{p_{2(n-1)}/q_{2(n-1)}}$. This is a contradiction because $p_{2(n-1)}/q_{2(n-1)} < p_{2n}/q_{2n}$.

    We can thus conclude that for all $y$ with $|x-y| < \delta_n$ the point $f(y)$ lies on the arc between $f(s_1)$ and $f(s_{p_{2n}-1})$ and has distance at most $1/p_{2n-1}$ from both endpoints. It follows that the arc described by the points at most distance $\delta_n$ from $x$ is mapped into an arc of length at most $1/p_{2n-1}$ by $f$. This proves that $f$ is continuous.
\end{proof}

\begin{theorem}\label{lem:rotation-reflection}
    Let $r \in \mathbb{R}_{\geq 2}$ irrational and let $f$ be a cohomomorphism from either $E_{r}^{\open}$ or $E_{r}^{\closed}$ to itself. Then $f$ is a rotation possibly composed with a reflection.
\end{theorem}

\begin{proof}
    From \autoref{lem:cohom-cont} we know that $f$ is continuous.
    By identifying $C$ with $\mathbb{R}/\mathbb{Z}$ the map $f$ can be lifted to a continuous map $g: \mathbb{R} \to \mathbb{R}$ under the canonical projection $\pi: \mathbb{R} \to C$, that is, $\pi \circ g = f \circ \pi$. The lift $g$ is unique up to an additive integer. There is an integer $d$, called the degree of $f$, such that $g(x+1) = g(x) + d$ for all $x$. (See, e.g., \cite[Proposition 4.3.1]{MR1995704} or \cite{MR3728284}.) 
    By possibly composing with a reflection we may assume that $d$ is nonnegative. 

    Because $f$ is a cohomomorphism we find that for every $(x,y)\in \mathbb{R}\times(1/r,1-1/r)$ there is an integer $m_{x,y}$ such that $g(x+y) - g(x) \in [m_{x,y}+1/r,m_{x,y}+1-1/r]$. Because $\mathbb{R}\times(1/r,1-1/r)$ is connected and $g$ is continuous, $m_{x,y}$ does not depend on $x$ or $y$, so we can write $m_{x,y} = m$. Furthermore, the inclusion extends to the boundary, that is, $g(x+y) - g(x) \in [m+1/r,m+1-1/r]$ for $y \in [1/r,1-1/r]$. It follows that the value of the following integral must be contained in the interval $[m+1/r,m+1-1/r]$ for all $y \in [1/r,1-1/r]$,
    \[
        \int_{0}^1 g(x+y)-g(x) dx = \int_{y}^1 g(x) dx + \int_{0}^{y} g(x+1) dx - \int_0^1 g(x)dx = d\cdot y,
    \]
    using $g(x+1) = g(x) + d$.
    It follows that $m=0$ and $d=1$. Because $g(x+1/r)-g(x)\geq 1/r$ and integrates to exactly $1/r$ on the interval $[0,1]$ we have that $g(x+1/r)-g(x) = 1/r$ for all $x \in [0,1]$ and thus for all $x \in \mathbb{R}$. It follows that the function $h(x) = g(x)-x$ has period $1$ and period $1/r$. Because $1/r$ is irrational, $h$ is constant and thus $g(x) = g(0) + x$. We conclude that $f$ is a rotation.
\end{proof}

\subsection{Non-equivalence of open and closed circle graphs}
\label{sec:nonequiv}

We will now prove the non-equivalence of open and closed circle graphs. We will first give the proof for the rational case, which is simpler, and then prove the irrational case using the self-cohomorphism characterization of \autoref{lem:rotation-reflection}.

\begin{theorem}\label{th:rat-inequiv}
    Let $p/q \in \QQ_{\geq2}$ be rational. Then $E_{p/q}^\closed$ and $E_{p/q}^\open$ are not equivalent.
\end{theorem}
\begin{proof}
    Suppose $E_{p/q}^\open \leq E_{p/q}^\closed$. We know that $E_{p/q}$ is equivalent to $E_{p/q}^\open$. Thus there is a cohomomorphism $f :  E_{p/q} \to E_{p/q}^\closed$. The image of $f$ is a finite induced subgraph of~$E_{p/q}^\closed$. Then by \autoref{lem:fin-cl-frac}, there is an $a/b < p/q$ such that $E_{p/q} \leq E_{a/b}$. This is a contradiction.
\end{proof}

\begin{theorem}\label{th:open-closed-noneq}
    Let $r \in \mathbb{R}_{>2}$ be irrational. Then $E_{r}^\closed$ and $E_{r}^\open$ are not equivalent.
\end{theorem}

\begin{proof} %
Any supposed cohomomorphism $f$ from $E_{r}^\open$ to~$E_{r}^\closed$ is also a cohomomorphism from~$E_{r}^\open$ to itself because the identity is a cohomomorphism from $E_{r}^\closed$ to $E_{r}^\open$. Therefore, by \autoref{lem:rotation-reflection}, $f$ is a rotation possibly composed with a reflection. However, such a map is distance preserving and in particular sends pairs of vertices with distance exactly~$1/r$ to pairs with distance exactly~$1/r$, which contradicts $f$ being a cohomomorphism from~$E_{r}^\open$ to~$E_{r}^\closed$.
\end{proof}

\begin{remark}
    \autoref{th:irr-closed-open-equiv} and \autoref{th:open-closed-noneq} together give a family of examples of pairs of infinite graphs that are asymptotically equivalent but not equivalent (namely, the circle graphs~$E_r^\open$ and $E_r^\closed$ for irrational $r$). Such examples also exist for finite graphs. For instance, let $G = (C_5\sqcup C_5) + E_5$, where $+$ is the join, so $G=\overline{\overline{C_5\sqcup C_5} \sqcup E_5}$. Then $\alpha(G) = \Theta(G) = \overline{\chi}_f(G) = 5$ and $\overline{\chi}(G) = 6$ \cite{MR0207590,MR0321810}. Thus $G$ is asymptotically equivalent to $E_5$, but not equivalent to $E_5$.\footnote{In the analogous asymptotic spectrum theory for tensors (see \autoref{sec:gen-asymp-spec}), no examples are known of objects (tensors) that are asymptotically equivalent but not equivalent. In fact, if the matrix multiplication exponent $\omega$ equals 2 (which is the central open problem in that setting), then the matrix multiplication tensors provide such examples.}
\end{remark}

\section{Independent sets in products of fraction graphs}\label{sec:computation}

In the previous sections we developed---ultimately motivated by the Shannon capacity problem---a theory of convergence of graphs in the asymptotic spectrum distance. 

In this section we return to the Shannon capacity problem, and use ideas from the previous sections (e.g., approximation by fraction graphs, and vertex removal, \autoref{th:remove-point}) to obtain explicit independent sets and new bounds.

First, we will give a brief overview of the best-known lower bounds on the Shannon capacity of small odd-cycle graphs, and observe that a simple, high-level framework encompasses all of these bounds.
Then, in the context of this framework, we introduce a  ``nondeterministic rounding'' technique, which we use to find a new lower bound on the Shannon capacity of the fifteen-cycle, $C_{15}$. %
Finally, we determine the independence number of products of fraction graphs, in several regimes.

\subsection{Shannon capacity of odd cycles via orbits and reductions}\label{subsec:odd-cycl-overview}

We observe (and explicitly describe in \autoref{tab:overview}) that the best-known lower bounds for the Shannon capacity of small odd cycle graphs can be obtained via a remarkably simple and uniform strategy (which may be thought of as a ``finite'' version of the graph limit approach).
Namely, they can be obtained by (1) relating the target graph $G = C_n = E_{n/2}$ to another fraction graph $H = E_{p/q}$ (which we may call an ``intermediate'' or ``auxiliary'' graph), and then (2) constructing a large independent set in a power of~$H$ using an ``orbit'' construction. We explain this approach using three examples, with increasing complexity in the reduction step to the intermediate graph~$H$.

\begin{enumerate}[(1)]
\item The simplest example is, famously, the five-cycle $G = C_5$. Here we simply take the intermediate graph to be $H = G$. The Shannon capacity $\Theta(C_5) = \sqrt{5}$ is attained by the independent set $\{t \cdot (1,2) : t \in \ZZ_5\} \subseteq G^{\boxtimes 2}$ \cite{MR0089131, Lovasz79}, which has a clear orbit structure (under the action of $\ZZ_5$). 

\item A more involved example is for $G = C_{11}$. Here the best-known lower bound on the Shannon capacity $\Theta(C_{11})$ can be obtained by taking $H = E_{148/27}$, which satisfies $H \leq G$ (since $148/27 \leq 11/2$), and verifying that $\{t \cdot (1,11,11^2) : t \in \ZZ_{148}\}$ is an independent set in $H^{\boxtimes 3}$ of size 148. Since $H \leq G$ implies $\alpha(H^{\boxtimes k}) \leq \alpha(G^{\boxtimes k})$, we directly find $\Theta(C_{11}) \geq 148^{1/3} \approx 5.29$. This bound was obtained in \cite{MR0337668}; the construction given there, however, is different.

\item Another type of reduction from $G$ to an intermediate $H$ is used for $G = C_7$ in~\cite{MR3906144} (and for a new bound that we obtain for $G = C_{15}$ in \autoref{subsec:fifteen}). Namely, here $H$ is taken to be $H = E_{382/108}$ which satisfies $G \leq H$ (since $7/2 \leq 382/108$). The set $\{t\cdot (1,7,7^2,7^3,7^4) : t \in \ZZ_{382}\}$ can be verified to be an independent set in $H^{\boxtimes 5}$ of size 382. Now $H$ sits above $G$, so $\alpha(G^{\boxtimes k}) \leq \alpha(H^{\boxtimes k})$, which is the wrong direction for getting a lower bound on $\Theta(G)$. However, intuitively, since~$H$ sits just above $G$, we may expect that the independent set in $H^{\boxtimes k}$ can be ``mapped down'' to $G^{\boxtimes k}$ with only a small loss in size. Indeed, using an ad hoc shifted rounding map and an integer programming optimization it is shown in  \cite{MR3906144} that this is the case, resulting in an independent set of size 367 in $G^{\boxtimes 5}$, and thus $\Theta(C_7) \geq 3.25$.
\end{enumerate}

We will develop reduction type (3) (the ``$G \leq H$'' type of reduction) further in \autoref{subsec:fifteen} using vertex removal and a new ``nondeterministic rounding'' technique that we introduce. With this we obtain a new lower bound on the Shannon capacity of $C_{15}$.

In \autoref{tab:overview} we list the best-known lower bounds on the Shannon capacity of small odd cycles, how they arise from an orbit construction in a power of an intermediate graph $H$, and the relation between $G$ and $H$ (as one of three types: ``$G = H$'', ``$H \leq G$'' or ``$G \leq H$''). We note that the constructions for $E_{11/2}$ and $E_{13/2}$ that we give in \autoref{tab:overview} do not appear in this form in the literature.

Independent sets of the form $\{t \cdot (g_1, g_2, \ldots, g_k) : t \in \ZZ_m\}$ (or a more general form with more generators) have played an important role in many constructions of independent sets in the literature \cite{MR0089131, MR0337668, MR3906144}. 
A useful property of these is that verifying that they are an independent set can be done very efficiently, by verifying that $t\cdot (g_1, g_2, \ldots, g_k)$ is non-adjacent to $(0, 0, \ldots, 0)$ for every $t \in \ZZ_m$.
Orbit constructions are remarkably effective for finding independent sets. However, we do not know whether orbit constructions (or natural generalisations) are strong enough to achieve the Shannon capacity of odd cycles or fraction graphs in general. Calderbank, Frankl, Graham, Li and Shepp \cite{MR1210400} and Guruswami and Riazanov \cite{LinearShannonCapacity} have shown that orbits in complements of fraction graphs do not achieve Shannon capacity. %
We will use orbits again in \autoref{subsec:disc} to obtain discontinuities of the independence number on products of fraction graphs.

The strategy of determining the Shannon capacity of a graph by reduction to an intermediate graph is reminiscent of the methods used for constructing fast matrix multiplication algorithms. Indeed, the (asymptotically) fastest matrix multiplication algorithms are obtained by a reduction of the ``matrix multiplication tensor'' to an ``intermediate tensor'' (e.g., the Coppersmith--Winograd tensors) for which the relevant properties (in this case the asymptotic tensor rank) are easier to analyse \cite{coppersmith1987matrix,blaser2013fast,DBLP:conf/stoc/AmbainisFG15,v017a002,alman2018limits,williams2023new, wigderson2022asymptotic}.

\begin{table}[H]
\begin{tabular}{lllll}
\toprule
$G$ & $H$ & orbit independent set in $H^{\boxtimes k}$ & reduction &  $\leq \Theta(G)$\\
\midrule
$E_{5/2}$ & $E_{5/2}$ & $\{t\cdot (1,2) : t \in \ZZ_5\}$ & $H = G$ & $2.23$ {\footnotesize\cite{MR0089131}}  \\
$E_{7/2}$ & $E_{382/108}$ & $\{t\cdot (1,7,7^2,7^3,7^4) : t \in \ZZ_{382}\}$ & $G \leq H$ & $3.25$ {\footnotesize\cite{MR3906144}} \\
$E_{9/2}$ & $E_{9/2}$ & $\{s\cdot (1,0,2) + t \cdot (0,1,4): s,t \in \ZZ_9\}$ & $H = G$ & $4.32$ {\footnotesize\cite{MR0337668}} \\
$E_{11/2}$ & $E_{148/27}$ & $\{t \cdot (1,11,11^2): t \in \ZZ_{148}\}$ & $H \leq G$ & 5.28 {\footnotesize\cite{MR0337668}}\\
$E_{13/2}$ & $E_{247/38}$ & $\{t \cdot (1, 19, 117) : t\in \ZZ_{247}\}$ & $H \leq G$ & 6.27 {\footnotesize\cite{MR0337668}\tablefootnote{Here the graph $H$ is in fact equivalent to $E_{13/2}$ (but not isomorphic).}} \\
$E_{15/2}$ & $E_{2873/381}$ & $\{t\cdot (1, 15, 1073, 1125) : t\in \ZZ_{2873}\}$ & $G \leq H$ & 7.30 {\footnotesize(\autoref{subsec:fifteen})}\\
\bottomrule
\end{tabular}
\caption{Overview of the best-known lower bounds on the Shannon capacity of odd cycle graphs $G = C_n = E_{n/2}$, and how they arise from orbits in the $k$th power of an intermediate graph $H = E_{p/q}$ for some $k$. The ``reduction'' column describes the cohomomorphism relation between the intermediate graph $H$ and the target graph $G$. The reduction type ``$G \leq H$'' requires special techniques, which we discuss in \autoref{subsec:fifteen}. The lower bound on $E_{15/2}$, which is new, is also discussed there. The constructions given here for $E_{11/2}$ and $E_{13/2}$ are different from previously published constructions, which were not obtained as orbits in powers of fraction graphs.}
\label{tab:overview}
\end{table}

\subsection{New lower bound on the Shannon capacity of the fifteen-cycle}\label{subsec:fifteen}
In this section we prove a new lower bound on the Shannon capacity of the cycle graph on fifteen vertices,~$C_{15}$.

\subsubsection*{Independent set in the third power via vertex removal}

We begin by reproving a known lower bound for the purpose of illustrating some of our methods.
Mathew and Östergård \cite{Mathew2017NewLB} proved that $\alpha(C_{15}^{\boxtimes 3}) \geq 381$. The following slightly better bound had previously been observed by Codenotti, Gerace and Resta~\cite{MR1961489} and later independently by Polak and Schrijver (personal communication): %

\begin{theorem}\label{th:C15-3}
    $\alpha(C_{15}^{\boxtimes 3}) \geq 382$.
\end{theorem}

We give a short proof of \autoref{th:C15-3} using an orbit construction in a product of fraction graphs and vertex removal (\autoref{th:remove-point}).

\begin{proof}
    One can verify directly that the set
    \[
    S = \{ t \cdot (1, 75, 75^2) : t \in \ZZ_{383}\}
    \]
    is an independent set in $E_{383/51}^{\boxtimes 3}$ of size 383. The number $383$ is prime and thus $(0,0,0)$ is the only vertex in $S$ containing $0$ as an entry. Therefore $S \setminus \{(0,0,0)\}$ is an independent set of size $382$ in $(E_{383/51} -\{0\})^{\boxtimes 3}$. Since $383 \cdot 2 - 51 \cdot 15 = 1$, it follows from \autoref{th:remove-point} that this graph is equivalent to $E_{15/2}^{\boxtimes 3} = C_{15}^{\boxtimes 3}$.
\end{proof}

\subsubsection*{Independent set in the fourth power via nondeterministic rounding}

\autoref{th:C15-3} gives $\Theta(C_{15}) \geq \alpha(C_{15}^{\boxtimes 3})^{1/3} \geq 7.25584$. We will now improve this bound by constructing a large independent set in the fourth power. For this we will introduce a new nondeterministic rounding technique, which we will apply to an orbit construction in a product of fraction graphs.

\begin{theorem}\label{th:C15-4}
    $\alpha(C_{15}^{\boxtimes 4}) \geq 2842$.
\end{theorem}

\autoref{th:C15-4} gives $\Theta(C_{15}) \geq \alpha(C_{15}^{\boxtimes 4})^{1/4} \geq 7.30139$. %
The best-known upper bound is $\Theta(C_{15}) \leq \vartheta(C_{15}) =  7.4171482$, where $\vartheta$ is the Lovász theta function. We provide an explicit independent set in \autoref{sec:explicit-ind-set}.

The method we use to obtain \autoref{th:C15-4} is inspired by the method used by Polak and Schrijver \cite{MR3906144} to give a lower bound for $\alpha(C_7^{\boxtimes 5})$. They considered the following independent set in the graph $E_{382/108}^{\boxtimes 5}$ 
\[
    S = \{t\cdot (1,7,7^2,7^3,7^4) : t \in \ZZ_{382}\}.
\]
Because $382/108 \approx 3.537 > 7/2$ this does not immediately give an independent set in~$C_7^{\boxtimes 5}$. However, roughly described, they use a rounded down image of $S$ into $C_7^{\boxtimes 5}$ as a starting point to optimize for an independent set of size $367$. 

Our construction also starts with a similarly shaped independent set $S$, that is, the orbit of a single generator in a product of fraction graphs slightly above $C_{15}^{\boxtimes 4}$. Subsequently we use a different method of rounding down $S$ that is based on the following heuristic. Whenever $p/q \leq 15/2$, the map $x \mapsto \lfloor 15x/p \rfloor$ is a cohomomorphism from $E_{p/q}$ to $E_{15/2}$. This is not the case when $p/q > 15/2$ but heuristically the map is still close to a cohomomorphism whenever $p/q$ is close to $15/2$. Therefore we still use this map, however if a vertex is close to being mapped to a different element, that is, when $15x/p$ is close to an integer, we add both choices to the image $S$. The following contains an exact description of our construction.
\begin{proof}[Proof of \autoref{th:C15-4}]
We describe the construction here and give the explicit independent set in \autoref{sec:explicit-ind-set}.
We start with the set
\[
S = \{t\cdot(1, 15, 1073, 1125) : t \in \ZZ_{2873}\}.
\]
This is an independent set in $E_{2873/383}\boxtimes E_{2873/382}\boxtimes E_{2873/381} \boxtimes E_{2873/381}$ of size $2873$. %
For any $x \in \{0,1,\ldots, 2872\}$ we let $r(x) = 15x/2873$. For any $\eps \in [0,1/2)$ we define the set $F_\eps(x) \subseteq \ZZ_{15}$ by %
\[
    F_\eps(x) = 
\begin{cases}
  \{\lfloor r(x) \rfloor, \lfloor r(x) \rfloor + 1 \} & \text{if } r(x) - \lfloor r(x) \rfloor  > 1 - \eps \\
  \{\lfloor r(x) \rfloor, \lfloor r(x) \rfloor - 1 \} & \text{if } r(x) - \lfloor r(x) \rfloor  < \eps \\
  \{\lfloor r(x) \rfloor\} & \text{otherwise.}
\end{cases}
\]
For any tuple $\vec{\eps} = (\eps_1,\eps_2,\eps_3,\eps_4)$, we define the set
\[
    T_{\vec{\eps}} = \hspace{-15 pt} \bigcup_{(v_1,\ldots,v_4) \in S} \hspace{-15 pt} F_{\eps_1}(v_1) \times \cdots \times F_{\eps_4}(v_4).
\]
Then $T_{\vec{\eps}}$ is a subset of $V(C_{15}^{\boxtimes 4})$.
For small $\vec{\eps}$ the set $T_{\vec \eps}$ is sufficiently small for optimization software (in our case Gurobi \cite{gurobi}) to quickly calculate close to optimal lower bounds for $\alpha(C_{15}^{\boxtimes 4}[T_{\vec\eps}])$.
For $\vec{\eps} = (0.12,0.22,0.32,0.32)$, we have $|T_{\vec{\eps}}| = 13718$, and we find $\alpha(C_{15}^{\boxtimes 4}[T_{\vec\eps}]) \geq 2842$.
\end{proof}

\begin{remark}
    The lower bound $\alpha(C_7^{\boxtimes 5}) \geq 367$ of \cite{MR3906144} can similarly be obtained using nondeterministic rounding (with $\eps = 0.125$) applied to the orbit independent set $S = \{t\cdot (1,7,7^2,7^3,7^4) : t \in \ZZ_{382}\}$.
\end{remark}

\subsection{Characterization of discontinuities of the independence number}\label{subsec:disc}

In this subsection we determine the independence number of products of fraction graphs in certain regimes. We will do this by determining where (and by how much) this value ``jumps'' (i.e., the discontinuity points) when we vary the fraction graphs.

\subsubsection*{Nested floor bound}

We start with a well-known upper bound on the independence number $\alpha$ of the product of any two graphs. %
Recall that $\overline{\chi}_f$ denotes the fractional clique covering number.
Hales~\cite{MR0321810}, using ideas of Rosenfeld \cite{MR0207590} (see also \cite[Equation~67.101]{MR1956924}), proved for any two graphs $G,H$ that
\begin{equation}\label{eq:alpha-prod}
\alpha(G \boxtimes H) \leq \lfloor \overline{\chi}_f(G) \alpha(H) \rfloor.    
\end{equation}
As a direct consequence we have the following ``nested floor bound'' on the independence number of any product of fraction graphs.
\begin{theorem}\label{th:nested-floor} Let $p_1/q_1, \ldots, p_k/q_k \in \QQ_{\geq2}$. Then
\begin{equation}\label{eq:nested-floor}
\alpha(E_{p_1/q_1} \boxtimes E_{p_2/q_2} \boxtimes \cdots \boxtimes E_{p_k/q_k}) \leq \Bigl\lfloor\cdots \Bigl\lfloor\Bigl\lfloor \frac{p_1}{q_1} \Bigr\rfloor\frac{p_2}{q_2}\Bigr\rfloor \cdots \frac{p_k}{q_k} \Bigr\rfloor.
\end{equation}
\end{theorem}
\begin{proof}
This follows from  $\overline{\chi}_f(E_{p/q}) = p/q$ and repeated application of \autoref{eq:alpha-prod}.
\end{proof}
We note that the left-hand side of \autoref{eq:nested-floor} is invariant under permuting the numbers $p_i/q_i$, while the right-hand side is not. Thus when applying \autoref{th:nested-floor} we may minimize over these permutations to get the best bound. Moreover, we note that the minimizing permutation is not obtained by simply sorting the $p_i/q_i$. For instance, note that  $\lfloor\lfloor\lfloor 5/2 \rfloor 2\rfloor 3\rfloor = \lfloor\lfloor\lfloor 5/2 \rfloor 3\rfloor 2\rfloor = 12$ while all other permutations of $5/2$, $2$ and $3$ lead to a larger value. 

\subsubsection*{Discontinuities for one or two fraction graphs}

For a single factor, the bound of \autoref{th:nested-floor} is easily seen to be optimal, that is,
\[
\alpha(E_{p/q}) = \Bigl\lfloor \frac{p}{q} \Bigr\rfloor.
\]
Hales \cite[Theorem 7.1]{MR0321810} proved that the bound of \autoref{th:nested-floor} is optimal for the product of any two odd cycles. Badalyan and Markosyan \cite{MR3016977} proved that the bound of \autoref{th:nested-floor} is optimal for the product of any two fraction graphs, that is:
\begin{theorem}\label{th:two-factors} Let $p_1 / q_1, p_2 / q_2 \in \QQ_{\geq2}$. Then
\[
\alpha(E_{p_1/q_1} \boxtimes E_{p_2/q_2}) = \min \Bigl\{ \Bigl\lfloor\Bigl\lfloor \frac{p_1}{q_1} \Bigr\rfloor\frac{p_2}{q_2}\Bigr\rfloor, \Bigl\lfloor\Bigl\lfloor \frac{p_2}{q_2} \Bigr\rfloor\frac{p_1}{q_1}\Bigr\rfloor \Bigr\}.
\]    
\end{theorem}
We give a short proof of \autoref{th:two-factors}. For this we use the following construction of an independent set in certain products of fraction graphs, which extends the well-known construction of an independent set of size five in $C_5^{\boxtimes 2}$.

\begin{lemma}\label{lem:ind-set-two-factor} Let $p/q \in \QQ_{\geq2}$. Then
    $\alpha(E_{p/q} \boxtimes E_{p / \lfloor p / q \rfloor}) = p$.
\end{lemma}
\begin{proof}
The upper bound follows from \autoref{th:nested-floor}. For the lower bound, define the set $S = \{t \cdot (1, \lfloor p/q \rfloor) : t \in \ZZ_p\}$.  One verifies that $S$ is an independent set in $E_{p/q} \boxtimes E_{p / \lfloor p / q \rfloor}$.
\end{proof}

\begin{proof}[Proof of \autoref{th:two-factors}]
    Let $n_i = \lfloor p_i /q_i \rfloor$. We can write $p_1 / q_1 = n_1 + (m_1 + \ell_1) / n_2$ and $p_2 / q_2 = n_2 + (m_2 + \ell_2) / n_1$ for some $m_i \in \NN$ and $\ell_i \in \RR$ such that $0 \leq m_1 < n_2$, $0 \leq m_2 < n_1$ and $0 \leq \ell_i < 1$. Then
    \[
    \Bigl\lfloor\Bigl\lfloor \frac{p_1}{q_1} \Bigr\rfloor\frac{p_2}{q_2}\Bigr\rfloor = n_1 n_2 + m_2
    \]
    and
    \[
    \Bigl\lfloor\Bigl\lfloor \frac{p_2}{q_2} \Bigr\rfloor\frac{p_1}{q_1}\Bigr\rfloor = n_1 n_2 + m_1.
    \]
    We may assume that $m_1 \leq m_2$. Apply \autoref{lem:ind-set-two-factor} with $p = n_1 n_2 + m_1$ and $q = n_2$ to get
    \[
    \alpha(E_{(n_1 n_2 + m_1) / n_2} \boxtimes E_{(n_1 n_2 + m_1) / n_1}) = n_1 n_2 + m_1.
    \]
    Because $(n_1 n_2 + m_1) / n_2 \leq p_1 / q_1$ and $(n_1 n_2 + m_1) / n_1 \leq p_2 / q_2$ we find that $\alpha(E_{p_1/q_1} \boxtimes E_{p_2 / q_2}) \geq \alpha(E_{(n_1 n_2 + m_1) / n_2} \boxtimes E_{(n_1 n_2 + m_1) / n_1}) = n_1 n_2 + m_1$.
\end{proof}

\subsubsection*{Discontinuities for three (or more) fraction graphs}

\autoref{th:two-factors} says that \autoref{th:nested-floor} is tight for two factors. For three factors, however, the bound in \autoref{th:nested-floor} is no longer always optimal and the independence number $\alpha(E_{p_1 / q_1} \boxtimes E_{p_2/q_2} \boxtimes E_{p_3 / q_3})$ shows more intricate behaviour. We fully determine this behaviour for all triples of fractions in the interval $[2,3]$. Since the function $\alpha(E_{p_1 / q_1} \boxtimes E_{p_2/q_2} \boxtimes E_{p_3 / q_3})$ is monotone in each $p_i/q_i$ and integral-valued, it has only finitely many discontinuities, and to determine the function it suffices to determine the value at every discontinuity.

To discuss the discontinuities of the independence number of products of fraction graphs, we define for any $k \in \NN$ the function
\[
    \alpha_k : \QQ_{\geq2}^k \to \NN : (p_1/q_1, \ldots, p_k/q_k) \mapsto \alpha(E_{p_1 / q_1} \boxtimes \cdots \boxtimes E_{p_k / q_k}).
\]
For $u,v \in \QQ_{\geq2}^k$ we write $u \leq v$ if there is a permutation of $u$ that is pointwise at most~$v$. Then $u \leq v$ implies $\alpha_k(u) \leq \alpha_k(v)$. We write $u < v$ if $u \leq v$ and $u \neq v$.
We call $v \in \QQ_{\geq2}^k$ an (extremal) discontinuity of $\alpha_k$ if for every $u \in \QQ_{\geq2}^k$ if $u < v$ then $\alpha_k(u) < \alpha_k(v)$.

\begin{lemma}\label{lem:disc-integer}
    For every $n,k \in \NN$, $(n,p_1/q_1, \ldots, p_k/q_k)$ is a discontinuity of $\alpha_{k+1}$ if and only if $(p_1/q_1, \ldots, p_k/q_k)$ is a discontinuity of $\alpha_k$.
\end{lemma}
\begin{proof}
We use the fact that the independence number is additive under disjoint union, so that $\alpha(E_n \boxtimes G) = n \alpha(G)$ for any graph $G$ and any $n \in \NN$. 

Suppose $(n,p_1/q_1, \ldots, p_k/q_k)$ is a discontinuity. Let $r/s < p_1/q_1$, say. Then 
\begin{align*}
\alpha(E_{r/s} \boxtimes E_{p_2/q_2} \boxtimes \cdots ) &= \tfrac 1n \alpha(E_n \boxtimes E_{r/s} \boxtimes E_{p_2/q_2} \boxtimes \cdots )\\
& <\tfrac 1n \alpha(E_n \boxtimes E_{p_1/q_1} \boxtimes E_{p_2/q_2} \boxtimes \cdots )\\
&= \alpha(E_{p_1/q_1} \boxtimes E_{p_2/q_2} \boxtimes \cdots),
\end{align*}
so $(p_1/q_1, \ldots, p_k/q_k)$ is a discontinuity.

Suppose $(p_1/q_1, \ldots, p_k/q_k)$ is a discontinuity. Suppose $r/s < p_1/q_1$, say. Then 
\begin{align*}
    \alpha(E_n \boxtimes E_{r/s} \boxtimes E_{p_2/q_2} \boxtimes \cdots) &=  n \alpha(E_{r/s} \boxtimes E_{p_2/q_2} \boxtimes \cdots ) \\
    &< n \alpha (E_{p_1/q_1} \boxtimes E_{p_2/q_2} \boxtimes \cdots)\\
    &= \alpha(E_n \boxtimes E_{p_1/q_1} \boxtimes E_{p_2/q_2} \boxtimes \cdots).
\end{align*}
Suppose $r/s < n$. Then \autoref{th:nested-floor} gives $\alpha(E_{r/s} \boxtimes E_{p_1/q_1} \boxtimes \cdots) < \alpha(E_n \boxtimes E_{p_1/q_1} \boxtimes \cdots )$. We conclude that $(n,p_1/q_1, \ldots, p_k/q_k)$ is a discontinuity. %
\end{proof}

\begin{example}\label{ex:disc}
The discontinuities of $\alpha_1$ are at the natural numbers~$\NN_{\geq2}$, and we have $\alpha_1(n) = n$ for every $\NN_{\geq2}$.

The discontinuities of $\alpha_2$ can be computed using \autoref{th:two-factors}. 
Concretely, restricted to $(\QQ\cap [2,3])^2$, we see from \autoref{lem:disc-integer} that $(2,2)$, $(2,3)$ and $(3,3)$ are discontinuities. Taking $2 < p_1/q_1, p_2/q_2 < 3$ we see from \autoref{th:two-factors} that $\alpha_2(p_1/q_1, p_2/q_2)$ equals $\min \{\lfloor 2p_1/q_1 \rfloor, \lfloor 2 p_2/q_2 \rfloor\}$, which equals 5 if and only if both $p_1/q_1$ and $p_2/q_2$ are at least $5/2$. Thus, the discontinuities of $\alpha_2$ on $(\QQ\cap [2,3])^2$ are, up to permutation, at
\[
\alpha_2(2,2) = 4,\, \alpha_2(5/2,5/2) = 5, \, \alpha_2(2, 3) = 6,\,\alpha_2(3,3) = 9.
\]
\end{example}

\begin{theorem}\label{th:discont}
The discontinuities of $\alpha_3$ restricted to $(\QQ\cap[2,3])^3$ are, up to permutation,~at
    \begin{align*}
        \alpha_3(2,2,2) &= 8            & \alpha_3(5/2, 5/2, 8/3) &= 11\\
        \alpha_3(2,2,3) &= 12           & \alpha_3(8/3, 8/3, 8/3) &= 12\\
        \alpha_3(2,3,3) &= 18           & \alpha_3(11/5, 11/4, 11/4) &= 11\\
        \alpha_3(2, 5/2, 5/2) &= 10     & \alpha_3(11/4, 11/4, 11/4) &= 13\\
        \alpha_3(5/2, 5/2, 3) &= 15     & \alpha_3(14/5, 14/5, 14/5) &= 14\\        
        \alpha_3(9/4, 7/3, 5/2) &= 9    & \alpha_3(3,3,3) &= 27.
    \end{align*}
\end{theorem}

\newcommand\tikzto[1][1.4em]{\tikz[baseline=-0.5ex, 
                                 shorten <=2pt, shorten >=2pt] \draw[-latex] (0,0) -- (#1,0);}

In order to use \autoref{th:discont} to determine the value of $\alpha_3$ at an arbitrary point $v = (p_1/q_1,p_2/q_2,p_3/q_3) \in \QQ\cap [2,3]$ we simply search for the largest $u$ from the listed discontinuity points such that $u \leq v$. Then $\alpha_3(v) = \alpha_3(u)$. As an aid for such a search we provide in \autoref{fig:hasse} the Hasse diagram of the partial ordering on the discontinuity points, in which $u \tikzto v$ means $u \leq v$.

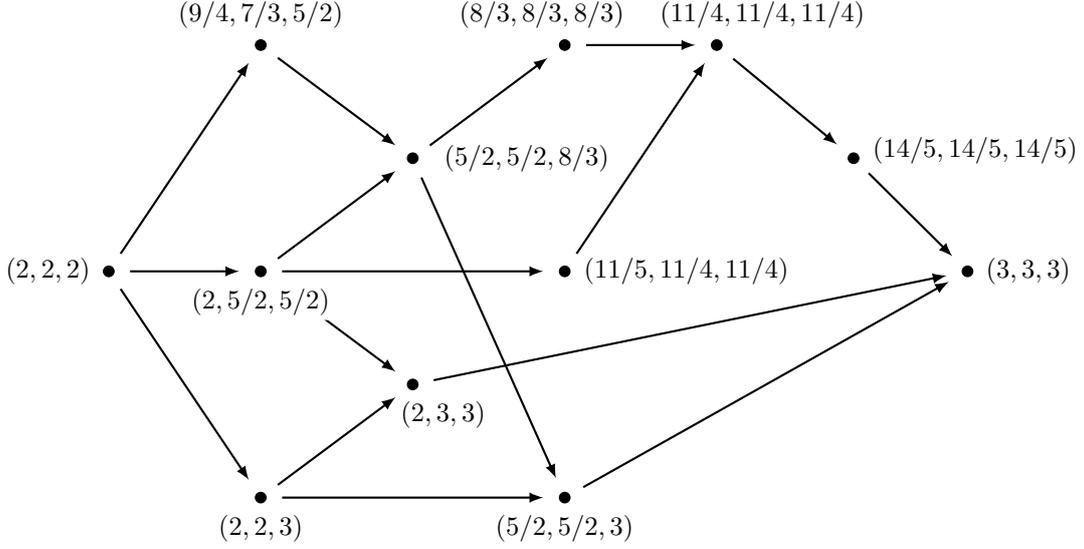
\begin{figure}[H]
\begin{tikzpicture}
\def\s{1} %
\def\t{1} %

\coordinate (A) at (0*\s,0);
\coordinate (B) at (2*\s,-3*\t);
\coordinate (C) at (2*\s,0*\t);
\coordinate (D) at (2*\s,3*\t);

\coordinate (E) at (4*\s,-1.5*\t);
\coordinate (F) at (4*\s,1.5*\t);

\coordinate (G) at (6*\s,-3*\t);
\coordinate (H) at (6*\s,0*\t);
\coordinate (I) at (6*\s,3*\t);

\coordinate (J) at (8*\s,3*\t);
\coordinate (K) at (9.8*\s,1.5*\t);

\coordinate (L) at (11.3*\s,0*\t);

\draw[-latex, shorten >=8pt,shorten <=8pt, thick] (A) -- (B);
\draw[-latex, shorten >=8pt,shorten <=8pt, thick] (A) -- (C);
\draw[-latex, shorten >=8pt,shorten <=8pt, thick] (A) -- (D);
\draw[-latex, shorten >=8pt,shorten <=8pt, thick] (B) -- (E);
\draw[-latex, shorten >=8pt,shorten <=8pt, thick] (B) -- (G);
\draw[-latex, shorten >=8pt,shorten <=8pt, thick] (C) -- (E);
\draw[-latex, shorten >=8pt,shorten <=8pt, thick] (C) -- (H);
\draw[-latex, shorten >=8pt,shorten <=8pt, thick] (C) -- (F);
\draw[-latex, shorten >=8pt,shorten <=8pt, thick] (D) -- (F);
\draw[-latex, shorten >=8pt,shorten <=8pt, thick] (E) -- (L);
\draw[-latex, shorten >=8pt,shorten <=8pt, thick] (F) -- (G);
\draw[-latex, shorten >=8pt,shorten <=8pt, thick] (F) -- (I);
\draw[-latex, shorten >=8pt,shorten <=8pt, thick] (G) -- (L);
\draw[-latex, shorten >=8pt,shorten <=8pt, thick] (H) -- (J);
\draw[-latex, shorten >=8pt,shorten <=8pt, thick] (I) -- (J);
\draw[-latex, shorten >=8pt,shorten <=8pt, thick] (J) -- (K);
\draw[-latex, shorten >=8pt,shorten <=8pt, thick] (K) -- (L);

\draw[black, fill=black] (A) circle (2pt);
\draw (A) ++(-0.8,0) node {\small$(2,2,2)$};

\draw[black, fill=black] (B) circle (2pt);
\draw (B) ++(0,-0.4) node {\small$(2,2,3)$};

\draw[black, fill=black] (C) circle (2pt);
\draw (C) ++(0,-0.4) node[fill=white,rounded corners=2pt,inner sep=2pt] {\small$(2,5/2,5/2)$};

\draw[black, fill=black] (D) circle (2pt);
\draw (D) ++(0,0.4) node {\small$(9/4,7/3,5/2)$};

\draw[black, fill=black] (E) circle (2pt);
\draw (E) ++(0.4,-0.4) node[fill=white,rounded corners=2pt,inner sep=2pt]  {\small$(2,3,3)$};

\draw[black, fill=black] (F) circle (2pt);
\draw (F) ++(1.5, 0) node[fill=white,rounded corners=2pt,inner sep=2pt]   {\small$(5/2,5/2,8/3)$};

\draw[black, fill=black] (G) circle (2pt);
\draw (G) ++(0,-0.4) node[fill=white,rounded corners=2pt,inner sep=2pt] {\small$(5/2,5/2,3)$};

\draw[black, fill=black] (H) circle (2pt);
\draw (H) ++(1.6, 0) node {\small$(11/5,11/4,11/4)$};

\draw[black, fill=black] (I) circle (2pt);
\draw (I) ++(-0.3,0.4) node {\small$(8/3,8/3,8/3)$};

\draw[black, fill=black] (J) circle (2pt);
\draw (J) ++(0.6,0.4) node {\small$(11/4,11/4,11/4)$};

\draw[black, fill=black] (K) circle (2pt);
\draw (K) ++(1.6,0.1) node {\small$(14/5,14/5,14/5)$};

\draw[black, fill=black] (L) circle (2pt);
\draw (L) ++(0.8,0) node {\small$(3,3,3)$};

\end{tikzpicture}
\caption{Hasse diagram of the poset of discontinuity points of \autoref{th:discont}. Here $u \tikzto v$ means $u \leq v$, and by definition we have $u \leq v$ if there is a permutation of $u$ that is pointwise at most $v$.}
\label{fig:hasse}
\end{figure}

\begin{remark}
    We note that for all discontinuity points in \autoref{th:discont} the nested floor bound \autoref{th:nested-floor} is tight, except for the points $(5/2,5/2,8/3)$ and $(8/3,8/3,8/3)$, where it gives 12 and 13, respectively.
\end{remark}

\begin{remark}
    In the recent work of Zhu \cite{zhu2024improved} on the Shannon capacity of the complement of odd cycles, the graphs $E_{9/4} \boxtimes E_{7/3} \boxtimes E_{5/2}$, $E_2 \boxtimes E_{5/2} \boxtimes E_{5/2}$ and $E_2 \boxtimes E_2 \boxtimes E_3$ appear (in a slightly different language). Our \autoref{th:discont} says that these graphs are discontinuity points (and notably they are the ``smallest'' discontinuity points above $(2,2,2)$).
\end{remark}

\begin{remark}
    We have seen in the proof of \autoref{th:two-factors} that all discontinuity points of  $\alpha_2$ can be obtained using independent sets that are orbits in a product of fraction graphs. This is not true for every discontinuity point of $\alpha_3$. Indeed, we claim that the discontinuity point $\alpha_3(8/3,8/3,8/3)=12$ cannot be realized by a subgroup in a very general sense.

    Let $H_1, H_2, H_3$ be finite groups with symmetric subsets $S_i$ and let $G_i = \text{Cayley}(H_i,S_i)$ be the corresponding Cayley graphs.
    Suppose that $G_i \leq E_{8/3}$ and suppose that there is a subgroup $I \subseteq H_1 \times H_2 \times H_3$ of size $12$ that is an independent set in $G_1 \boxtimes G_2 \boxtimes G_3$. We will show that this gives a contradiction.
    Consider the auxiliary graphs $L_i$ with vertex set~$I$ where $x \not\sim y$ if and only if $x_i \not\sim y_i$ in $G_i$. Now $L_i \leq G_i \leq E_{8/3}$ and thus $\overline{\chi}_f(L_i) \leq 8/3$. Moreover, $L_i$ is vertex-transitive so $\overline{\chi}_f(L_i) = |L_i|/\omega(L_i) = 12 /\omega(L_i)$, where $\omega$ denotes the clique number. Because $\omega(L_i)$ is an integer we can conclude that $\overline{\chi}_f(L_i) \leq 12/5$. Observe that $\{(x,x,x): x\in I\}$ is an independent set in $L_1 \boxtimes L_2 \boxtimes L_3$ implying $\alpha(L_1 \boxtimes L_2 \boxtimes L_3) \geq 12$. However, using the nested floor bound (\autoref{th:nested-floor}), we observe that 
    \[
        \alpha(L_1 \boxtimes L_2 \boxtimes L_3) \leq \left\lfloor\overline{\chi}_f(L_1) \left\lfloor \overline{\chi}_f(L_2)  \left\lfloor\overline{\chi}_f(L_3)  \right\rfloor \right\rfloor \right\rfloor  \leq \Bigl\lfloor\frac{12}{5} \Bigl\lfloor \frac{12}{5} \Bigl\lfloor\frac{12}{5} \Bigr\rfloor \Bigr\rfloor \Bigr\rfloor = 9,
    \]
    which gives a contradiction.
\end{remark}

The proof of \autoref{th:discont} consists of a computer-assisted component together with several reductions to the problem, for which we use \autoref{lem:disc-integer} and the following lemma. (We note that all the ideas we present here work also outside  the interval~$[2,3]$. However, the computer-assisted component becomes more costly.)

\begin{lemma}\label{lem:numerator-bound}
    If $(p_1/q_1, \ldots, p_k/q_k)$ is a discontinuity of $\alpha_k$, then \[
    \max \{p_1, \ldots, p_k\} \leq \alpha_k(p_1/q_1, \ldots, p_k/q_k).
    \]
\end{lemma}
\begin{proof}
    Let $m = \alpha_k (p_1/q_1,\ldots, p_k/q_k)$. Suppose $\max_i p_i > m$. Let $S$ be any independent set of $E_{p_1/q_1} \boxtimes \cdots \boxtimes E_{p_k/q_k}$ of size $m$. Then there is an $i \in [k]$ and a vertex $v$ of $E_{p_i/q_i}$ such that for every $s \in S$, $s_i \neq v$. Then $S \subseteq E_{p_1/q_1} \boxtimes \cdots \boxtimes (E_{p_i/q_i} - \{v\}) \boxtimes \cdots \boxtimes E_{p_k/q_k}$. We know that $E_{p_i/q_i} - \{v\}$ is equivalent to a fraction graph $E_{a/b}$ for some $a/b\in \QQ_{\geq2}$ that is strictly smaller than $p_i/q_i$ (\autoref{th:remove-point}). Then $\alpha_k(p_1/q_1, \ldots, a/b, \ldots, p_k/q_k) = m$, so $(p_1/q_1,\ldots, p_k/q_k)$ is not a discontinuity.
\end{proof}

\begin{proof}[Proof of \autoref{th:discont}]
    We will make a set of candidate discontinuities that is small enough so that we can compute the independence number of them with optimization software (Gurobi \cite{gurobi}). 

    Suppose $(p_1/q_1, p_2/q_2, p_3/q_3) \in (\QQ\cap [2,3])^3$ is a discontinuity. We have 
    \[
    \alpha_3(p_1/q_1, p_2/q_2, p_3/q_3) \leq \alpha_3(3, 3, 3) = 27.
    \]
    Then from \autoref{lem:numerator-bound} we know that 
    \begin{equation}\label{eq:bound-27}
    \max_{i\in [k]} p_i \leq 27.
    \end{equation}
    Moreover, again from \autoref{lem:numerator-bound}, and using the theta number and the nested floor bound (\autoref{th:nested-floor}), we know that 
    \begin{equation}\label{eq:bound-theta}
    \max_{i\in [k]} p_i \leq \vartheta(E_{p_1/q_1} \boxtimes \cdots \boxtimes E_{p_k/q_k})
    \end{equation}
    and
    \begin{equation}\label{eq:bound-nested-floor}
    \max_{i\in [k]} p_i \leq \min_{\pi \in S_k} \Bigl\lfloor\cdots \Bigl\lfloor\Bigl\lfloor \frac{p_{\pi(1)}}{q_{\pi(1)}} \Bigr\rfloor\frac{p_{\pi(2)}}{q_{\pi(2)}}\Bigr\rfloor \cdots \frac{p_{\pi(k)}}{q_{\pi(k)}} \Bigr\rfloor.
    \end{equation}
    Our set of candidate discontinuities consists of all $(p_1/q_1, p_2/q_2, p_3/q_3) \in (\QQ\cap [2,3])^3$ such that \autoref{eq:bound-27}, \autoref{eq:bound-theta} and \autoref{eq:bound-nested-floor} hold. Note that \autoref{eq:bound-27} already implies that this set is finite. On these elements we run optimization software to compute the independence numbers, resulting in the claim.
\end{proof}

\autoref{th:discont} describes the discontinuities of $\alpha_3(p_1/q_1, p_2/q_2, p_3/q_3)$ as a function of three inputs. We now consider the symmetric restriction $\alpha_3(p/q, p/q, p/q) = \alpha(E_{p/q}^{\boxtimes3})$.  
Indeed, \autoref{th:discont} directly implies the complete description of the discontinuities of this function, and thus the complete description of the function, for $p/q \in \QQ \cap [2,3]$.
\begin{theorem}\label{th:symm-disc}
For $p/q \in \QQ \cap [2,3]$ we have
        \[
    \alpha(E_{p/q}^{\boxtimes 3}) = 
        \begin{cases}
            8 & p/q \in [2,5/2) \\
            10 & p/q \in [5/2,8/3) \\
            12 & p/q \in [8/3,11/4) \\
            13 & p/q \in [11/4,14/5) \\
            14 & p/q \in [14/5, 3) \\
            27 & p/q = 3.
        \end{cases}
    \]
\end{theorem}
\begin{proof}
    This follows from \autoref{th:discont} by restricting to symmetric triples.
\end{proof}

The values of $\alpha(E_{p/q}^{\boxtimes 3})$ at the discontinuity points listed in \autoref{th:symm-disc} match the previously known values given in \cite[Table 1]{jurkiewicz2014some} and \cite[Table 8.1]{jurkiewicz2014survey}.

We visualize the right-continuous step function described in \autoref{th:symm-disc} as a graph in \autoref{fig:step}. (In this graph we omit the point $(3,27)$ to save space.)

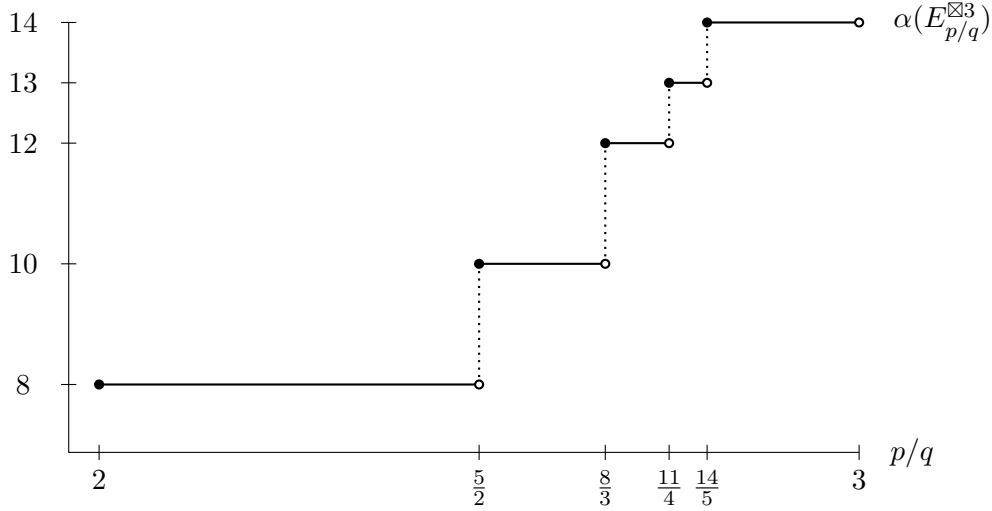
\begin{figure}[H]
\centering
\vspace{1em}
\begin{tikzpicture}
\def\s{10} %
\def\t{0.8} %
\def\a{0.4}
\draw (3*\s,8*\t - 0.9) -- (2*\s-\a,8*\t - 0.9) -- (2*\s-\a,14*\t);
\draw (3*\s,14*\t) ++(1.1,0) node[overlay] {$\alpha(E_{p/q}^{\boxtimes 3})$};
\draw (3*\s,8*\t - 0.9) ++(0.7,0) node[overlay] {$p/q$};

\draw[thick] (2*\s,8*\t) -- (2.5*\s,8*\t); %
\draw[thick] (2.5*\s,10*\t) -- (8/3*\s,10*\t);  %
\draw[thick] (8/3*\s,12*\t) -- (11/4*\s,12*\t); %
\draw[thick] (11/4*\s,13*\t) -- (14/5*\s,13*\t); %
\draw[thick] (14/5*\s,14*\t) -- (3*\s,14*\t); %

\draw[thick, dotted] (2.5*\s,8*\t) -- (2.5*\s,10*\t);
\draw[thick, dotted] (2.666*\s,10*\t) -- (2.666*\s,12*\t);
\draw[thick, dotted] (2.75*\s,12*\t) -- (2.75*\s,13*\t);
\draw[thick, dotted] (2.8*\s,13*\t) -- (2.8*\s,14*\t);

\foreach \x/\y in {2/8, 2.5/10, 2.666/12, 2.75/13, 2.8/14}
   \draw[thick, fill=black] (\x * \s,\y * \t) circle (1.5pt);

\foreach \x/\y in {2.5/8, 2.666/10, 2.75/12, 2.8/13, 3/14}
   \draw[thick, fill=white] (\x * \s,\y * \t) circle (1.5pt);

\foreach \x/\label in {2/{$2$}, 2.5/{$\frac{5}{2}$}, 2.666/{$\frac{8}{3}$}, 2.75/{$\frac{11}{4}$}, 2.8/{$\frac{14}{5}$}, 3/{$3$}}
   \draw (\x * \s,8*\t - 0.8) -- (\x * \s,8*\t - 1.0) node[below] {\label};

\foreach \y/\label in {8/{$8$}, 10/{$10$}, 12/{$12$}, 13/{$13$}, 14/{$14$}}
   \draw (2*\s-\a-0.1,\y*\t) -- (2*\s-\a+0.1,\y*\t) ++(-0.7,0) node {\label};
\end{tikzpicture}
\caption{Graph of the right-continuous step function $\alpha(E_{p/q}^{\boxtimes 3})$ for $p/q \in \QQ \cap [2,3)$ with discontinuity points as described in \autoref{th:symm-disc}.}
\label{fig:step}
\end{figure}

\begin{remark}
    Besides the discontinuity points described in \autoref{ex:disc} and \autoref{th:discont}, some more known discontinuity points are the following:

        (1) For any $n \in \NN$, the function $\alpha_{2n}$ has a discontinuity at $(5/2,\ldots, 5/2)$. This follows from \cite[Theorem 4.2]{MR3144803}, which states that $\vartheta(E_{p/q}) = \vartheta(E_{p'/q'})$ implies that $p/q = p'/q'$, where $\vartheta$ denotes the Lovász theta function. Indeed, let $p/q < 5/2$. Then $\vartheta(E_{p/q}) < \vartheta(E_{5/2}) = \sqrt{5}$ and thus 
        \[
            \alpha(E_{p/q} \boxtimes E_{5/2}^{\boxtimes (2n-1)}) \leq \vartheta(E_{p/q}) \vartheta(E_{5/2})^{2n-1}
            < 5^n = \alpha(E_{5/2}^{\boxtimes 2n}).
        \]

        (2) Let $r \in \mathbb{Z}_{\geq 3}$ and define the sequence $q_n$ by $q_0 = 1$ and $q_{n+1} = r \cdot q_n -1 $. In \cite[Theorem 9.3.3]{SvenThesis} it is shown that for all $n \geq 1$, $\alpha(E_{q_n/q_{n-1}}^{\boxtimes n}) \geq q_n$. It is not hard to verify inductively that for $1 \leq m \leq n$ the following equality holds
        \[
            \underbrace{\Bigl\lfloor\cdots \Bigl\lfloor\Bigl\lfloor \frac{q_n}{q_{n-1}} \Bigr\rfloor\frac{q_n}{q_{n-1}}\Bigr\rfloor \cdots \frac{q_{n}}{q_{n-1}} \Bigr\rfloor}_{\text{$m$ copies}} = q_{m}.
        \]
        It thus follows from \autoref{th:nested-floor} that for $a/b < q_n/q_{n-1}$,
        \[
            \alpha(E_{q_n/q_{n-1}}^{\boxtimes (n-1)} \boxtimes E_{a/b}) \leq \Bigl\lfloor\Bigl\lfloor\cdots \Bigl\lfloor\Bigl\lfloor \frac{q_n}{q_{n-1}} \Bigr\rfloor\frac{q_n}{q_{n-1}}\Bigr\rfloor \cdots \frac{q_{n}}{q_{n-1}} \Bigr\rfloor \frac{a}{b} \Bigr\rfloor = \Bigl\lfloor q_{n-1} \cdot \frac{a}{b} \Bigr\rfloor < q_n,
        \]
        which shows that $(q_n/q_{n-1}, \ldots, q_n/q_{n-1})$ is a discontinuity point for $\alpha_n$. The discontinuity point $(14/5,14/5,14/5)$ for $\alpha_3$ recorded in \autoref{th:discont} is of this form (obtained by setting $r = 3$ and $n = 3$).
        These independent sets were used in \cite{SvenThesis} to prove that the functions $p/q \mapsto F(E_{p/q})$ are left-continuous at integers $r \geq 3$, where $F$ may either represent the Shannon capacity $\Theta$ or any element of the spectrum $\mathcal{X}$. This follows from the facts that $q_n/q_{n-1}$ approaches $r$ from below and that $\lim_{n \to \infty }q_n^{1/n} = r$ and thus
        \[
            r \geq \lim_{n \to \infty} F(E_{q_n/q_{n-1}}) \geq \lim_{n \to \infty} \Theta(E_{q_n/q_{n-1}}) \geq \lim_{n \to \infty} \alpha(E_{q_n/q_{n-1}}^{\boxtimes n})^{\frac{1}{n}} \geq \lim_{n \to \infty} q_n^{\frac{1}{n}} = r. 
        \]
\end{remark}

\normalsize

\section{Discussion and open problems}\label{sec:discussion}

Motivated by the Shannon capacity problem, and with the approach in mind of determining the Shannon capacity of ``hard'' graphs (e.g.,~the prototypical odd cycles) by approximating them with ``easier'' graphs (e.g.,~very structured graphs), we have developed the theory of the asymptotic spectrum distance on graphs and the corresponding notion of convergence. In particular, in \autoref{sec:frac} we have developed several methods for understanding convergence of fraction graphs, including effective distance bounds and topological incompleteness results.

A central problem that we leave open regarding convergence is left-continuity on fraction graphs:
\begin{problem}\label{prob:left-cont}
Does $p_n/q_n \to a/b$ from below, imply that $E_{p_n/q_n} \to E_{a/b}$?    
\end{problem}
In \autoref{th:op-cl-left} we proved that \autoref{prob:left-cont} is equivalent to determining whether the (infinite) open and closed circle graphs $E^\open_{a/b}$ and $E^\closed_{a/b}$ are asymptotically equivalent. 

Generally, we feel that the study of natural infinite graphs like the circle graphs of \autoref{sec:completion} will shed new light on the theory of Shannon capacity. This angle and the use of asymptotic spectrum distance in general hint at the further potential of analytic methods for this otherwise combinatorial and algebraic problem. 
A central open problem regarding infinite graphs is about their completeness: 
\begin{problem}
    Are the infinite graphs $\mathcal{G}_{\infty}$ complete?\footnote{Recall (\autoref{sec:asymp-spec-infinite}) that we defined $\mathcal{G}_\infty$ as the set of infinite graphs with cardinality at most the cardinality of the reals and with finite clique cover number.} %
\end{problem}
Another open problem is about density of the finite graphs in the infinite graphs:
\begin{problem}
    Are the finite graphs $\mathcal{G}$ dense in the infinite graphs $\mathcal{G}_\infty$? 
\end{problem}
In fact, the following is also open:
\begin{problem}\label{prob:gen-conv}
    Is there, for every finite graph $G$, a non-trivial sequence of finite graphs~$H_i$ converging to it, $H_i \to G$?
\end{problem}
Specifically, \autoref{prob:gen-conv} is open for graphs that are not of the form $G = E_{p/q} \boxtimes G'$ with $p/q \in \QQ_{\geq2}$, since those are covered by \autoref{th:rational-right-cont}.

In \autoref{sec:computation} we took a computational approach and used some of the developed methods to determine precisely several relevant independence numbers. In particular, we obtained a new lower bound on the Shannon capacity of $C_{15}$, using what could be considered a ``finite'' version of the converging sequence approach (following Polak--Schrijver~\cite{MR3906144}), by using a fraction graph that sits 
``just above'' it. We expect the development and further integration of these ideas to lead to deeper, structural understanding of the Shannon capacity.

\paragraph{Acknowledgements.} We thank Lex Schrijver, Sven Polak and Siebe Verheijen for helpful discussions, and Dion Gijswijt for helpful comments on the draft. DdB was supported by the Dutch Research Council (NWO) grant 613.001.851. PB was supported by the Dutch Research Council (NWO) grant 639.032.614. JZ was supported by the Dutch Research Council (NWO) Veni grant VI.Veni.212.284. This work used the Dutch national e-infrastructure with the support of the SURF Cooperative using grant no.~EINF-4415 and EINF-8278.

\appendix
\section{General theory of asymptotic spectrum distance}\label{sec:gen-asymp-spec}

In this section we discuss the general theory of asymptotic spectrum distance on preordered semirings (with the appropriate conditions), which has the theory for graphs and for infinite graphs (as discussed in \autoref{sec:dist}) as special cases. The goal of this section is to put our ideas in a broader perspective and to clarify which parts of the theory discussed in the previous sections are general. We begin by reviewing the theory of asymptotic spectrum duality, as can be found in \cite{strassen1988asymptotic, Zui18_thesis, MR4039606, wigderson2022asymptotic, MR4609385} (this is closely related to the real representation theorem of Kadison and Dubois \cite{MR0707730, MR1829790, MR2383959, MR4388320}), and some examples. %
Then we discuss the general notion of asymptotic spectrum distance and basic properties.

\subsubsection*{Semiring and order}
Let $\sr$ be a commutative semiring with additive unit 0 and multiplicative unit 1. For every $n\in \NN$, let $n \in \sr$ denote the sum of $n$ copies of the element~1. Let $\leq$ be a partial order on $\sr$ that satisfies the following properties:
\begin{enumerate}[\upshape(i)]
    \item For every $a,b,c \in \sr$, if $a \leq b$, then $a + c \leq b + c$ and $ac \leq bc$.
    \item For every $n,m \in \NN$ we have $n \leq m$ in $\NN$ if and only if $n \leq m$ in $\sr$.
    \item For every nonzero element $a \in \sr$ there is an element $n \in \NN$ such that $1 \leq a \leq n$. %
\end{enumerate}

Let $\asympleq$ be the preorder defined on $\sr$ by $a \asympleq b$ if and only if $a^n \leq b^n 2^{o(n)}$. (In all the examples that we consider we will have that for every $a,b \in \sr$, $a^n \leq b^n 2^{o(n)}$ if and only if $a^n \leq b^{n + o(n)}$.)
We may without loss of generality assume that $\asympleq$ is a partial order, by identifying any two elements $a,b \in \sr$ for which $a \asympleq b$ and $b \asympleq a$. 
For every $a \in \sr$, let $\rank(a) = \min \{n \in \NN : a \leq n\}$ and $\subrank(a) = \max \{n \in \NN : n \leq a\}$. These functions are called rank and subrank, respectively.
Let $\asymprank(a) = \lim_{n\to\infty} \rank(a^n)^{1/n} = \inf_n \rank(a^n)^{1/n}$ and $\asympsubrank(a) = \lim_{n \to \infty} \subrank(a^n)^{1/n} = \sup_n \subrank(a^n)^{1/n}$. These functions are called asymptotic rank and asymptotic subrank, respectively.

\subsubsection*{Asymptotic spectrum duality}

Let $\X$ be the set of $\leq$-monotone semiring-homomorphisms $\sr \to \RR_{\geq0}$. We call $\X$ the asymptotic spectrum. For every $a \in \sr$, let $\widehat{a} : \X \to \RR_{\geq0} : F \mapsto F(a)$. We equip $\X$ with the coarsest topology such that for every $a \in \sr$ the map $\widehat{a}$ is continuous.

We call an element $a \in \sr$ \emph{gapped} if there is an element $k \in \NN$ such that $a^k \geq 2$ or there is an element $\phi \in \X$ such that $\phi(a) \leq 1$.

\begin{theorem}[{e.g.,~\cite[Section 3]{wigderson2022asymptotic}}]\label{th:abstract-duality}\hfill
\begin{enumerate}[\upshape(i)]
    \item $\X$ is a non-empty compact Hausdorff space.
    \item Let $a,b \in \sr$. Then $a \asympleq b$ if and only if for every $F \in \X$, $F(a) \leq F(b)$.
    \item Let $a \in \sr$. Then $\asymprank(a) = \max_{F \in \X} F(a)$.
    \item Let $a \in \sr$ be gapped. Then $\asympsubrank(a) = \min_{F \in \X} F(a)$.
\end{enumerate}
\end{theorem}

\begin{theorem}[{e.g.,~\cite[Section 4.3]{wigderson2022asymptotic}}]\label{th:abstract-surj}
    Let $(\sr_1, \leq_1)$ and $(\sr_2, \leq_2)$ be a pair of semirings and preorders as above. Let $i : \sr_1 \to \sr_2$ be a semiring homomorphism such that for every $a,b \in \sr_1$ we have $a \leq_1 b$ if and only if $i(a) \leq_2 i(b)$. Let $i^* : \X_2 \to \X_1 : F \to F \circ i$. Then $i^*$ is surjective.
\end{theorem}

\subsubsection*{Examples}
We discuss some examples that fit in the above framework. We refer to \cite{wigderson2022asymptotic} for a more detailed discussion of examples.

\emph{Finite graphs.} This is the setting of \autoref{sec:dist}. Let $\sr$ be the set of graphs with addition given by disjoint union, multiplication by strong product, the graph with no vertices as~0 and the graph with one vertex as $1$. Let $\leq$ be the cohomomorphism preorder on $\sr$ defined by $G \leq H$ if there is a cohomomorphism $G \to H$. We identify graphs that are equivalent under cohomomorphism. Then $\sr$ is a commutative semiring, and $\sr$ and $\leq$ satisfy the properties outlined above. %
In \autoref{sec:dist} we discussed known elements of $\X$.

\emph{Infinite graphs.} This is the setting of \autoref{sec:asymp-spec-infinite}. The above structure for finite graphs is readily extended to infinite graphs with finite clique covering number. Let $\sr$ be the set of infinite graphs with finite clique covering number. (Technically, for this to be a set we restrict ourselves to infinite graphs with bounded cardinality. For us, the cardinality of the reals will suffice, so we will tacitly assume our infinite graphs to have at most this cardinality.) We endow $\sr$ with the same operations and preorder as above, and identify graphs that are equivalent under cohomomorphism. Then~$\sr$ is a commutative semiring, and~$\sr$ and $\leq$ satisfy the properties outlined above. As discussed in \autoref{sec:asymp-spec-infinite}, \autoref{th:abstract-surj} can be used to prove that the restriction map from the asymptotic spectrum of finite graphs $\X$ to the asymptotic spectrum of infinite graphs~$\X_\infty$ is surjective. Thus every element of $\X$ can be extended to an element of~$\X_\infty$ (possibly in multiple ways).

\emph{Tensors.} The study of tensors (and in particular the matrix multiplication tensors) was the motivating application for Strassen to develop the theory of asymptotic spectra~\cite{strassen1988asymptotic}. While this setting is not the focus of this paper, we briefly discuss it here for the sake of comparison. Let $\sr$ be the set of tensors (over some fixed field, of order three, and with arbitrary dimensions) with addition given by direct sum, multiplication by tensor product, the $1\times 1\times 1$ tensor with coefficient 1 as 1 and the $1\times 1 \times 1$ tensor with coefficient 0 as 0. Let $\leq$ be the ``restriction'' preorder on $\sr$ defined by $S \leq T$ if there are linear maps $A_1, A_2, A_3$ such that $S = (A_1 \otimes A_2 \otimes A_3)T$. We identify tensors that are equivalent under restriction. Then $\sr$ is a commutative semiring, and $\sr$ and $\leq$ satisfy the properties outlined above. A known family of elements in the asymptotic spectrum are the quantum functionals \cite{MR4495838}. It is not known whether this is all of~$\X$, and this question is related to several hard questions in the field.

More examples can be found in \cite{MR4053385,MR4231964,chee2024asymptotic}.

\subsubsection*{Asymptotic spectrum distance}

Leading up to the introduction of the asymptotic spectrum distance, a first natural step is to rephrase \autoref{th:abstract-duality} in terms of the elements $\widehat{a}$ for $a \in \sr$. By construction of the topology on $\X$ each $\widehat{a}$ is a continuous map $\X \to \RR$.
Let $C(\X)$ denote the ring of all continuous functions $\X \to \RR$ under pointwise addition and multiplication, and for $\phi, \psi \in C(\X)$ write $\phi \leq \psi$ if $\phi$ is at most $\psi$ pointwise on~$\X$. Let $\max \phi$ denote the maximum of $\phi$ on $\X$, and let $\min \phi$ denote the minimum of $\phi$ on $\X$. The minimum and maximum are indeed attained since $\X$ is compact.
\autoref{th:abstract-duality} (ii), (iii) and (iv) can be rephrased as saying that 
    $a \asympleq b$ if and only if\, $\widehat{a} \leq \widehat{b}$; 
    $\asymprank(a) = \max\, \widehat{a}$; and
    $\asympsubrank(a) = \min\, \widehat{a}$.

We equip $C(\X)$ with the supremum-norm $||\phi|| = \sup_{x \in \X} |\phi(x)|$ for $\phi \in C(\X)$ and the corresponding distance function $d(\phi, \psi) = ||\phi - \psi||$ for $\phi, \psi \in C(\X)$. Then $C(\X)$ is complete \cite[Theorem~7.15]{MR0385023}.
The map $\Phi : \sr \to C(\X) : a \mapsto \widehat{a}$ induces a distance on $\sr$ as follows.

\begin{definition}
    For every $a,b \in \sr$, we define $d(a,b) \coloneqq d(\widehat{a}, \widehat{b})$. 
    We call $d$ the \emph{asymptotic spectrum distance}.
\end{definition}
It follows directly from the definitions that $d(a,b) = ||\widehat{a} - \widehat{b}|| = \sup_{x \in \X} |\widehat{a}(x) - \widehat{b}(x)| = \sup_{F \in \X} |F(a) - F(b)|.$ Identifying in~$\sr$ any two elements $a,b$ for which $a \asympleq b$ and $b \asympleq a$, the set $\sr$ with the asymptotic spectrum distance becomes a metric space.

From the asymptotic spectrum duality (\autoref{th:abstract-duality} (ii)), it follows directly that the asymptotic spectrum distance can be rephrased in terms of the cohomomorphism preorder as follows.

\begin{lemma}\label{lem:abstract-dist-char} Let $a,b \in \sr$ and $n,m \in \NN$, $m \geq1$. The following are equivalent.
\begin{enumerate}[\upshape(i)]
    \item $d(a,b) \leq n/m$
    \item For every $F \in \X$ we have $|F(a) - F(b)| \leq n/m$
    \item $m  a \asympleq m b + n$ and $m b \asympleq m a + n$
    \item $(ma)^{ k} \leq (mb + n)^{k} 2^{o(k)}$ and $(m b)^{k} \leq (ma + n)^{k} 2^{o(k)}$
\end{enumerate}
\end{lemma}
\begin{proof}
    (i) $\Leftrightarrow$ (ii) is by definition of asymptotic spectrum distance. 

    (ii) $\Rightarrow$ (iii). For every $F \in \X$ we have $mF(a) \leq mF(b) + n$ and $mF(b) \leq mF(a) + n$, and so $F(ma) \leq F(mb+ n)$ and $F(mb) \leq F(ma+n)$. Now apply \autoref{th:abstract-duality} (ii).

    (iii) $\Rightarrow$ (ii). Apply $F \in \X$ to both sides of the two inequalities.

    (iii) $\Leftrightarrow$ (iv) is by definition of the asymptotic preorder.
\end{proof}

\subsubsection*{Convergence}

We have the standard notions of Cauchy and convergence in $\sr$:
A sequence $a_1, a_2, \ldots$ in $\sr$ is Cauchy if for every $n/m>0$ $(n,m \in \NN)$ there is an $N \in \NN$ such that for every $i,j > N$, we have $d(a_i,a_j) \leq n/m$.
The sequence $a_1, a_2, \ldots$ in $\sr$ converges to $b \in \sr$ if for every $n/m>0$ $(n,m \in \NN)$ there is an $N \in \NN$ such that for every $i > N$, we have $d(a_i,b) \leq n/m$.

\begin{lemma}\label{lem:abstract-Theta-convergence}
    If the sequence $a_1, a_2, \ldots$ in $\sr$ converges to $b \in \sr$, then $\asympsubrank(a_i)$ converges to $\asympsubrank(b)$, and $\asymprank(a_i)$ converges to $\asymprank(b)$.
\end{lemma}
\begin{proof}
    This follows from the general fact that if a sequence of functions $f_1, f_2,\ldots$ on a compact space converges uniformly to $f$, then the sequence $\min f_1, \min f_2, \ldots$ converges to $\min f$, and the same for the maximum. See also the proof of \autoref{lem:Theta-conv} which directly generalizes to this setting.
\end{proof}

\begin{lemma}\label{th:abstract-dini}
    Let $a_1, a_2, \ldots$ be a sequence in $\sr$ and let $b \in \sr$. If for every $F \in \X$ we have that $F(a_1), F(a_2), \ldots$ is monotone (i.e.,~non-increasing for every $F\in \X$ or non-decreasing for every $F\in \X$) and converges to $F(b)$, then $a_1, a_2,\ldots$ converges to $b$.
\end{lemma}
\begin{proof}
        This follows from Dini's theorem \cite[Theorem~7.13]{MR0385023}, which says that for any sequence of continuous functions $f_1, f_2, \ldots$ on a compact space~$\X$, if the sequence $f_1, f_2,\ldots$ converges pointwise to a continuous function, and for every~$x \in \X$ we have that $i \mapsto f_i(x)$ is monotone in $i$, then the sequence $f_1, f_2, \ldots$ converges uniformly to~$f$. 
        Apply this theorem to the sequence of continuous functions $\widehat{a}_1, \widehat{a}_2, \ldots$ to get the claim. %
\end{proof}

\section{Explicit independent set}\label{sec:explicit-ind-set}

We give an independent set in $C_{15}^{\boxtimes 4}$ of size 2842. This was obtained as described in the proof of \autoref{th:C15-4}. %
We have encoded the vertices as 4-tuples of elements in $\{0,1,\ldots, 14\}$ where we use $A = 10, B = 11$, \ldots, $E = 14$.

\lstset{ %
  basicstyle=\tiny\ttfamily,        %
  breakatwhitespace=true,          %
  breakindent=0em,
  breaklines=true,                 %
  keepspaces=true,                 %
}
\tiny
\begin{lstlisting}[breaklines]
0018 001A 003B 004D 005B 0061 006E 0082 00A2 00A4 00C5 00E7 0116 0137 0139 0159 017A 017C 0180 019D 01A0 01C1 01C3 01E3 01E5 0218 021A 023B 024D 025B 0261 026E 0282 02A2 02A4 02C5 02C7 02E7 030C 0320 032D 0340 0342 0363 0384 0386 03A6 03A8 03C9 03DB 03E9 0401 040E 0422 0424 0444 0465 0467 0488 049A 04BB 04BD 04DD 0503 0505 0526 0546 0548 0569 057B 059C 059E 05B0 05D0 05D2 0607 061A 0628 063A 065B 065D 0670 067D 0691 06B2 06B4 06D4 06D6 071C 073C 073E 0750 0752 0772 0793 0795 07A8 07B6 07C8 07E9 07EB 0811 081E 0831 0833 0854 0866 0874 0887 0889 08AA 08CA 08CC 08ED 0905 0913 0925 0946 0948 0968 098B 09AC 09AE 09CE 09E0 09E2 0A07 0A27 0A29 0A4A 0A6A 0A6C 0A80 0A8D 0AA1 0AB3 0AC1 0AD4 0AD6 0B09 0B0B 0B2B 0B4C 0B4E 0B51 0B6E 0B72 0B93 0B95 0BB5 0BB7 0BD8 0C0D 0C10 0C2D 0C31 0C33 0C53 0C74 0C76 0C97 0C99 0CB9 0CDA 0CDC 0D12 0D14 0D35 0D55 0D57 0D78 0D9B 0DBB 0DBD 0DDE 0DE1 0E16 0E37 0E39 0E59 0E7A 0E7C 0E80 0E9D 0EA0 0EC1 0EC3 0EE3 0EE5 1001 100E 1022 1024 1044 1065 1067 1088 109A 10BB 10BD 10C7 10DD 10E9 110C 1120 112D 1140 1142 1163 1184 1186 11A6 11A8 11C9 11DB 1201 120E 1222 1224 1244 1265 1267 1288 129A 12BB 12BD 12DD 1303 1305 1326 1346 1348 1369 137B 139C 139E 13B0 13D0 13D2 1407 141A 1428 143A 145B 145D 1470 147D 1491 14B2 14B4 14D4 14D6 151C 153C 153E 1550 1552 1572 1593 1595 15A8 15B6 15C8 15E9 15EB 1611 161E 1631 1633 1654 1666 1674 1687 1689 16AA 16CA 16CC 16ED 1705 1713 1725 1746 1748 1768 178B 17AC 17AE 17CE 17E0 17E2 1807 1827 1829 184A 186A 186C 1880 188D 18A1 18B3 18C1 18D4 18D6 1909 190B 192B 194C 194E 1951 196E 1972 1993 1995 19B5 19B7 19D8 1A0D 1A10 1A2D 1A31 1A33 1A53 1A74 1A76 1A97 1A99 1AB9 1ADA 1ADC 1B12 1B14 1B35 1B55 1B57 1B78 1B9B 1BAD 1BBB 1BCE 1BE0 1BE2 1C16 1C37 1C39 1C59 1C6C 1C7A 1C80 1C8D 1CA0 1CC1 1CC3 1CE4 1CE6 1D0A 1D18 1D2B 1D4B 1D4D 1D61 1D6E 1D82 1DA2 1DA4 1DC5 1DC7 1DE8 1E0C 1E20 1E2D 1E40 1E42 1E63 1E84 1E86 1EA6 1EA8 1EC9 1EDB 2007 201A 2028 203A 205B 205D 2070 207D 2091 20B2 20B4 20D4 2103 2105 2126 2146 2148 2169 217B 219C 219E 21B0 21D0 21D2 2207 221A 2228 223A 225B 225D 2270 227D 2291 22B2 22B4 22D4 22D6 231C 233C 233E 2350 2352 2372 2393 2395 23A8 23B6 23C8 23E9 23EB 2411 241E 2431 2433 2454 2466 2474 2487 2489 24AA 24CA 24CC 24ED 2505 2513 2525 2546 2548 2568 258B 25AC 25AE 25CE 25E0 25E2 2607 2627 2629 264A 266A 266C 2680 268D 26A1 26B3 26C1 26D4 26D6 2709 270B 272B 274C 274E 2751 276E 2772 2793 2795 27B5 27B7 27D8 280D 2810 282D 2831 2833 2853 2874 2876 2897 2899 28B9 28DA 28DC 2912 2914 2935 2955 2957 2978 299B 29AD 29BB 29CE 29E0 29E2 2A16 2A37 2A39 2A59 2A6C 2A7A 2A80 2A8D 2AA0 2AC1 2AC3 2AE4 2AE6 2B0A 2B18 2B2B 2B4B 2B4D 2B61 2B6E 2B82 2BA2 2BA4 2BC5 2BC7 2BE8 2C0C 2C20 2C2D 2C40 2C42 2C63 2C84 2C86 2CA6 2CA8 2CC9 2CDB 2D01 2D0E 2D22 2D24 2D44 2D65 2D67 2D88 2D9A 2DBB 2DBD 2DDD 2E03 2E05 2E26 2E46 2E48 2E69 2E7B 2E9C 2E9E 2EB0 2ED0 2ED2 3011 301E 3031 3033 3054 3066 3074 3087 30AA 30CA 30CC 30D6 30ED 311C 313C 313E 3150 3152 3172 3193 3195 31A8 31B6 31C8 31E9 31EB 3211 321E 3231 3233 3254 3266 3274 3287 3289 32AA 32CA 32CC 32ED 3305 3313 3325 3346 3348 3368 338B 33AC 33AE 33CE 33E0 33E2 3407 3427 3429 344A 346A 346C 3480 348D 34A1 34B3 34C1 34D4 34D6 3509 350B 352B 354C 354E 3551 356E 3572 3593 3595 35B5 35B7 35D8 360D 3610 362D 3631 3633 3653 3674 3676 3697 3699 36B9 36DA 36DC 3712 3714 3735 3755 3757 3778 379B 37AD 37BB 37CE 37E0 37E2 3816 3837 3839 3859 386C 387A 3880 388D 38A0 38C1 38C3 38E4 38E6 390A 3918 392B 394B 394D 3961 396E 3982 39A2 39A4 39C5 39C7 39E8 3A0C 3A20 3A2D 3A40 3A42 3A63 3A84 3A86 3AA6 3AA8 3AC9 3ADB 3B01 3B0E 3B22 3B24 3B44 3B65 3B67 3B88 3B9A 3BBB 3BBD 3BDD 3C03 3C05 3C26 3C46 3C48 3C69 3C7B 3C9C 3C9E 3CB0 3CD0 3CD2 3D07 3D1A 3D28 3D3A 3D5B 3D5D 3D70 3D7D 3D91 3DB2 3DB4 3DD4 3DD6 3E1C 3E3C 3E3E 3E50 3E52 3E72 3E93 3E95 3EA8 3EB6 3EC8 3EE9 3EEB 4007 4027 4029 404A 406A 406C 4080 4089 408D 40A1 40B3 40C1 40D4 4105 4113 4125 4146 4148 4168 418B 41AC 41AE 41CE 41E0 41E2 4207 4227 4229 424A 426A 426C 4280 428D 42A1 42B3 42C1 42D4 42D6 4309 430B 432B 434C 434E 4351 436E 4372 4393 4395 43B5 43B7 43D8 440D 4410 442D 4431 4433 4453 4474 4476 4497 4499 44B9 44DA 44DC 4512 4514 4535 4555 4557 4578 459B 45AD 45BB 45CE 45E0 45E2 4616 4637 4639 4659 466C 467A 4680 468D 46A0 46C1 46C3 46E4 46E6 470A 4718 472B 474B 474D 4761 476E 4782 47A2 47A4 47C5 47C7 47E8 480C 4820 482D 4840 4842 4863 4884 4886 48A6 48A8 48C9 48DB 4901 490E 4922 4924 4944 4965 4967 4988 499A 49BB 49BD 49DD 4A03 4A05 4A26 4A46 4A48 4A69 4A7B 4A9C 4A9E 4AB0 4AD0 4AD2 4B07 4B1A 4B28 4B3A 4B5B 4B5D 4B70 4B7D 4B91 4BB2 4BB4 4BD4 4BD6 4C1C 4C3C 4C3E 4C50 4C52 4C72 4C93 4C95 4CA8 4CB6 4CC8 4CE9 4CEB 4D11 4D1E 4D31 4D33 4D54 4D66 4D74 4D87 4D89 4DAA 4DCA 4DCC 4DED 4E05 4E13 4E25 4E46 4E48 4E68 4E8B 4EAC 4EAE 4ECE 4EE0 4EE2 500D 5010 502D 5031 5033 5053 5074 5076 5097 50B9 50D6 50DA 50DC 5109 510B 512B 514C 514E 5151 516E 5172 5193 5195 51B5 51B7 51D8 520D 5210 522D 5231 5233 5253 5274 5276 5297 5299 52B9 52DA 52DC 5312 5314 5335 5355 5357 5378 539B 53AD 53BB 53CE 53E0 53E2 5416 5437 5439 5459 546C 547A 5480 548D 54A0 54C1 54C3 54E4 54E6 550A 5518 552B 554B 554D 5561 556E 5582 55A2 55A4 55C5 55C7 55E8 560C 5620 562D 5640 5642 5663 5684 5686 56A6 56A8 56C9 56DB 5701 570E 5722 5724 5744 5765 5767 5788 579A 57BB 57BD 57DD 5803 5805 5826 5846 5848 5869 587B 589C 589E 58B0 58D0 58D2 5907 591A 5928 593A 595B 595D 5970 597D 5991 59B2 59B4 59D4 59D6 5A1C 5A3C 5A3E 5A50 5A52 5A72 5A93 5A95 5AA8 5AB6 5AC8 5AE9 5AEB 5B11 5B1E 5B31 5B33 5B54 5B66 5B74 5B87 5B89 5BAA 5BCA 5BCC 5BED 5C05 5C13 5C25 5C46 5C48 5C68 5C8B 5CAC 5CAE 5CCE 5CE0 5CE2 5D07 5D27 5D29 5D4A 5D6A 5D6C 5D80 5D8D 5DA1 5DB3 5DC1 5DD4 5DD6 5E09 5E0B 5E2B 5E4C 5E4E 5E51 5E6E 5E72 5E93 5E95 5EB5 5EB7 5ED8 6006 6027 6029 604A 606A 606C 6080 608D 6099 60A1 60C1 60C3 60E4 6112 6114 6135 6155 6157 6178 619B 61AD 61BB 61CE 61E0 61E2 6216 6237 6239 6259 626C 627A 6280 628D 62A0 62C1 62C3 62E4 62E6 630A 6318 632B 634B 634D 6361 636E 6382 63A2 63A4 63C5 63C7 63E8 640C 6420 642D 6440 6442 6463 6484 6486 6498 64A6 64B9 64DA 64DC 6501 650E 6522 6524 6544 6557 6565 6578 659A 65BB 65BD 65DE 6603 6616 6636 6638 6659 667A 667C 669C 669E 66B0 66D1 66D3 66E5 6718 671A 673A 675B 675D 6771 677E 6791 67A4 67B2 67C5 67C7 67E7 681C 683C 683E 6842 6850 6863 6883 6885 68A6 68A8 68C9 68E9 68EB 6901 691E 6921 6923 6944 6965 6967 6987 6989 69AA 69CB 69CD 69ED 6A03 6A05 6A25 6A46 6A48 6A69 6A8B 6A9E 6AAC 6AB0 6AD0 6AD2 6B07 6B27 6B29 6B4A 6B5D 6B6B 6B70 6B7D 6B91 6BB2 6BB4 6BD4 6BD6 6C09 6C1B 6C3C 6C3E 6C50 6C52 6C72 6C93 6C95 6CB6 6CB8 6CD8 6CEB 6D10 6D1D 6D31 6D33 6D54 6D74 6D76 6D89 6D97 6DAA 6DCA 6DCC 6DED 6E12 6E14 6E35 6E48 6E56 6E68 6E8B 6EAC 6EAE 6ECE 6EE0 6EE2 700C 7020 702D 7033 7041 7053 7074 7076 7097 70B9 70D6 70DA 70DC 7108 710A 712B 714C 714E 7161 716E 7182 7195 71A3 71B5 71B7 71D8 720C 7220 722D 7233 7241 7253 7274 7276 7297 7299 72B9 72DA 72DC 730E 7312 7314 7335 7355 7357 7378 739B 73BB 73BD 73DE 73E1 7416 7437 7439 7459 747A 747C 7480 749D 74A0 74C1 74C3 74E3 74E5 7518 751A 753B 754D 755B 7561 756E 7582 75A2 75A4 75C5 75C7 75E7 760C 7620 762D 7640 7642 7663 7684 7686 76A6 76A8 76C9 76DB 76E9 7701 770E 7722 7724 7744 7765 7767 7788 779A 77BB 77BD 77DD 7803 7805 7826 7846 7848 7869 787B 789C 789E 78B0 78D0 78D2 7907 7928 793A 795B 795D 7970 797D 7991 79B2 79B4 79D4 79D6 7A09 7A1B 7A3C 7A3E 7A50 7A52 7A72 7A93 7A95 7AB6 7AB8 7AD8 7AEB 7B10 7B1D 7B31 7B33 7B54 7B74 7B76 7B89 7B97 7BAA 7BCA 7BCC 7BED 7C12 7C14 7C35 7C48 7C56 7C68 7C8B 7CAC 7CAE 7CCE 7CE0 7CE2 7D06 7D27 7D29 7D4A 7D6A 7D6C 7D80 7D8D 7DA1 7DC1 7DC3 7DD6 7DE4 7E08 7E0A 7E2B 7E4C 7E4E 7E61 7E6E 7E82 7E95 7EA3 7EB5 7EB7 7ED8 8016 8037 8039 8059 807A 807C 8080 8099 809D 80A0 80C1 80C3 80E3 810E 8112 8114 8135 8155 8157 8178 819B 81BB 81BD 81DE 81E1 8216 8237 8239 8259 827A 827C 8280 829D 82A0 82C1 82C3 82E3 82E5 8318 831A 833B 834D 835B 8361 836E 8382 83A2 83A4 83C5 83C7 83E7 840C 8420 842D 8440 8442 8463 8484 8486 84A6 84A8 84C9 84DB 84E9 8501 850E 8522 8524 8544 8565 8567 8588 859A 85BB 85BD 85DD 8603 8605 8626 8646 8648 8669 867B 869C 869E 86B0 86D0 86D2 8707 8728 873A 875B 875D 8770 877D 8791 87B2 87B4 87D4 87D6 8809 881B 883C 883E 8850 8852 8872 8893 8895 88B6 88B8 88D8 88EB 8910 891D 8931 8933 8954 8974 8976 8989 8997 89AA 89CA 89CC 89ED 8A12 8A14 8A35 8A48 8A56 8A68 8A8B 8AAC 8AAE 8ACE 8AE0 8AE2 8B06 8B27 8B29 8B4A 8B6A 8B6C 8B80 8B8D 8BA1 8BC1 8BC3 8BD6 8BE4 8C08 8C0A 8C2B 8C4C 8C4E 8C61 8C6E 8C82 8C95 8CA3 8CB5 8CB7 8CD8 8D0C 8D20 8D2D 8D33 8D41 8D53 8D74 8D76 8D97 8D99 8DB9 8DDA 8DDC 8E0E 8E12 8E14 8E35 8E55 8E57 8E78 8E9B 8EBB 8EBD 8EDE 8EE1 900C 9020 902D 9040 9042 9063 9084 9086 90A6 90C9 90DB 90E5 90E9 9118 911A 913B 914D 915B 9161 916E 9182 91A2 91A4 91C5 91C7 91E7 920C 9220 922D 9240 9242 9263 9284 9286 92A6 92A8 92C9 92DB 92E9 9301 930E 9322 9324 9344 9365 9367 9388 939A 93BB 93BD 93DD 9403 9405 9426 9446 9448 9469 947B 949C 949E 94B0 94D0 94D2 9507 9528 953A 955B 955D 9570 957D 9591 95B2 95B4 95D4 95D6 9609 961B 963C 963E 9650 9652 9672 9693 9695 96B6 96B8 96D8 96EB 9710 971D 9731 9733 9754 9774 9776 9789 9797 97AA 97CA 97CC 97ED 9812 9814 9835 9848 9856 9868 988B 98AC 98AE 98CE 98E0 98E2 9906 9927 9929 994A 996A 996C 9980 998D 99A1 99C1 99C3 99D6 99E4 9A08 9A0A 9A2B 9A4C 9A4E 9A61 9A6E 9A82 9A95 9AA3 9AB5 9AB7 9AD8 9B0C 9B20 9B2D 9B33 9B41 9B53 9B74 9B76 9B97 9B99 9BB9 9BDA 9BDC 9C0E 9C12 9C14 9C35 9C55 9C57 9C78 9C9B 9CBB 9CBD 9CDE 9CE1 9D16 9D37 9D39 9D59 9D7A 9D7C 9D80 9D9D 9DA0 9DC1 9DC3 9DE3 9DE5 9E18 9E1A 9E3B 9E4D 9E5B 9E61 9E6E 9E82 9EA2 9EA4 9EC5 9EC7 9EE7 A003 A026 A046 A048 A069 A07B A09C A09E A0A8 A0B0 A0D0 A0D2 A101 A10E A122 A124 A144 A165 A167 A188 A19A A1BB A1BD A1DD A203 A205 A226 A246 A248 A269 A27B A29C A29E A2B0 A2D0 A2D2 A307 A328 A33A A35B A35D A370 A37D A391 A3B2 A3B4 A3D4 A3D6 A409 A41B A43C A43E A450 A452 A472 A493 A495 A4B6 A4B8 A4D8 A4EB A510 A51D A531 A533 A554 A574 A576 A589 A597 A5AA A5CA A5CC A5ED A612 A614 A635 A648 A656 A668 A68B A6AC A6AE A6CE A6E0 A6E2 A706 A727 A729 A74A A76A A76C A780 A78D A7A1 A7C1 A7C3 A7D6 A7E4 A808 A80A A82B A84C A84E A861 A86E A882 A895 A8A3 A8B5 A8B7 A8D8 A90C A920 A92D A933 A941 A953 A974 A976 A997 A999 A9B9 A9DA A9DC AA0E AA12 AA14 AA35 AA55 AA57 AA78 AA9B AABB AABD AADE AAE1 AB16 AB37 AB39 AB59 AB7A AB7C AB80 AB9D ABA0 ABC1 ABC3 ABE3 ABE5 AC18 AC1A AC3B AC4D AC5B AC61 AC6E AC82 ACA2 ACA4 ACC5 ACC7 ACE7 AD0C AD20 AD2D AD40 AD42 AD63 AD84 AD86 ADA6 ADA8 ADC9 ADDB ADE9 AE01 AE0E AE22 AE24 AE44 AE65 AE67 AE88 AE9A AEBB AEBD AEDD B005 B01C B03C B03E B050 B052 B072 B093 B095 B0B6 B0C8 B0E9 B0EB B107 B128 B13A B15B B15D B170 B17D B191 B1B2 B1B4 B1D4 B1D6 B209 B21B B23C B23E B250 B252 B272 B293 B295 B2B6 B2B8 B2D8 B2EB B310 B31D B331 B333 B354 B374 B376 B389 B397 B3AA B3CA B3CC B3ED B412 B414 B435 B448 B456 B468 B48B B4AC B4AE B4CE B4E0 B4E2 B506 B527 B529 B54A B56A B56C B580 B58D B5A1 B5C1 B5C3 B5D6 B5E4 B608 B60A B62B B64C B64E B661 B66E B682 B695 B6A3 B6B5 B6B7 B6D8 B70C B720 B72D B733 B741 B753 B774 B776 B797 B799 B7B9 B7DA B7DC B80E B812 B814 B835 B855 B857 B878 B89B B8BB B8BD B8DE B8E1 B916 B937 B939 B959 B97A B97C B980 B99D B9A0 B9C1 B9C3 B9E3 B9E5 BA18 BA1A BA3B BA4D BA5B BA61 BA6E BA82 BAA2 BAA4 BAC5 BAC7 BAE7 BB0C BB20 BB2D BB40 BB42 BB63 BB84 BB86 BBA6 BBA8 BBC9 BBDB BBE9 BC01 BC0E BC22 BC24 BC44 BC65 BC67 BC88 BC9A BCBB BCBD BCDD BD03 BD05 BD26 BD46 BD48 BD69 BD7B BD9C BD9E BDB0 BDD0 BDD2 BE07 BE1A BE28 BE3A BE5B BE5D BE70 BE7D BE91 BEB2 BEB4 BED4 BED6 C013 C025 C046 C048 C068 C08B C0A8 C0AC C0AE C0CE C0E0 C0E2 C111 C11E C131 C133 C154 C166 C174 C187 C189 C1AA C1CA C1CC C1ED C213 C225 C246 C248 C268 C28B C2AC C2AE C2CE C2E0 C2E2 C306 C327 C329 C34A C36A C36C C380 C38D C3A1 C3C1 C3C3 C3D6 C3E4 C408 C40A C42B C44C C44E C461 C46E C482 C495 C4A3 C4B5 C4B7 C4D8 C50C C520 C52D C533 C541 C553 C574 C576 C597 C599 C5B9 C5DA C5DC C60E C612 C614 C635 C655 C657 C678 C69B C6BB C6BD C6DE C6E1 C716 C737 C739 C759 C77A C77C C780 C79D C7A0 C7C1 C7C3 C7E3 C7E5 C818 C81A C83B C84D C85B C861 C86E C882 C8A2 C8A4 C8C5 C8C7 C8E7 C90C C920 C92D C940 C942 C963 C984 C986 C9A6 C9A8 C9C9 C9DB C9E9 CA01 CA0E CA22 CA24 CA44 CA65 CA67 CA88 CA9A CABB CABD CADD CB03 CB05 CB26 CB46 CB48 CB69 CB7B CB9C CB9E CBB0 CBD0 CBD2 CC07 CC1A CC28 CC3A CC5B CC5D CC70 CC7D CC91 CCB2 CCB4 CCD4 CCD6 CD1C CD3C CD3E CD50 CD52 CD72 CD93 CD95 CDA8 CDB6 CDC8 CDE9 CDEB CE11 CE1E CE31 CE33 CE54 CE66 CE74 CE87 CE89 CEAA CECA CECC CEED D009 D00B D02B D04C D04E D051 D06E D072 D093 D095 D0B5 D0D8 D106 D127 D129 D14A D16A D16C D180 D18D D1A1 D1C1 D1C3 D1D6 D1E4 D208 D20A D22B D24C D24E D261 D26E D282 D295 D2A3 D2B5 D2B7 D2D8 D30C D320 D32D D333 D341 D353 D374 D376 D397 D399 D3B9 D3DA D3DC D40E D412 D414 D435 D455 D457 D478 D49B D4BB D4BD D4DE D4E1 D516 D537 D539 D559 D57A D57C D580 D59D D5A0 D5C1 D5C3 D5E3 D5E5 D618 D61A D63B D64D D65B D661 D66E D682 D6A2 D6A4 D6C5 D6C7 D6E7 D70C D720 D72D D740 D742 D763 D784 D786 D7A6 D7A8 D7C9 D7DB D7E9 D801 D80E D822 D824 D844 D865 D867 D888 D89A D8BB D8BD D8DD D903 D905 D926 D946 D948 D969 D97B D99C D99E D9B0 D9D0 D9D2 DA07 DA1A DA28 DA3A DA5B DA5D DA70 DA7D DA91 DAB2 DAB4 DAD4 DAD6 DB1C DB3C DB3E DB50 DB52 DB72 DB93 DB95 DBA8 DBB6 DBC8 DBE9 DBEB DC11 DC1E DC31 DC33 DC54 DC66 DC74 DC87 DC89 DCAA DCCA DCCC DCED DD05 DD13 DD25 DD46 DD48 DD68 DD8B DDAC DDAE DDCE DDE0 DDE2 DE07 DE27 DE29 DE4A DE6A DE6C DE80 DE8D DEA1 DEB3 DEC1 DED4 DED6 E012 E014 E035 E055 E057 E078 E09B E0B7 E0BB E0BD E0DE E0E1 E10D E110 E12D E131 E133 E153 E174 E176 E197 E199 E1B9 E1DA E1DC E212 E214 E235 E255 E257 E278 E29B E2BB E2BD E2DE E2E1 E316 E337 E339 E359 E37A E37C E380 E39D E3A0 E3C1 E3C3 E3E3 E3E5 E418 E41A E43B E44D E45B E461 E46E E482 E4A2 E4A4 E4C5 E4C7 E4E7 E50C E520 E52D E540 E542 E563 E584 E586 E5A6 E5A8 E5C9 E5DB E5E9 E601 E60E E622 E624 E644 E665 E667 E688 E69A E6BB E6BD E6DD E703 E705 E726 E746 E748 E769 E77B E79C E79E E7B0 E7D0 E7D2 E807 E81A E828 E83A E85B E85D E870 E87D E891 E8B2 E8B4 E8D4 E8D6 E91C E93C E93E E950 E952 E972 E993 E995 E9A8 E9B6 E9C8 E9E9 E9EB EA11 EA1E EA31 EA33 EA54 EA66 EA74 EA87 EA89 EAAA EACA EACC EAED EB05 EB13 EB25 EB46 EB48 EB68 EB8B EBAC EBAE EBCE EBE0 EBE2 EC07 EC27 EC29 EC4A EC6A EC6C EC80 EC8D ECA1 ECB3 ECC1 ECD4 ECD6 ED09 ED0B ED2B ED4C ED4E ED51 ED6E ED72 ED93 ED95 EDB5 EDB7 EDD8 EE0D EE10 EE2D EE31 EE33 EE53 EE74 EE76 EE97 EE99 EEB9 EEDA EEDC
\end{lstlisting}

\newpage
\normalsize
\bibliographystyle{alphaurl}
\bibliography{main.bib}
\addcontentsline{toc}{section}{References}

\vspace{2em}
\noindent\raggedright
\textbf{David de Boer,\, Pjotr Buys}\\[0.5em]
Korteweg-de Vries Institute for Mathematics, University of Amsterdam\\
Science Park 105-107, 1098 XG Amsterdam, Netherlands\\[0.5em]

\textbf{Jeroen Zuiddam}\\[0.5em]
Korteweg-de Vries Institute for Mathematics, University of Amsterdam\\
Science Park 105-107, 1098 XG Amsterdam, Netherlands\\[0.5em]
Centrum Wiskunde \& Informatica\\
Science Park 123, 1098 XG Amsterdam\\[1em]
Email: \href{daviddeboer2795@gmail.com}{daviddeboer2795@gmail.com}, \href{pjotr.buys@gmail.com}{pjotr.buys@gmail.com}, \href{j.zuiddam@uva.nl}{j.zuiddam@uva.nl}

\end{document}